\documentclass[10pt]{amsart}
\usepackage{amsmath,amssymb}
\usepackage{color}
\usepackage{graphicx}
\usepackage{amsthm}
\usepackage[all]{xy}
\usepackage{verbatim}
\usepackage{enumerate}
\usepackage{caption}
\usepackage{pgf,tikz}
\usetikzlibrary{arrows}
\usepackage{longtable}
\usepackage{hyperref} 

\def\oC{\overline{\mathcal{C}}}
\def\oM{\overline{\mathcal{M}}}
\def\M{{\mathcal{M}}}

\def\oH{\overline{\mathcal{H}}}
\def\ofH{\overline{\mathfrak{H}}}
\def\H{{\mathcal{H}}}
\def\L{{\mathcal{L}}}
\def\oR{\mathfrak{R}}
\def\oP{\mathcal{P}}
\def\oZ{\mathcal{Z}}
\def\oN{\mathcal{N}}

\def\R{\mathbb{R}}
\def\C{\mathbb{C}}
\def\Z{\mathbb{Z}}
\def\N{\mathbb{N}}
\def\Q{\mathbb{Q}}
\def\P{\mathbb{P}}
\def\O{\mathcal{O}}

\def\bg{\mathbf{g}}
\def\bn{\mathbf{n}}
\def\bm{\mathbf{m}}
\def\bZ{\mathbf{Z}}
\def\bP{\mathbf{P}}

\def\D{{\rm Div}}
\def\bic{{\rm Bic}}
\def\pic{{\rm Pic}}
\def\Hom{{\rm Hom}}

\begin{document}

\theoremstyle{definition}
\newtheorem{mydef}{Definition}[section]
\newtheorem{mynot}[mydef]{Notation}
\newtheorem{example}[mydef]{Example}
\newtheorem{remark}[mydef]{Remark}
\newtheorem{assumption}[mydef]{Assumption}
\theoremstyle{plain}
\newtheorem{myconj}{Conjecture}
\renewcommand*{\themyconj}{\Alph{myconj}}
\newtheorem{myth}{Theorem}
\newtheorem{mypr}[mydef]{Proposition}
\newtheorem{prdef}[mydef]{Proposition-Definition}
\newtheorem{mylem}[mydef]{Lemma}
\newtheorem{mycor}[mydef]{Corollary}

\title{Cohomology classes of strata of differentials}
\author{Adrien Sauvaget}
\address{Universit\'e Pierre et Marie Curie\\ 4 place Jussieu\\ 75005 Paris, France}
\email{adrien.sauvaget@imj-prg.fr}
\date{\today}
\maketitle

\begin{abstract}
We introduce a space of stable meromorphic differentials with poles of prescribed orders and define its tautological cohomology ring. This space, just as the space of holomorphic differentials, is stratified according to the set of multiplicities of zeros of the differential. The main goal of this paper is to compute the Poincar\'e-dual cohomology classes of all strata. We prove that all these classes are tautological and give an algorithm to compute them.

In a second part of the paper we study the Picard group of the strata. We use the tools introduced in the first part to deduce several relations in these Picard groups.
\end{abstract}

\setcounter{tocdepth}{1}
\tableofcontents

\section{Introduction}

\subsection{Stratification of the Hodge Bundle}

Let $g\geq 1$. Let $\mathcal{M}_g$ be the space of smooth curves of genus $g$. The \textit{Hodge bundle},
\begin{equation*}
\mathcal{H}_g \to \mathcal{M}_g
\end{equation*}
is the vector bundle whose fiber over a point $[C]$ of $\mathcal{M}_g$ is the space of holomorphic differentials on $C$. A point of $\mathcal{H}_g$ is then a pair $([C], \alpha)$, where $C$ is a curve and $\alpha$ a differential on $C$. We will denote by $\mathbb{P}\mathcal{H}_g \to \mathcal{M}_g$ the projectivization of the Hodge bundle.

\begin{mynot} \label{Not:Z}
Let $Z$ (for zeros) be a vector $(k_1,k_2,\ldots,k_n)$ of positive integers satisfying
\begin{equation*}
\sum_{i=1}^{n} k_i =2g-2.
\end{equation*}
We will denote by $\mathbb{P}\mathcal{H}_{g}(Z)$ the subspace of $\mathbb{P}\mathcal{H}_{g}$ composed of pairs $([C],\alpha)$ such that $\alpha$ is a differential (defined up to a multiplicative constant) with zeros of orders $k_1,\ldots,k_n$. 
\end{mynot}

The locus $\mathbb{P}\mathcal{H}_{g}(Z)$ is a smooth orbifold (or a Deligne-Mumford stack), see for instance,~\cite{Pol}. However, neither $\mathbb{P}\mathcal{H}_{g}$, nor the strata $\mathbb{P}\mathcal{H}_{g}(Z)$ are compact.

The Hodge bundle has a natural extension to the space of stable curves:
\begin{equation*}
\overline{\mathcal{H}}_g \to \overline{\mathcal{M}}_g.
\end{equation*}
We recall that abelian differentials over a nodal curve are allowed to have simple poles at the nodes with opposite residues on the two branches.

The space $\mathbb{P}\overline{\mathcal{H}}_g$ is compact and smooth, and we can consider the closures $\mathbb{P}\oH_g(Z)$ of the strata inside this space. Computing the Poincar\'e-dual cohomology classes of these strata is our motivating problem. In this paper we solve this problem and 
present a more general computation in the case of meromorphic differentials.

\subsection{Stable differentials} 

On a fixed smooth curve~$C$ with one marked point~$x$ consider a family of meromorphic differentials with one pole of order~$p$ at~$x$, such that the leading coefficient of the differential at the pole tends to~0. In order to construct a compact moduli space of meromorphic differentials we need to decide what the limit of a family like that should be. One natural idea is to include differentials with poles of orders less than~$p$ in the moduli space. It turns out, however, that a more convenient way to represent the limit is to allow the underlying curve to bubble at~$x$; in other words, to allow differentials defined on semi-stable curves. 

The first uses of semi-stable objects to compactify moduli problems can be found in the work of Gieseker for the moduli space of stable bundles (see~\cite{Gie}), or in Caporaso's construction of a universal Picard variety over the moduli space of curves (see~\cite{Cap}).

A {\em semi-stable curve} is a nodal curve with smooth marked points such that every genus~0 component of its normalization contains at least two marked points and preimages of nodes (instead of at least three for stable curves). In the example above, the limit of the family would be a meromorphic differential defined on a semi-stable curve with one unstable component and on marked point~$x$ on it. The curve maps to $C$ under the contraction of the unstable component. The meromorphic differential still has a pole of order exactly $p$ at~$x$.

\begin{mydef}\label{def:stablestack}
Let $n,m \in \mathbb{N}$ and let $P$ (for poles) be a vector $(p_1,p_2,\ldots,p_m)$ of positive integers. 
A {\em stable differential} of type $(g,n,P)$ is a tuple $(C,x_1,\dots,x_{n+m},\alpha)$ where $(C,x_1, \dots, x_{n+m})$ is a semi-stable curve with $n+m$ marked points and $\alpha$ is a meromorphic differential on $C$, such that 
\begin{itemize}
\item the differential $\alpha$ has no poles outside the $m$ last marked points and nodes;  
\item the poles at the nodes are at most simple and have opposite residues on the two branches;
\item if $p_i>1$ then the pole at the marked point $x_{n+i}$ is of order exactly $p_i$; if $p_i=1$ then $x_i$ can be a simple pole, a regular point, or a zero of any order;
\item the group of isomorphisms of $C$ preserving $\alpha$ and the marked points is finite.
\end{itemize}
\end{mydef}

\begin{mydef}\label{def:stable} A {\em family of stable differentials} is a tuple $(C\to B, \sigma_1,\ldots, \sigma_n, \alpha)$ where $(C\to B, \sigma_1,\ldots, \sigma_n) $ is a family of marked semi-stable curves and $\alpha$ is a meromorphic section of the relative dualizing line bundle $\omega_{C/B}$ such that for each geometric point $b$ of  $B$, the tuple $(C_b, \sigma_1(b),\ldots, \sigma_n(b), \alpha|_{C_b})$ is a stable differential.

The {\em stack $\overline{\mathfrak{H}}_{g,n,P}$ of stable differentials} of type $(g,n,P)$ is the category of families of stable differentials of type $(g,n,P)$, fibered over the category of $\C$-schemes.
\end{mydef}

\begin{mypr}\label{cone}
The moduli space $\overline{\mathfrak{H}}_{g,n,P}$ is a smooth Deligne-Mumford (DM) stack. It is of dimension $4g-4+\sum p_i$ if $P$ is non-empty and $4g-3$ otherwise. 
\end{mypr}

The space $\overline{\mathfrak{H}}_{g,n,P}$ carries a natural $\C^*$-action given by the multiplication of the differential by non-zero scalars. Besides, there exists a forgetful map $\overline{\mathfrak{H}}_{g,n,P}\to \oM_{g,n+m}$ that maps a family stable differentials to the stabilization of its underlying family of semi-stable curves.  However, the space  $\overline{\mathfrak{H}}_{g,n,P}$  does not have a natural vector bundle structure in general because there is no natural definition of the sum of two differentials with fixed orders of poles. 

We will construct a partial coarsification of $\overline{\mathfrak{H}}_{g,n,P}$ that has the structure of an orbifold cone over $\oM_{g,n+m}$.
\begin{mypr}\label{pr:stable}
There exists a unique DM stack $\oH_{g,n,P}$ fitting in the following commutative triangle
$$
\xymatrix{
\ofH_{g,n,P}\ar[r]\ar[rd] & \oH_{g,n,P}\ar[d]^{\pi}\\
&\oM_{g,n+m}.
}
$$ 
and satisfying
\begin{itemize}
\item the morphism $\pi$ is schematic, i.e. for any $\C$-scheme $U$ with a morphism $U\to \oM_{g,n+m}$, the pull-back $\oH_{g,n,P}\!\!\!\underset{\oM_{g,n+m}}{\times}\!\!\! U$ is representable by a $\C$-scheme;
\item for any such $U\to \oM_{g,n+m}$, the scheme $\oH_{g,n,P}\!\!\!\underset{\oM_{g,n+m}}{\times}\!\!\! U$ is the coarse space of $\ofH_{g,n,P}\!\!\!\underset{\oM_{g,n+m}}{\times}\!\!\! U$.
\end{itemize}
\end{mypr}
\begin{mydef}
The space $\oH_{g,n,P}$ is the called the {\em space of stable differentials}.
\end{mydef}

\begin{mypr}
The space of stable differentials is an orbifold cone over  $\oM_{g,n+m}$. Besides the space $\oH_{g,n,P}$ and its projectivization are normal.
\end{mypr}
We prove these propositions in Section~\ref{sec:stdiff}, where we will also give a definition of an orbifold cone. At present it suffices to note that the cone structure on $\overline{\mathcal{H}}_{g,n,P}$ allows one to define a projectivization $\P\overline{\mathcal{H}}_{g,n,P}$, a line bundle $\O(1)$ over the projectivization, and the Segre class. Besides, the morphism $\ofH_{g,n,P}\to \oH_{g,n,P}$ is $\C^*$-equivariant. 

\begin{remark}
The stack $\overline{\mathfrak{H}}_{g,n,P}$ can be endowed with the structure of an orbifold cone over a different moduli space $\oM_{g,n,P}$. The space $\oM_{g,n,P}$ is a $\left( \prod\limits_{i=1}^m \Z/(p_i-1) \Z\right)$-gerb over $\oM_{g,n+m}$. The fibers of $\ofH_{g,n,P} \to \oM_{g,n,P}$ are vector spaces, but the $\C^*$-action on these spaces has nontrivial weights.

One can define the projectivization of $\ofH_{g,n,P}$ and the tautological line bundle over this projectivization. Then we have a map $\P\overline{\mathfrak{H}}_{g,n,P}\to \P\overline{\mathcal{H}}_{g,n,P}$ which is a bijection between the geometric points of these two stacks. Therefore we have natural isomorphisms $H^*(\P\ofH_{g,n,P},\Q)\simeq H^*(\P\overline{\mathcal{H}}_{g,n,P},\Q)$ and $A^*(\P\ofH_{g,n,P},\Q)\simeq A^*(\P\overline{\mathcal{H}}_{g,n,P},\Q)$.  Thus, all the results of this text are valid for both spaces. 

While the space $\ofH_{g,n,P}$ is the more natural choice for the moduli space of differentials, in this paper we prefer to work with $\oH_{g,n,P}$ in order to have $\oM_{g,n+m}$ as the base of our cone.
\end{remark}

\begin{mynot} \label{Not:AZP}
Let $P = (p_1, \dots, p_m)$ be a vector of positive integers and $Z = (k_1, \dots, k_n)$ a vector of nonnegative integers. We denote by $A_{g,Z,P} \subset \oH_{g,n,P}$, the locus of stable differentials $(C,x_1,\dots,x_{n+m},\alpha)$ such that $C$ is smooth and $\alpha$ has zeros exactly of orders prescribed by $Z$ at the first $n$ marked points. The locus $A_{g,Z,P}$ is invariant under the $\C^*$-action. We denote by $\P A_{g,Z,P}$ the projectivization of $A_{g,Z,P}$. Moreover, we denote by $\overline{A}_{g,Z,P}$ (respectively $\P\overline{A}_{g,Z,P}$) the closures of $A_{g,Z,P}$ (resp. $\P A_{g,Z,P}$) in the space $\oH_{g,n,P}$ (respectively in $\P\oH_{g,n,P}$).
\end{mynot}

\subsection{The tautological ring of $\overline{\mathcal{M}}_{g,n}$} 

Let $g$ and $n$ be nonnegative integers satisfying $2g-2+n> 0$. Let  $\overline{\mathcal{M}}_{g,n}$
be the space of stable curves of genus $g$ with $n$ marked points. Define the following cohomology classes:

\begin{itemize}
\item $\psi_i= c_1(\mathcal{L}_i) \in H^{2}(\overline{\mathcal{M}}_{g,n},\mathbb{Q})$, where $\mathcal{L}_i$ is the cotangent line bundle at the $i^{\rm{th}}$ marked point,
\item $\kappa_m=\pi_*({\psi}_{n+1}^{m+1}) \in H^{2m}(\overline{\mathcal{M}}_{g,n},\mathbb{Q})$, where $\pi: \overline{\mathcal{M}}_{g,n+1} \to \overline{\mathcal{M}}_{g,n}$ is the forgetful map,
\item $\lambda_k=c_k(\oH_{g,n}) \in H^{2k}(\overline{\mathcal{M}}_{g,n},\mathbb{Q})$, for $k=1, \ldots, g$.
\end{itemize}

\begin{mydef}\label{def:stgraph} A {\em stable graph} is the datum of
\begin{equation*}
\Gamma = (V, H, g : V \to \mathbb{N}, a : H \to V, i : H \to H, E, L)
\end{equation*}
satisfying the following properties:
\begin{itemize}
\item $V$ is a vertex set with a genus function $g$;
\item $H$ is a half-edge set equipped with a vertex assignment $a$ and an involution~$i$;
\item$E$, the edge set, is defined as the set of length 2 orbits of $i$ in $H$ (self-edges at vertices are
permitted);
\item $(V, E)$ define a connected graph;
\item $L$ is the set of fixed points of $i$ called {\em legs};
\item for each vertex $v$, the stability condition holds:
$2g(v) - 2 + n(v) > 0$,
where $n(v)=\#(a^{-1}(v))$ (the cardinal of $a^{-1}(v)$.
\end{itemize}

The genus of $\Gamma$ is defined by $\sum g(v) + \#(E) - \#(V) + 1$.
\end{mydef}

 Let $v(\Gamma)$, $e(\Gamma)$, and $n(\Gamma)$ denote the cardinalities of $V, E$, and $L$, respectively. A boundary stratum of the moduli space of curves naturally determines a stable graph of genus $g$ with $n$ legs by considering the dual graph of a generic pointed curve parameterized by the stratum. Thus the boundary strata of $\oM_{g,n}$ are in 1-to-1 correspondence with stable graphs.
 
 Let $\Gamma$  be a stable graph. Define the moduli space $\oM_\Gamma$ by the product
\begin{equation*}
\oM_\Gamma =\prod_{v \in V} \oM_{g(v),n(v)},
\end{equation*}
and let $\zeta_\Gamma:\oM_\Gamma \to  \oM_{g,n}$ be the natural morphism.

\begin{mydef} A \textit{tautological class} is a linear combination of classes $\beta$ of the form
\begin{equation*}
\beta={\zeta_\Gamma}_*( \prod_{v \in V} P_v),
\end{equation*}
where $\Gamma$ is a stable graph and $P_v$ is a polynomial in $\kappa$, $\lambda$ and $\psi$ classes on $\oM_{g(v),n(v)}$. 
\end{mydef}


\begin{prdef}
Let $RH^*(\oM_{g,n})\subset H^*(\oM_{g,n},\Q)$ the vector subspace spanned by tautological classes.  This subspace is a subring called the {\em tautological ring} of $\oM_{g,n}$.
\end{prdef}
See~\cite{GraPan} for the description of the product of tautological classes.

\begin{remark} Actually the classes $\alpha$ as above that do not involve $\lambda$-classes span the tautological ring. However it will be more convenient for us to use this larger set of generators.
\end{remark}

\subsection{The tautological ring of ${\P\oH}_{g,n,P}$}

Let $P$ be a vector of positive integers. From now on, unless specified otherwise, we will denote by $\pi:\oM_{g,n+1}\to \oM_{g,n}$ the forgetful map and by $p:\overline{\mathcal{H}}_{g,n,P} \to \overline{\mathcal{M}}_{g,n+m}$ the projection from the space of stable differentials to $\oM_{g,n}$. Moreover we use the same notation $p:\P\oH_{g,n,P} \to \overline{\mathcal{M}}_{g,n+m}$ for the projectivized cone. Let 
\begin{equation*}
\mathcal{L}=\mathcal{O}(1) \to \mathbb{P}\overline{\mathcal{H}}_{g,n,P}
\end{equation*}
be the tautological line bundle of $\mathbb{P}\overline{\mathcal{H}}_{g,n,P}$, and let $\xi= c_1(\mathcal{L})$.
\begin{mydef} The {\em tautological ring of} $\mathbb{P}\overline{\mathcal{H}}_{g,n,P}$ is the subring of the cohomology ring $H^*(\mathbb{P}\overline{\mathcal{H}}_{g,n,P},\mathbb{Q})$ generated by~$\xi$ and the pull-back of $RH^*(\overline{\mathcal{M}}_{g,n+m})$ under $p$. We denote it by $RH^*(\P\oH_{g,n,P})$.
\end{mydef} 

\begin{remark} We have $\xi^d=0$ for $d>{\rm{dim}}(\mathbb{P}\overline{\mathcal{H}}_{g,n,P})$. Therefore the tautological ring of $\mathbb{P}\overline{\mathcal{H}}_{g,n,P}$ is a finite extension of the tautological ring of $\oM_{g,n+m}$.
\end{remark}

\begin{example}
In absence of poles, the Hodge bundle is a vector bundle and we have
\begin{equation*}
RH^*(\mathbb{P}\overline{\mathcal{H}}_{g,n})=RH^*(\overline{\mathcal{M}}_{g,n})[\xi]/(\xi^g+\lambda_1 \xi^{g-1}+\ldots+\lambda_g).
\end{equation*}
\end{example}

\begin{mypr}
The Segre class of the cone $\overline{\mathcal{H}}_{g,n,P} \to \oM_{g,n+m}$ equals
\begin{eqnarray*}
\prod_{i=1}^{m} \frac{({p_i-1})^{p_i-1}}{(p_i-1)!} \cdot \frac{1-\lambda_1+\ldots+ (-1)^g \lambda_g}{\prod_{i=1}^{m}(1-(p_i-1) \psi_i)}.
\end{eqnarray*}
\end{mypr}

This proposition will be proved in Section 2. An important corollary of this proposition is that the push-forward of a tautological class from $\P\overline{\mathcal{H}}_{g,n,P}$ to $\oM_{g,n+m}$ is tautological.

\subsection{Statement of the results}\label{ssec:results}

Now, we have all elements to state the main theorems of this article.

\begin{myth}\label{main} For any vectors $Z$ and $P$, the class $\left[\P \overline{A}_{g,Z,P}\right]$ introduced in Notation~\ref{Not:AZP},  lies in the tautological ring of $\mathbb{P}\overline{\mathcal{H}}_{g,n,P}$ and is explicitly computable.
\end{myth}

The main ingredient to prove this theorem will be the induction formula of Theorem~\ref{ind}. 

\begin{mydef} Let $V$ be a vector, in this article we will denote by $|V|$ the sum of elements of $V$ and by $\ell(V)$ the length of $V$.

Given $g$ and $P$, we will say that $Z$ is \textit{complete} if it satisfies $|Z| -|P|=2g-2$. If $Z$ is complete, we denote by $Z-P$ the vector $(k_1,\ldots,k_n,-p_1,\ldots,-p_m)$.
\end{mydef}

Restricting ourselves to the holomorphic case and applying the forgetful map of the marked points we obtain the following corollary.

\begin{myth}\label{mainbis} For any complete vector $Z$, the class $\left[\P\oH_{g}(Z)\right]$ introduced in Notation~\ref{Not:Z} lies in the tautological ring of $\mathbb{P}\overline{\mathcal{H}}_{g}$ and is explicitly computable.
\end{myth}

\begin{remark} As a guideline for the reader, it will be important to understand that the holomorphic case in Theorem~\ref{main} cannot be proved without using strictly meromorphic differentials. Thus  Theorem~\ref{mainbis} is a consequence of a specific case of Theorem~\ref{main} but one cannot avoid to prove Theorem~\ref{main} in its full generality.
\end{remark}

The second important corollary is obtained by forgetting the differential instead of the marked points. Let $P=(p_1,\ldots,p_m)$ be a vector of poles and $Z=(k_1,\ldots,k_n)$ be a complete vector of zeros. We define $\mathcal{M}_g(Z-P)\subset \mathcal{M}_{g,n+m}$ as the locus of  points $(C, x_1,\ldots, x_n)$ that satisfy
$$\omega_C\big(-\sum_{i=1}^n k_i x_i + \sum_{j=1}^{m} p_j x_{n+j}\big)\simeq \mathcal{O}_C.$$
We denote by $\oM_g(Z-P)$ the closure of $\mathcal{M}_g(Z-P)$ in $\oM_{g,n+m}$.

\begin{myth}\label{mainter}
For any vectors $Z$ and $P$, the class $\left[\oM_g(Z-P)\right]$ lies in the tautological ring of $\oM_{g,n+m}$ and is explicitly computable.
\end{myth}

\begin{remark}
Theorems \ref{main},~\ref{mainbis} and~\ref{mainter} are stated for the Poincar\'e-dual rational cohomology classes. However, all identities of this paper are actually valid in the Chow groups.
\end{remark}

In a second part of the text (Section~\ref{sec:picard}) we will consider the rational Picard group of the space $\oM_g(Z-P)$. We will define several natural classes in this Picard group and apply the tools developed in the first part of the paper to deduce a series of relations between these classes (see Theorem~\ref{th:rel}). 

\subsection{An example}

Here we illustrate the general method used in this article by computing the class of differentials with a double zero $\left[\P\oH_{g}(2,1,\ldots,1)\right]$. This computation was carried out by D. Zvonkine in an unpublished note~\cite{Zvo} and was the starting point of the present work. 

We begin by marking a point, i.e. we study the space $\mathbb{P}\overline{\mathcal{H}}_{g,1}$ of triples $(C,x_1,[\alpha])$ composed of a stable curve $C$ with one marked point $x_1$ and an abelian differential $\alpha$ modulo a multiplicative constant. Recall that $\P\overline{A}_{g,(2)}\subset \P\oH_{g,1}$ is the closure of the locus of smooth curves with a double zero at the marked point.  In order to compute $[\P\overline{A}_{g,(2)}]$, we consider the line bundle 
\begin{equation*}
\mathcal{L} \otimes \mathcal{L}_1 \simeq \rm{Hom}(\mathcal{L}^\vee, \mathcal{L}_1)
\end{equation*}
over $\mathbb{P}\overline{\mathcal{H}}_{g,1}$. (Recall that $\mathcal{L}^\vee$ is the dual tautological line bundle of the projectivization $\mathbb{P}\overline{\mathcal{H}}_{g,1}$ and $\mathcal{L}_1$ is the cotangent line bundle at the marked point~$x_1$.) We construct a natural section $s_1$ of this line bundle,
\begin{eqnarray*}
s_1 :   \mathcal{L}^\vee & \to & \mathcal{L}_1\\
  \alpha &\mapsto & \alpha(x_1).
\end{eqnarray*}
Namely, an element of $ \mathcal{L}^\vee$ is an abelian differential on~$C$, and we take its restriction to the marked point. 

The section~$s_1$ vanishes if and only if the marked point is a zero of the abelian differential. Thus we have the following identity in $H^2(\mathbb{P}\overline{\mathcal{H}}_{g,1})$:
\begin{equation*}
[\P\overline{A}_{g,(1)}]=[\{s_1=0\}] = c_1(\mathcal{L}\otimes \mathcal{L}_1)=\xi+\psi_1.
\end{equation*}

Now we restrict ourselves to the locus $\{s_1=0\}$ and consider the line bundle  
\begin{equation*}
\mathcal{L}\otimes \mathcal{L}_1^{\otimes 2}.
\end{equation*}
We build a section $s_2$ of this new line bundle. An element of ${\mathcal{L}}^{\vee}_{|\{s_1=0\}}$ is an abelian differential with at least a simple zero at the marked point~$x_1$. Its first derivative at $x_1$ is then an element of $\mathcal{L}_1^{\otimes 2}$ (we can verify this assertion using a local coordinate at $x_1$).

As before, $s_2$ is equal to zero if and only if the marked point is at least a double zero of the abelian differential. However, $\{s_2=0\}$ is composed of three components:
\begin{itemize}
\item $\P\overline{A}_{g,(2)}$;
\item the locus $a_e$ where the marked point lies on an elliptic component attached to the rest of the stable curve at exactly one point and the abelian differential vanishes identically on the elliptic component;
\item the locus $a_r$ where the marked point lies on a ``rational bridge'', that is, a rational component attached to two components of the stable curve that are not connected except by this rational component (in this case the abelian differential automatically vanishes on the rational bridge).
\end{itemize}

We deduce the following formula for $[\P\overline{A}_{g,(2)}]$:
\begin{eqnarray*}
[\P\overline{A}_{g,(2)}]&=& [\{s_2=0\}] -[a_e]-[a_r] \\
&=&(\xi+\psi_1)(\xi+2\psi_1)-[a_e]-[a_r] \\
&=&\xi^2+3 \psi_1 \xi +2\psi_1^2-[a_e]-[a_r].
\end{eqnarray*}

\begin{remark} We make a series of remarks on this result.
\begin{itemize}
\item To transform the above considerations into an actual proof we need to check that the vanishing multiplicity of $s_2$ along all three components equals~1. We will prove this assertion and its generalization in Section 3.
\item Denote by $\pi : \P\overline{\mathcal{H}}_{g,1}\to \P\overline{\mathcal{H}}_{g}$ the forgetful map, by $\delta_{\rm sep}$ the boundary divisor composed of curves with a separating node, and $\delta_{\rm nonsep}$ the boundary divisor of curves with a nonseparating node. Let us apply the push-forward by $\pi$ to the above expression of $[\P\overline{A}_{g,(2)}]$.
\begin{itemize}
\item
The term $\pi_*(\xi^2)$ vanishes by the projection formula, since it is a push-forward of a pull-back. 
\item
The term $\pi_*(3 \xi \psi_1)$ gives $3 \kappa_0 \xi = (6g-6) \xi$ by the projection formula.
\item 
The term $\pi_*(2\psi_1^2)$ gives $2 \kappa_1$. 
\item
The term $\pi_*([a_e])$ vanishes, because the geometric image of $a_e$ is of codimention~2 in $\P\overline{\mathcal{H}}_{g}$.
\item
The term $\pi_*([a_r])$ gives $\delta_{\rm sep}$ since $\pi$ induces a degree one map from $a_r$ onto $\delta_{\rm sep}$.
\end{itemize}
Thus we get
\begin{equation*}
[\P\oH(2,1,\ldots,1)]=\pi_*[\P\overline{A}_{g,(2)}]=(6g-6)\xi+2\kappa_1-\delta_{\rm sep}.
\end{equation*}
Using the relation $\kappa_1=12 \lambda_1 - \delta_{\rm sep} - \delta_{\rm nonsep}$ on $\oM_g$ (see, for example, \cite{GOAC2}, chapter 17), we have
\begin{equation*}
[\P\oH(2,1,\ldots,1)]=(6g-6) \xi+24\lambda_1- 3 \delta_{\rm sep}- 2\delta_{\rm nonsep}.
\end{equation*}
This formula was first proved by Korotkin and Zograf in 2011 using an analysis of the Bergman tau function~\cite{KorZog}. Dawei Chen gave another proof of this result in 2013 using test curves \cite{Chen2}.
\item In general, to prove Theorem~\ref{main} we will work by induction. Let $Z=(k_1,k_2,\ldots,k_n)$ and $P$ be vectors of positive integers. Let $Z'=(k_1,\ldots,k_i+1,\ldots,k_n)$. Then we will show that
 \begin{equation*}
\left[\P\overline{A}_{g,Z',P}\right]=\left(\xi+(k_i+1)\psi_i\right) \left[\P\overline{A}_{g,Z,P}\right] - \text{ boundary terms}.
\end{equation*}
The computation of these boundary terms is the crucial part of the proof.
\end{itemize}
\end{remark}

\subsection{Applications and related work}\label{ssec:appli}

\subsubsection*{Classes in the Picard group of $\oM_g$} Scott Mullane and Dawei Chen gave a closed formula for the class of $\pi_*\left[\oM_{g}(Z)\right]$ in the rational Picard group of $\oM_{g}$ for all $Z$ of length $g-2$  (see \cite{Chen1} and \cite{Mul}). They used test curves and linear series to compute this formula. This result has the advantage of giving explicit expressions, however it has the drawback of not keeping track of the positions of the zeros and of being restricted to the vectors $Z$ of length $g-2$ (see Section~\ref{ssec:compu2} for an example of computation).

\subsubsection*{Incidence variety compactification.} The problem of the compactification of the strata is extensively studied from different approaches in a joint work of  Bainbridge, Chen, Gendron, Grushevsky, and Moeller (see \cite{BCGGM} and \cite{Gen}). Their compactification (called {\em incidence variety compactification}) is slightly different from the one that we use here. We will recall their definitions in Section~\ref{ssec:boundary} since we will make use of some of their results.

\subsubsection*{Moduli space of twisted canonical divisors.} In~\cite{FarPan}, Farkas and Pandharipande proposed another compactification of the strata. Let $g,n,m$ be non-integers such that $2g-2+n+m>0$. Let $P$ be a vector of positive integers of lenght $m$ and let $Z$ be vector of non-negative integers of length $n$ that is complete for $g$ and $P$. We recall that $\M_g(Z-P)\subset \M_{g,n+m}$ is the locus of smooth curves such that $\omega_C(-k_1x_1-\ldots-k_n x_x+p_1 x_{n+1}+\ldots+p_{m} x_{n+m})\simeq \O_C$ and that we denote by $\oM_g(Z-P)$ its closure in $\oM_{g,n+m}$.  In~\cite{FarPan}, Farkas and Pandharipande defined the space of {\em twisted canonical divisors denoted} by $\widetilde{\M}(Z-P)$. The space of twisted canonical divisors is a singular closed subspace of $\oM_{g,n+m}$ such that $\oM(Z-P)$ is one of the irreducible components of $\widetilde{\M}(Z-P)$.

We assume that $m\geq 1$. In the appendix of~\cite{FarPan}, Farkas and Pandharipande defined a class ${\rm H}_g(Z-P)$ in $A_g(\oM_{g,n+m})$ (or $H^{2g}(\oM_{g,n+m})$): this class is a weighted sum over the classes of irreducible components. 

\subsubsection*{Conjectural expression of ${\rm H}_g(Z-P)$.} Let $r$ be a positive integer and $(C,x_1,\ldots,x_{n+m})$ be a smooth curve with markings. A $r$-spin structure is a line bundle $L$ such that $L^{\otimes r}\simeq \omega_C(-k_1x_1-\ldots-k_n x_n+ p_1 x_{n+1}+ \ldots + p_m x_{n+m})$. We denote the moduli space of $r$-spin structures by $\M_{g,Z-P}^{1/r}$. This space admits a standard compactification by twisted $r$-spin structures: $\oM_{g,Z-P}^{1/r}$. We denote by $\pi:\oC_{g,Z-P}^{1/r}\to \oM_{g,Z-P}^{1/r}$ the universal curves and by $\L\to \oC_{g,Z-P}^{1/r}$ the universal line bundle. The moduli space of twisted $r$-spin structures has a natural forgetful map $\epsilon:\oM_{g,Z-P}^{1/r}\to \oM_{g,n+m}$; the map $\epsilon$ is finite of degree $r^{2g-1}$. We consider $R\pi_*(\L)$ the image of $\L$ in the derived category of $\oM_{g,Z-P}^{1/r}$. 

If $m\geq 1$, then we consider the class $c^r_g(Z-P)\overset{\rm def}{=} c_g\left(R\pi_*\L\right)\in A_g(\oM_{g,Z-P}^{1/r})$.  If $m=0$, then we consider a different class, namely Witten's class: $c_W^{r}(Z)\in A_{g-1}(\oM_{g,Z}^{1/r})$. There are several equivalent definitions of Witten's class, all of which require several technical tools that we will not describe here (see~\cite{PolVai},~\cite{Chi} or~\cite{ChaLiLi}). 

We consider the two following functions:
\begin{eqnarray*}
P_{g,Z-P}, \ P^W_{g,Z}: \N^*&\to& A_*(\oM_{g,n+m})\\
r&\mapsto & r \epsilon_* (c^{r}_g(Z-P)),\ r^{g-1}\epsilon_* (c_W^{r}(Z)).
\end{eqnarray*}
Both $P_{g,Z-P}$ and $P^W_{g,Z}$ are polynomials for large values of $r$ (this result is due to Aaron Pixton, see~\cite{JanPanPixZvo} and~\cite{PanPixZvo1}). We denote by $\widetilde{P}_{g,Z-P}$ and $\widetilde{P}^W_{g,Z}$ the asymptotic polynomials.  The two following conjectures have been proposed.

\begin{myconj} (see~\cite{FarPan})  If $m\geq 1$ then the equality ${\rm H}_g(Z-P)=\widetilde{P}_{g,Z-P}(0)$ holds in $A_g(\oM_{g,n+m})$.
\end{myconj}

\begin{myconj} (see~\cite{PanPixZvo})  If $m=0$ then the equality $[\oM_g(Z)]=(-1)^g \widetilde{P}^W_{g,Z}(0)$ holds in $A_{g-1}(\oM_{g,n})$.
\end{myconj}

As a consequence of Theorem~\ref{mainter}, we know that the classes ${\rm H}_g(Z-P)$ and $[\oM_g(Z)]$ are tautological and we have an algorithm to check the validity of the conjectures  case by case (see Section~\ref{ssec:compu2} for examples of computations). 

These two conjectures are the analogous for differentials of the formula for the so-called double-ramification cycles (DR cycles): the DR cycle is a natural extension of to $\oM_{g,n}$ of the cycle in $\M_{g,n}$ defined as the locus of marked curves $(C,x_1,\ldots, x_n)$ such that
$$
\sum_{i=1}^n a_i (x_i) \simeq \mathcal{O}_C
$$
for any fixed vector of integers $(a_i)_{1\leq i\leq n}$ such that  $\sum a_i=0$ (see~\cite{JanPanPixZvo}). 

\subsubsection*{Compactification via log-geometry.} J\'er\'emy Gu\'er\'e constructed a moduli space of  ``rubber'' differentials using log geometry. He proved that this space is endowed with a perfect obstruction theory. Moreover, if $m\geq 1$, this moduli space surjects onto the moduli space of twisted canonical divisors and the class ${\rm H}_g(Z-P)$ is the push-forward of the virtual fundamental cycle (see~\cite{gue}).

 If $m=0$ has only positive values, Dawei Chen and Qile Chen have also used log geometry to define a compactification of the strata $\H_g(Z)$ (see~\cite{CheChe}). 

\subsubsection*{Induction formula for singularities in families.} The central result of the present work is the induction formula of Section~\ref{sec:indfor}. A similar formula has been proved by  Kazarian, Lando and Zvonkine for classes of singularities in families of genus 0 stable maps (see~\cite{KazLanZvo}). Their formula contains only the genus 0 part of our induction formula. 

They gave an interpretation of the induction formula in genus 0 as a generalization of the completed cycle formula of Okounkov and Pandharipande (see~\cite{OkoPan}). For, now it is not clear if this generalized completed cycle formula has an extension to higher genera.

This type of induction formula had been previously introduced by Gathmann in the context of genus 0 Relative Gromov-Witten invariants (see~\cite{Gat}) and has been recently adapted to the genus 0 quasimap invariants (see~\cite{BatNab}).

\subsubsection*{Computation of the Lyapunov exponents of strata.} Strata of differentials are endowed with a structure of dynamical system. Several numerical invariants have been introduced to characterize the dynamics of the strata: volumes, Siegel-Veech constants,  Lyapunov exponents.  Some relations exist between these invariants. These relations come in general from relations in the cohomology of the strata.  

Our computation of cohomology classes of strata of differentials could be useful to compute these numerical invariants. This idea is developed for example in \cite{KorZog}  and \cite{Chen1} based on the work of Eskin, Kontsevich, and Zorich \cite{EskKonZor} (see Section~\ref{ssec:KonZor}). This has been explored in the subsequent paper (see~\cite{Sau4})

\subsection{Plan of the paper}

 In Section~\ref{sec:stdiff}  we construct the space of stable differentials and compute its  Segre class. Then we generalize the definition of stable differentials for disconnected curves and for unstable irreducible curves. In the last subsection we present the tautological rings of spaces of stable differentials in this most general setting (with possible disconnected and semi-stable curves).

 In Section~\ref{sec:stratification} we introduce the stratification of the interior of spaces of stable differentials according to the orders of zeros and we study the geometry of the strata: local parameters, dimension, neighborhood in the space of differentials. Then, in Section~\ref{sec:boundary}  we describe the boundary components of the Zariski closure of strata.
 
Theorems \ref{main}, \ref{mainbis} and \ref{mainter} are proved in In Section~\ref{sec:indfor}. The main tool involved in their proof is the induction formula for the Poincar\'e-dual classes of strata of differentials with prescribed orders of zeros (see Theorem~\ref{ind}). 
  
  In Section~\ref{sec:compu} we present two examples of explicit computations. 
  
  Finally, in Section~\ref{sec:picard} we introduce several classes in the Picard group of strata of differentials and prove several relations between these classes by using the induction formula.
 
 \subsubsection*{Acknowledgments.} I am very grateful to my PhD. advisor Dimitri Zvonkine who proposed me this interesting problem, supervised my work with great implication and brought many beneficial comments.
 
I would also like to thank J\'er\'emy Gu\'er\'e, Charles Fougeron, Samuel Grushevsky,  Rahul  Pandharipande, Dawei Chen, Qile Chen and Anton Zorich, for helpful conversations. They have enriched my understanding of the subject in many ways. I am grateful to Felix Janda for the computations that he made and for the long discussions we had regarding the present work. Moreover, I would like to thank the organizers of the conference ``Dynamics in the Teichm\"uller Space'' at the CIRM (Marseille, France) for their invitation as well as the people I have met for the first time at this conference : Pascal Hubert, Erwan Lanneau, Samuel Leli\`evre,  Anton Zorich, Quentin Guendron and Martin M\"oller.

Finally, I am very grateful to the anonymous referee of {\em Geometry and Topology} whose dedicated work significantly contributed to the final shape of the paper. 

\section{Stable differentials}\label{sec:stdiff}

In this section, we construct the space of stable differentials and compute its Segre class. We also define stable differentials on disconnected and/or unstable curves. Finally, we define and describe the tautological rings in this generalized set-up. 

\subsection{The cone of generalized principal parts}\label{ssec:principalparts}

\subsubsection{Orbifold cones}
We follow here the approach of \cite{ELSV}. Let $X$ be a projective DM stack.

\begin{mydef} An {\em orbifold cone} is a finitely generated sheaf of graded $\O_X$-algebras  $S=S^0\oplus S^1 \oplus S^2\oplus \ldots$ such that $S^0=\O_X$.
\end{mydef}

\begin{remark} This definition of cone is weaker than the classical definition of Fulton (see~\cite{Fulton}) because we do not ask that $S$ be generated by $S^1$. In the classical definition, a cone is a subvariety of a vector bundle (the dual of $S^1$) given by homogeneous equations. Its projectivization is a subvariety of a bundle of projective spaces. In the orbifold case, the cone is, again, a suborbifold of a vector bundle, but is now given by quasi-homogeneous equations. Its projectivization is a suborbifold of the corresponding bundle of weighted projective spaces, which carries a tautological line bundle $\O(1)$ in the orbifold sense (called canonical line bundle in~\cite{Fulton}). Thus the projectivization $\P\mathcal{C}$ of a cone is an orbifold and carries a natural orbifold line bundle $\O(1)$, the tautological line bundle. We denote $p:\P\mathcal{C}={\rm{Proj}}(S) \to X$ and  $\xi=c_1(\mathcal{O}(1))$. Let $\mathcal{C}\to X$ be a pure-dimensional cone  and $r$ the rank of the cone defined as ${\rm dim}(\mathcal{C})-{\rm dim}(X)$. The $i$-th Segre class of $\mathcal{C}$ is defined as
\begin{equation*}
s_i=p_*(\xi^{r+i-1}) \in H^{2i}(X,\Q).
\end{equation*}
\end{remark}

\begin{example} Let us consider the graded algebra $\C[x,y,z]$ such that $x$ is an element of weight 2, $y$ is an element of weight 3 and $z$ is an element of weight 1. This graded algebra is not generated by its degree 1 elements. The associated projectivized cone over a point is the weighted projective space $\P(2,3)$ which is the quotient of $(\C^3)^*$ by $\C^*$ with the action:
$$
\lambda \cdot (x,y,z)= ( \lambda^2 x,\lambda^3 y, \lambda z).
$$
\end{example}

\begin{example}
More generally, consider a sheaf of algebras of the form $\O_X \otimes_{\C} S$, where $S$ is a graded algebra over~$\C$. The projective spectrum of this sheaf is a direct product of $X$ with ${\rm{Proj}}(S)$. We call this a {\em trivial orbifold cone}. 
\end{example}

\subsubsection{Cone of generalized principal parts}

\begin{mydef} Let $p$ be an integer greater than 1. A {\em principal part} of order $p$ at a smooth point of a curve is an equivalence class of germs of meromorphic differentials with a pole of order $p$ ; two germs $f_1,f_2$ are equivalent  if $f_1-f_2$ is a meromorphic differential with  at most a simple pole. 
\end{mydef}

First, we parametrize the space of principal parts at a point. Let $z$ be a local coordinate at 0 $\in \C$. A principal part at 0 of order $p$ is given by:
\begin{equation*}
\left[\left(\frac{u}{z}\right)^{p-1}+a_1\left(\frac{u}{z}\right)^{p-2}+\ldots+a_{p-2} \left(\frac{u}{z}\right)\right] \frac{dz}{z}
\end{equation*}
with $u\neq 0$.  However, given a principal part, the choice of $(u,a_1,\ldots,a_{p-2})$ is not unique. Indeed there are $p-1$ choices for $u$ given by the $\zeta^\ell\cdot u$ (with $\zeta^\ell=\exp(\frac{2i\pi\cdot\ell}{p-1})$, for $0\leq \ell \leq p-1$) and, once the value of $u$ is chosen, the $a_i$'s are determined uniquely. Therefore the coordinates $(u,a_1,\ldots,a_{p-2})$ parametrize a degree $p-1$ covering of the space of principal parts. This motivates the following definition.

\begin{mydef}
Assign to $u$ the weight $1/(p-1)$ and to $a_j$ the weight $j/(p-1)$.
The graded algebra $S \subset \C[u,a_1,\ldots,a_{p-2}]$ spanned by the monomials of integral weights is called the \textit{algebra of generalized principal parts} and $\oP={\rm{Spec}}(S)$ is the \textit{space of generalized principal parts}. 
\end{mydef}

The space $\oP$ is the quotient of $\C^{p-1}$ by the group $\Z\big/(p-1)\Z$, which, from now on, we will denote by $\Z_{p-1}$ for shortness. An element $\zeta \in \Z_{p-1}$ acts by
\begin{equation*}
\zeta \cdot (u,a_1,\ldots,a_{p-2})= (\zeta u,\zeta a_1,\ldots,\zeta^{p-2} a_{p-2}).
\end{equation*}
Moreover, the natural action of $\C^*$ on $\oP$ is given by
\begin{equation*} 
\lambda\cdot (u,a_1,\ldots,a_{p-2})= (\lambda^{\frac{1}{p-1}} u,\lambda^{\frac{1}{p-1}} a_1,\ldots,\lambda^{\frac{p-2}{p-1}} a_{p-2}).
\end{equation*}
Note that this action is not well-defined on the covering space $\C^{p-1}$, but is well defined on its $\Z_{p-1}$ quotient~$\oP$.  

\begin{mynot}
Denote by $I_u \subset S$ the ideal of polynomials divisible by~$u$. Denote by $\mathcal{A} \subset \oP$ the suborbifold defined by~$I_u$.
\end{mynot} 

The suborbifold $\mathcal{A} \subset \oP$ is the Weil divisor obtained as the image of the Cartier divisor $\{u=0\}\subset \C^{p-1}$ under the quotient of $\C^{p-1}$ by the action of $\Z_{p-1}$. The divisor $(p-1)\mathcal{A}$ is the Cartier divisor given by the equation $u^{p-1}=0$. (Note that $u^{p-1}$ lies in $S$ while $u$ does not.) The space of principal parts embeds into $\oP$ as the complement of $\mathcal{A}$.

\begin{mylem} \label{lem:varchange}
A change of local coordinate $z$ induces an isomorphism of $S$ that preserves the grading and acts trivially on the quotient algebra $S/I_u$.
\end{mylem}

\begin{proof}
Let $z=f(w)=\alpha_1 w+\alpha_2 w^2+\ldots$ be a local coordinates change. We denote by $(u',a_1',\ldots,a_{p-2}')$ the parameters of the presentation of principal parts in coordinate $w$. We have the transformation:
\begin{eqnarray*}
u&\mapsto& \alpha_1 u\\
a_1 &\mapsto& a_1 + \gamma_{1,1} u \\
a_2 &\mapsto& a_2 + \gamma_{2,1} u a_1 + \gamma_{2,2} u^2 \\
...
\end{eqnarray*}
where the $\gamma_{i,j}$ are polynomials in $\alpha_1, \alpha_2, \dots$ depending only on the order of the principal part. By taking $u$ to be $0$, we see that the coordinates $(a_1,\ldots,a_{p-2})$ of $\mathcal{A}$ are independent of the choice of local coordinate.
\end{proof}

\begin{remark}
In Section~\ref{ssec:stdiff} we will see that the locus $\mathcal{A}$ corresponds to the appearence of a semi-stable bubble of the underlying curve~$C$ at the $i$th marked point. The coordinate on the bubble is $w=u/z$.
\end{remark}

\begin{remark} The cone of principal parts of differentials differs from the cone of principal parts of functions of \cite{ELSV} only by the coefficients $\gamma_{i,j}$.
\end{remark}

Now, let $g,n$ be nonnegative integers such that $2g-2+n> 0$. Let $i\in [\![1,n]\!]$ and $p_i\geq2$. We denote by $\P_i$ the following sheaf of graded algebras over $\oM_{g,n}$.

Pick an open chart $U \subset  \oM_{g,n}$ together with a trivialization of a tubular neighborhood of the $i$th section $\sigma_i$ of the universal curve over~$U$. In other words, denoting by $\Delta$ the unit disc, we choose an embedding 
$$
U \times \Delta \hookrightarrow \oC_{g,n}
$$
commuting with $U \hookrightarrow \oM_{g,n}$ and such that $U \times \{0 \}$ is the $i$-th section of the universal curve. The sheaf $\P_i$ over $U$ is given by $\P_i(U) = \O_U \otimes S$.

Now, given two overlapping charts $U$ and $V$ we need to define the gluing map between the sheaves on their intersection. To do that, denote by $z$ the coordinate on $\Delta$ in the product $U \times \Delta$ and by $w$ the coordinate on $\Delta$ in the product $V \times \Delta$. Over the intersection $U \cap V$ we get a change of local coordinates $z(w)$. We use this change of local coordinate and the constants $\gamma_{i,j}$ from Lemma~\ref{lem:varchange} to construct an identification between the two algebras $\P_i(U)|_{U \cap V}$ and $\P_i(V)|_{U \cap V}$. 

Note that the sheaf of ideals $I_u$ is well-defined and the quotients $S/I_u$ are identified with each other in a canonical way that does not depend on the local coordinates $z$ and $w$. 
 
We denote by $\mathcal{P}_{i}={\rm{Spec}}(\P_i)$ the spectrum of $\P_i$ and by $\mathcal{A}_{i} = {\rm{Spec}}(\P_i/I_u)$ the spectrum of the quotient. The latter is a trivial cone over $\oM_{g,n}$. 

\begin{mypr}\label{normal}
The cone $\oP_i$ and its projectivization are normal.
\end{mypr}

\begin{proof}
Indeed the space $\oM_{g,n}$ is smooth and the sheaf of fractions of the algebra $\mathbb{P}_i$ is the same as the sheaf of fractions of $\mathbb{P}_i^1$. 
\end{proof}

\begin{mylem} The cone $\mathcal{A}_{i}$ is the product of $\oM_{g,n}$ with the weighted projective space with weights $(\frac{1}{p_i-1},\ldots,\frac{p_i-2}{p_i-1})$ quotiented by the action of $\Z_{p_i-1}$. Moreover the Segre classes of $\mathcal{A}_{i}$ and $\mathcal{P}_{i}$ are given by
\begin{eqnarray*}
s(\mathcal{A}_{i})&=&\frac{{(p_i-1)}^{p_i-2}}{(p_i-1)!}\\
s(\mathcal{P}_{i})&=& \frac{{(p_i-1)}^{p_i-1}}{(p_i-1)!}\cdot \frac{1}{1-(p_i-1) \psi_i}.
\end{eqnarray*}
\end{mylem}

\begin{proof} The proof is based on the same arguments as for the cone of principal parts of functions. The section $u^{ p_i-1}$ is a section of the line bundle $\mathcal{L}_i^{-\otimes (p_i-1)}$ which vanishes with multiplicity $p_i-1$ along $\mathcal{A}_{i}$. 
\end{proof}

\subsubsection{Stack of generalized principal parts}

In the above paragraph we defined the cone of generalized principal parts which is a normal scheme over $\oM_{g,n}$. We introduce here another approach to the quotient by the $\Z_{p_i-1}$-action.
Let $\widetilde{\mathbb{P}}_i$ be the sheaf of algebra  defined locally by
$$
\widetilde{\mathbb{P}}_i(U)= \O_U[u,a_1,\ldots,a_{p_i-2}]
$$
where $U$ is a chart with a trivialization of a tubular neighborhood of the $i$-th section of the universal curve and the coordinates $(u,a_1,\ldots,a_{p_i-2})$ are defined as above.

\begin{mydef} The  {\em stack of generalized principal parts} $\mathfrak{P}_i$ is the stack quotient  $${\rm Spec}(\widetilde{\mathbb{P}}_i)/ \Z_{p_i-1}.$$ 
\end{mydef}

By construction we have the following proposition.
\begin{mypr}\label{coarse1}
For all schemes $U$ with a map $U\to \oM_{g,n}$, the scheme $U\times_{\oM_{g,n}} \oP_i$ is the coarse space of $U\times_{\oM_{g,n}} \mathfrak{P}_i$.
\end{mypr}

\begin{mypr}\label{pr:smooth}
The stack of generalized principal parts is a smooth DM stack. 
\end{mypr}

\begin{proof}
The space $\oM_{g,n}$ is a smooth DM stack and $\mathfrak{P}_i$ is locally the quotient of an affine smooth scheme over $\oM_{g,n}$ by a finite group. 
\end{proof}

\subsubsection{Cones of generalized principal parts and jet bundles}
From now on in the text, unless otherwise mentioned,  for any family of semi-stable curves $C\to S$ we denote by $\omega$ the relative cotangent line bundle $\omega_{C/S}$.

\begin{mydef}\label{poljet} Let $\pi: \oC_{g,n}\to \oM_{g,n}$ be the universal curve and $(\sigma_i)_{1\leq i\leq n}:\oM_{g,n}\to \oC_{g,n}$ the global sections of marked points. Let $1\leq i\leq n$ and $p_i\geq 1$. The vector bundle $J_i\to \oM_{g,n}$ of {\em polar jets of order $p_i$ at the $i$-th marked point} is defined as the quotient
$$
J_i = R^0\pi_*\left(\omega(p_i \sigma_i)\right) \big/  R^0\pi_*\left(\omega(\sigma_i)\right).
$$
\end{mydef}
We fix $1\leq i\leq n$ and $p_i>0$. The bundle of polar jet of order $p_i$ is a vector bundle of rank $p_i-1$. As before, we consider an open chart $U$ of $\oM_{g,n}$ with a trivialization $z_i$ of a tubular neighborhood of the section $\sigma_i$. Over the chart $U$ the jet bundle is trivial.  Indeed an element of $J_i$ over $U$ is given by
$$
\left[ \frac{b_0}{z_i^{p_i-1}} + \ldots + \frac{b_{p_i-2}}{z_i^{p_i-2}}\right] \frac{dz_i}{z_i}.
$$
Thus, the jet bundle $J_i$ restricted to $U$ is given by ${\rm Spec}(\O_U[b^i_0,\ldots, b_{p_i-2}])$. Recall that, using the trivialization $z_i$ we have defined coordinates $u, a_1,\ldots, a_{p_i-2}$ such that $\mathbb{P}_i(U)$ is the sub-algebra of $$ \O_U[u, a_1,\ldots, a_{p_i-2}]$$ generated by monomials with integral weights. We define the following morphism of graded algebras over $\O_U$
\begin{eqnarray*}
\phi_i(U): &{\rm Sym}^* (J^{i\; \vee})(U) & \to \mathbb{P}_i(U)\\
&b_0  & \mapsto u^{p_i-1},\\
&b_j  & \mapsto u^{p_i-1-j}a_{j} \; \; \text{(for $1\leq j\leq p_i-2$)}.
\end{eqnarray*}
The morphism $\phi_i(U)$ is defined for a chart $U$ with a choice of trivialization $z_i$. We can easily check that the $\phi_i(U)$ can be glued into a morphism of sheaves of graded algebras. Thus we have constructed a morphism of cones
$$
\phi_i: \oP_i \to J_i.
$$

It is important to remark that for $p_i \geq 3$ the morphism $\phi_i$ is neither surjective nor injective. 
\begin{mylem}\label{lem:bij1}
We define the following two spaces
\begin{eqnarray*}
\oP_i\supset \widetilde{\oP}_i&=& \left(\oP_i\setminus \mathcal{A}_i\right) \cup \; \text{\rm  the zero section,}\\ 
J_i\supset \widetilde{J}_i&=&\left( J_i\setminus \{b_0=0\} \right) \cup \; \text{\rm the zero section.}
\end{eqnarray*}
The image of the morphism $\phi_i$ is the space $\widetilde{J}_i$. Moreover, the morphism $\phi_i$ restricted to $\widetilde{\oP}_i$ induces an isomorphism from $\widetilde{\oP}_i$ to $\widetilde{J}_i$.
\end{mylem}

The proof is a simple check.

\begin{remark}
Note in particular that the morphism $\phi_i$ does not define a morphism of projectivized cones. Indeed, certain points outside of the zero section of $\oP_i$ are mapped to zero section of~$J_i$.
\end{remark}

\subsection{The space of stable differentials}\label{ssec:stdiff}

Let $g,n,$ and $m$ be nonnegative integers satisfying $2g-2+n+m>0$.  Let $P=(p_1,p_2,\ldots,p_m)$ be a vector of positive integers. For all $1\leq i\leq m$, we denote by $\oP_{n+i}$ (respectively $\mathfrak{P}_{n+i}$ and $J_{n+i}$) the cone of principal parts (respectively the stack of principal parts and the vector bundle of polar jets) of order $p_i$ at the $(n+i)$-th marked point. Let $
p:\overline{\mathfrak{H}}_{g,n,P}\to \oM_{g,n+m} $ be the space of stable differentials of Definition~\ref{def:stablestack} together with the forgetful map. 

We recall that $\pi: \oC_{g,n+m}\to \oM_{g,n+m}$ is the universal curve and the $(\sigma_i)_{1\leq i\leq n+m}:\oM_{g,n+m}\to \oC_{g,n+m}$ are the global sections corresponding to marked points. 

\begin{mynot}\label{notKM} Let $K\oM_{g,n}(P)\to \oM_{g,n+m}$ be the vector bundle
$$
R^0\pi_*\big(\omega\big(\sum_{i=1}^m p_i \sigma_{n+i}\big)\big) \to \oM_{g,n+m}.
$$ 
It is a vector bundle of rank $g-1+\sum p_i$ if $P$ is not empty.
\end{mynot}
We have the following exact sequence of vector bundles over $\oM_{g,n+m}$
\begin{equation}\label{jetexact}
0 \to R^0\pi_*\big(\omega\big(\sum_{i=1}^m \sigma_{n+i}\big)\big) \to K\oM_{g,n}(P) \to \bigoplus_{i=1}^{m} J^{n+i} \to 0,
\end{equation}
This exact sequence is simply the long exact sequence obtained from the residue exact sequence.

\begin{mypr}\label{equivstack} The stack $\ofH_{g,n,P}$ is isomorphic to the fiber product of $K\oM_{g,n}(P)$ and $\bigoplus_{i=1}^m \mathfrak{P}_{n+i}$ over 
 ${\bigoplus_{i=1}^m J_{n+i}}$
where the map $ \mathfrak{P}_{n+i}\to J_{n+i}$ is the composition of maps $\mathfrak{P}_{n+i} \to {\oP}_{n+i}\overset{\phi_i}{\to} J_{n+i}\!$.
\end{mypr}

\begin{proof} We denote by  $\widetilde{\mathfrak{H}}_{g,n,P}$ the fiber product 
\begin{equation}\label{cartesianspec1}
\xymatrix{
\widetilde{\mathfrak{H}}_{g,n,P} \ar[r] \ar[d] & \bigoplus_{i=1}^m \mathfrak{P}_{n+i}\ar[d] \\ 
K\oM_{g,n}(P) \ar[r] & {\bigoplus_{i=1}^m J_{n+i}}.}
\end{equation}
 We construct the two directions of the isomorphism $\widetilde{\mathfrak{H}}_{g,n,P}\simeq \ofH_{g,n,P}$ separately.
\bigskip

\noindent {\em From $\ofH_{g,n,P}$ to $ \widetilde{\mathfrak{H}}_{g,n,P}$.} To construct a morphism $F_1: \ofH_{g,n,P} \to \widetilde{\mathfrak{H}}_{g,n,P}$ we define morphisms $\Phi_i:\ofH_{g,n,P}\to \mathfrak{P}_{n+i}$ for all $1\leq i\leq m$ and $\chi: \ofH_{g,n,P}\to K\oM_{g,n}(P)$ fitting in the diagram~\eqref{cartesianspec1}.

Let $(C\to S,\sigma_1,\ldots, \sigma_{n+m}, \alpha)$ be a family of stable differentials. Let $s\to S$ be a geometric point of $S$ and $(C_s, x_1,\ldots, x_{n+m}, \alpha_s)$ be the stable differential determined by $s$. The element $\Phi_i(\alpha_s) $ is determined as follows 
\begin{itemize}
\item If $x_{n+i}$ does not belong to a rational component then $\Phi_i(s)$ is the  principal part at the marked point. It belongs to $\mathfrak{P}^{n+i}\setminus \{u=0\}$.
\item  If $x_{n+i}$ belongs to a rational component, let $w_{n+i}$ be a global coordinate of the rational component such that: $x_{n+i}$ is at infinity, the node is at 0 and the term of $\alpha$ in front of $w_{n+i}^{p_i-2}dw_{n+i}$ is $-1$. Then $\alpha_s$ is of the form 
\begin{equation*}\label{coordglob}
-\left(w_{n+i}^{p_i-1}+ a_1 w_{n+i}^{p_i-2}+\ldots + a_{p_i-2}w_{n+i}+ {\rm res}_{\sigma_{n+i}}(\alpha) \right) \frac{dw_{n+i}}{w_{n+i}}
\end{equation*}
and we set $\Phi_i(s)=(0,a_1,\ldots, a_{p_i-2})$. Indeed the substack $\{u=0\}$ is the quotient of a trivial vector bundle by $\Z_{p_i-1}$ and the $a_i$'s are the global coordinates of this vector bundle.
\end{itemize}
We will prove that the map $\Phi_i$ depends holomorphically on $s$. If $s$ is a point of the first type this is an obvious statement. If $s$ is a point of the second type, let $U$ be an open neighborhood of $s$ in $S$ with a trivialization $z_{n+i}$ of a tubular neighborhood of $\sigma_{n+i}$ in $C$ (see the previous section). Let $C'$ be the stabilization of $C$. The differential $\alpha$ restricted to $C'$ is a differential with poles of order at most $p_i$ at $z_{n+i}=0$. The differential $\alpha$ in this coordinate is given by 
\begin{equation}\label{coordglobz} {\alpha}=\left(\left(\frac{u}{z_{n+i}}\right)^{p_i-1}+\ldots + a^i_{p_i-2}\frac{u}{z_{n+i}}+ {\rm res}_{\sigma_{n+i}}(\alpha)+ \underset{z_{n+i\mapsto 0}}{O}(z_{n+i})\right) \frac{dz_{n+i}}{z_{n+i}}
\end{equation}
The value $u^{p_i-1}$ depends holomorphically on $s$ and up to a choice of smaller $U$ we can fix a choice of $(p_i+1)$-st root $u$. The function $u$ depends holomorphically on $s$. 

Now we use the function $u$ and the local trivialization $z_{n+i}$ to construct the family of semi-stable curves $C''\subset C'\times \P^1$  defined by the equation $z_{n+i}w=u$ (where, as previously, $w$ is  the global coordinate of the rational component and the pole is located at $w=\infty$). This family of curves is isomorphic to $C$ (the stabilization of $C''$ and $C'$ are isomorphic and each fiber of these two families have the same dual graph). In particular, $\alpha$ is a meromorphic differential on $C''$ with constant order of pole at $w=\infty$. Besides in the chart $w$ the highest order coefficient of $\alpha$ is given by $1$. In particular the coordinate $w$ is equal to the coordinate $w_{n+i}$ on the unstable rational component of the fiber of $s$. In the chart $w$, the meromorphic differential $\alpha$ is given by 
\begin{equation}\label{coordglobw}
{\alpha}=-\left(w^{p_i-1}+ a_1 w^{p_i-2}+\ldots + a_{p_i-2}w+ {\rm res}_{\sigma_{n+i}}(\alpha) + \underset{w\mapsto 0}O(w)\right) \frac{dw}{w}
\end{equation}
where the $a_i$'s depend holomorphically on $s$. Therefore $\Phi$ depends holomorphically on $s$.

Now, we construct the map $\chi: \ofH_{g,n,P}\to K\oM_{g,n}(P)$. Let $(C\to S,\sigma_1,\ldots, \sigma_{n+m}, \alpha)$ be a family of stable differentials. We denote by $\widetilde{C}\to S$ the stabilization of $C$ and by 
$\widetilde{\alpha}= \alpha|_{\widetilde{C}}$. The family $(\widetilde{C}\to S, \sigma_1,\ldots, \sigma_{n+m}, \widetilde{\alpha})$ is a section of $\omega_{C/S}(\sum p_i \sigma_{n+i})$, thus a map $S\to K\oM_{g,n}(P)$. By construction, the morphisms $\chi$ and the $(\Phi_i)_{i=1,\ldots,m}$ fit in diagram~\eqref{cartesianspec1}. 
\bigskip

\noindent {\em From $ \widetilde{\mathfrak{H}}_{g,n,P}$ to $\ofH_{g,n,P}$.} Let $S$ be a $\C$-scheme and let $S\to \widetilde{\mathfrak{H}}_{g,n,P}$ be a morphism. By composition with the morphism $\widetilde{\mathfrak{H}}_{g,n,P}\to K\oM_{g,n}(P)$, we get a family of stable curves $C\to S$ with $n+m$ sections  $\sigma_i$ and a section $\alpha$ of $\omega_{C/S}(\sum p_i \sigma_{n+i})$. The family $S\to \widetilde{\mathfrak{H}}_{g,n,P}$ determines also families of generalized principal parts.  From the family of meromorphic differentials $\alpha$ and the principal parts we will construct a family of stable differentials.

Let $z_{n+i}$ be local trivializations of the tubular neighborhoods of the sections $\sigma_{n+i}$ of the curve $C/S$ for $1\leq i\leq m$. Let $w_{n+i}$ be global coordinates of the complex plane. We denote by $(u^i,a^i_{1}, \ldots, a^{i}_{p_i-2})$ the standard coordinates of the principal parts $\mathfrak{P}_{n+i}$ obtained from the trivializations $z_{n+i}$.  We construct a family of semi-stable curves $\widetilde{C}\to S$ defined by the equation $z_{n+i}w_{n+i}= u^i$.  On the curve $\widetilde{C}$ we construct a differential $\widetilde{\alpha}$. This differential is given by the expression~\eqref{coordglobw} in coordinate $w_{n+i}$ and by the expression~\eqref{coordglobz} in coordinate $z_{n+i}$. The tuple $(\widetilde{C},\sigma_1,\ldots, \sigma_{n+m},\widetilde{\alpha})$ is a family of stable differentials over $S$. 

Therefore we have determined a morphism $F_2: \widetilde{\mathfrak{H}}_{g,n,P}\to \ofH_{g,n,P}$. By construction it is the inverse of $F_1$ previously defined.
\end{proof}
The following proposition finishes the proof of Proposition~\ref{cone} and thus completes Definition~\ref{def:stable}.
\begin{mypr}\label{coarse} We denote by ${\oH}_{g,n,P}$ the following fiber product (in the category of cones over $\oM_{g,n+m}$ or in the category of DM-stacks)
\begin{equation}\label{cartesianspec}
\xymatrix{
{\oH}_{g,n,P} \ar[r] \ar[d] &\ar[d]^{\bigoplus \phi_i} \bigoplus_{i=1}^m \oP_{n+i} \\
\ar[r]  K\oM_{g,n}(P) & \bigoplus_{i=1}^m J_{n+i}.
}
\end{equation}
Then space $\oH_{g,n,P}$ is the unique space that satisfies the properties of Proposition~\ref{pr:stable}.
\end{mypr}
\begin{proof}
The fact that $\oH_{g,n,P}$ satisfies the properties of   Proposition~\ref{pr:stable} is a direct consequence of Propositions~\ref{coarse1},~\ref{equivstack}. The uniqueness of this stack follows from the uniqueness of coarse spaces. 
\end{proof}
From now on we will denote by ${\rm stab}: \oH_{g,n,P}\to K\oM_{g,n}(P)$ the vertical projection in diagram~\eqref{cartesianspec}.

\subsection{Properties of spaces of stable differentials}\label{ssec:properties}

We keep the notation $g,n,m,$ and $P$ of the previous sections. We state here several general properties of $\ofH_{g,n,P}$ and $\oH_{g,n,P}$ that will be needed further in the text. 

\begin{mypr}
Suppose that $P$ is not empty. Then the spaces $\ofH_{g,n,P}$ $\oH_{g,n,P}$ are  irreducible  DM stacks of pure dimension $4g-4+\sum p_i$ and $\P\oH_{g,n,P}$ is a proper DM stack (of dimension one less). The space $\oH_{g,n,P}$ and its projectivization are normal. The space $\ofH_{g,n,P}$ is a smooth DM stack.

If $P$ is empty then both $\ofH_{g,n,P}$ and $\oH_{g,n,P}$ are isomorphic to the Hodge bundle, which is a smooth DM stack of dimension $4g-3$.
\end{mypr}

\begin{proof}
The first part of the proposition follows from Propositions~\ref{pr:smooth},~\ref{equivstack}, and~\ref{normal}. The second part is straightforward.
\end{proof}

We consider the following two maps: on the one hand the inclusion of vector bundles $R^0\pi_*\big(\omega\big(\sum_{i=1}^m \sigma_{n+i}\big)\big) \to K\oM_{g,n}(P)$, and on the other hand the zero map $K\oM_{g,n,P}\to \bigoplus\oP^{n+i}$. Then we get an embedding $R^0\pi_*\big(\omega\big(\sum_{i=1}^m \sigma_{n+i}\big)\big) \to \oH_{g,n,P}$ by the universal property of the cartesian diagram~\eqref{cartesianspec}.

\begin{mypr}
For all $g,n,$ and $P$, we have the following exact sequence of cones (in the sense of~\cite{Fulton}~Proposition~4.1.6) 
$$
0\to R^0\pi_*\big(\omega\big(\sum_{i=1}^m \sigma_{n+i}\big)\big) \to \oH_{g,n,P} \to \bigoplus_{i=1}^m \mathcal{P}_{n+i}\to 0.
$$
\end{mypr}

\begin{proof}
 By construction, the sheaf of algebras defining $\oH_{g,n,P}$  is locally the tensor product of the sheaves of algebras ${\rm Sym}^\vee\bigg( R^0\pi_*\big(\omega\big(\sum_{i=1}^m \sigma_{n+i}\big)\big)\bigg)$ and the $\mathbb{P}_{n+i}$.
\end{proof}

The action of $\C^*$ on the space $\oH_{g,n,P}$ is determined by multiplication of the differential by a scalar. Let us give a description of the $\C^*$-fixed locus, i.e. the locus of points that are invariant under the action of $\C^*$.

Let $(C,x_1,\ldots,x_{n+m})$ be a curve in $\oM_{g,n+m}$. We denote by $m'$ the number of entries of $P$ greater than $1$. From $C$ we construct a semi-stable curve $\widetilde{C}$ as follows. The curve $\widetilde{C}$ has $m'+1$ irreducible components: one main component isomorphic to $C$ and $m'$ rational components attached to $C$ at the marked points corresponding to poles of order greater than $1$. We mark points $(x_1', \ldots, x_{n+m}')$ on $\widetilde{C}$. The first $n$ marked points and the points corresponding to poles of order at most $1$ are on the main component and satisfy $x_i=x_i'$. The poles of orders greater than one are carried by the rational components.

Now we define a meromorphic differential $\alpha$ on $\widetilde{C}$ by
\begin{itemize}
\item  the differential $\alpha$ vanishes identically on the main component;
\item on an exterior rational component,  if we assume that the marked point is at 0 and the node at $\infty$ then $\alpha$ is given by $dz/z^{p_i}$.
\end{itemize}
 The tuple $(\widetilde{C},x_1',\ldots,x_{n+m}', \alpha)$ is a stable differential invariant under the action of $\C^*$. Indeed, let $\lambda$ be a scalar in $\C^*$, the differential $\lambda \alpha$ vanishes on the main component and $\lambda dz/z^{p_i}$ is equal to $dw/ w^{p_i}$ if we use the change of coordinate  $z=w/\lambda^{1/p_i}$ for any $p_i$-th root of $\lambda$. 
 
 Conversely any $\C^*$-invariant point of $\oH_{g,n,P}$ is of this type. Indeed $\oH_{g,n,P}$ is a cone thus the locus of $\C^*$-invariant points is  a section of this cone and we have constructed this section here.

\subsection{Residues}\label{ssec:res}

Let $g,n,m$ and $P$ be as in the previous sections. 

\begin{mydef} Let $\oR$ be the vector subspace of $\mathbb{C}^{m}$ defined by
\begin{equation*}
\oR=\{(r_1,r_2,\ldots,r_m), r_1+r_2+\ldots+r_m=0\}.
\end{equation*}
The vector space $\oR$ will be called the {\em space of residues}.  The  {\em residue map}  is the following map of cones over $\oM_{g,n+m}$ 
\begin{eqnarray*}
{\rm{res}}: \oH_{g,n,P} &\to& \oR\\
\alpha &\mapsto& ({\rm{res}}_{x_{n+1}}(\alpha),{\rm{res}}_{x_{n+2}}(\alpha),\ldots,{\rm{res}}_{x_{n+m}}(\alpha))
\end{eqnarray*}
where $\oR$ stands for the trivial cone. We use the same notation for the residue map ${\rm res}: K\oM_{g,n}(P) \to \oR$. In this case it is a morphism of vector bundles.
\end{mydef}
These two residue maps fit in the following commutative triangle
\begin{equation}\label{triangleres}
\xymatrix{
\oH_{g,n,P} \ar[r]^{\rm stab} \ar[rd]_{\rm res} & K\oM_{g,n}(P) \ar[d]^{\rm res} \\
& \oR.
}
\end{equation}

Let $\oH^0_{g,n,P}\subset \oH_{g,n,P}$ (respectively $K\oM^0_{g,n}(P) \subset K\oM_{g,n}(P)$) be the sub-cone (resp. sub vector bundle)    of differentials without residues. 

We recall that the Hodge bundle is by definition equal to $\oH_{g,n+m}=R^0\pi_*\omega$. The following sequence of vector bundles over $\oM_{g,n+m}$ is exact
\begin{equation}\label{exactres1}
0\to \oH_{g,n+m} \to R^0 \pi_*(\omega(\sum_{i=1}^m \sigma_{n+i})) \overset{\rm res}{\to} \oR \to 0
\end{equation}
(this is the exact sequence obtained from the residue exact sequence $0\to \omega_C(\sum x_i)\to \omega_C \to \C^{m}\to 0$). The vector bundle $K\oM^0_{g,n}(P) $ fits into the following commutative diagram of vector bundles over $\oM_{g,n+m}$:
\begin{equation}\label{exactres2}
\xymatrix{
0\ar[r] &K\oM^0_{g,n}(P)  \ar[r]& K\oM_{g,n}(P)  \ar[r]^{\rm res}& \ar[r] \oR & 0 \\
&\oH_{g,n+m} \ar[u]\ar[r] & R^0\pi_*(\omega(\sum_{i=1}^m \sigma_{n+i}))\ar[u] \ar[ur]
}
\end{equation}
where the central square is cartesian. The first line of diagram~\eqref{exactres2} is exact by exactness of the sequence~\eqref{exactres1}.  Therefore, the cone structure of $\oH^0_{g,n,P}$ can be defined equivalently from the cone structure of $\oH_{g,n,P}$ or by saying that $\oH^0_{g,n,P}$ is the fiber product
$$
\xymatrix{
\oH_{g,n,P}^{0} \ar[r] \ar[d] & \bigoplus P_{n+i} \ar[d]\\
K\oM^0_{g,n}(P) \ar[r] & \bigoplus J_{n+i}.
}
$$
We have the following exact sequence of cones
$$
0\to \oH_{g,n+m} \to \oH^0_{g,n,P} \to \bigoplus \oP_{n+i}\to 0.
$$
\begin{remark}
Note that we cannot say that sequence $$0 \to \oH^0_{g,n,P}\to \oH_{g,n,P}\to \oR\to 0$$ is exact because exactness for morphism of cones is ill-defined if the first term is not a vector bundle.
\end{remark}

More generally we define the following.
\begin{mydef}
Let $R$ be a vector subspace of $\oR$. Let
 $\oH^R_{g,n,P}\subset \oH_{g,n,P}$ (respectively $K\oM^R_{g,n}(P) \subset K\oM_{g,n}(P)$) be the sub-cone (resp. sub vector bundle)    of differentials with a vector of residues lying in $R$. We will call $R$ a {\em space of residue conditions}.
\end{mydef}

\begin{mylem}\label{lemres}
Let $R \subset \oR$ be a vector subspace.
\begin{itemize} 
\item The space $\oH_{g,n,P}^R$ is a closed subcone of $\oH_{g,n,P}$ of codimension $\dim(\oR/R)$ (where we set $\dim(\oR/R)=0$ if $P$ is empty)
\item The Segre classes of $\oH_{g,n,P}^R$ and $\oH_{g,n,P}$ are equal. 
\item The Poincar\'e-dual class of $\mathbb{P}\oH_{g,n,P}^R$ in $H^*(\P\oH_{g,n,P},\Q)$ is given by
\begin{equation*}
\left[\mathbb{P}\oH_{g,n,P}^R\right]=\xi^{\dim(\oR/R)}.
\end{equation*}
\end{itemize}
\end{mylem}

\begin{proof} Let us denote by ${\rm res}_R$ the composition of morphisms $\oH_{g,n,P}\to \oR\to \oR/R$ (we use the same notation for its alter ego for $K\oM_{g,n}(P)$). We denote by $\oH^{R}_{g,n+m}$ the kernel of the morphism
$$
R^0\pi_*(\omega(\sum_{i=1}^{m} \sigma_{n+i})) \overset{{\rm res}_R}{\to} \oR/R\to 0.
$$
It is a vector bundle of rank $g+\dim(R)$. By repeating the above argument, we have the following exact sequence of cones:
$$0\to  \oH_{g,n+m}^R \to \oH^{R}_{g,n,P}\to \oR/R.$$ 
We deduce from this exact sequence that:
\begin{itemize}
\item the co-dimension of $\oH^{R}_{g,n,P}$ in $\oH_{g,n,P}$ is $\dim(\oR/R)$;
\item the Segre class of $\oH^R_{g,n,P}$ is given by
$$
c_*\left(\oH^R_{g,n+m}\right)\cdot s_*\left(\bigoplus \oP_{n+i}\right)
$$
(see~\cite{Fulton} Proposition 4.1.6).
\end{itemize}
Besides, the vector bundle $\oR/R$ is trivial thus 
$$
c_*\left(\oH^R_{g,n+m}\right) = c_*\left(R^0\pi_*(\omega(\sum_{i=1}^{m} \sigma_{n+i})) \right) 
$$
and the Segre class of $\oH^{R}_{g,n,P}$ does not depend on the choice of $R$.

To prove the last statement, we study the vector bundle $\O(1)\otimes p^*(\oR/R)\to \P\oH_{g,n,P}$, where we recall that $p:\P\oH_{g,n,P}\to \oM_{g,n+m}$ is the forgetful map. We have $\O(1)\otimes p^*(\oR/R)\simeq \Hom(\O(-1),p^*(\oR/R))$. A section of this vector bundle is given by:
$$
s:\alpha \mapsto {\rm res}_R (\alpha).
$$
The vanishing locus of $s$ is $\P\oH_{g,n,P}^R$ which is of codimension $\dim(\oR/R)$ and irreducible. Thus the Poincar\'e-dual class of $\P\oH_{g,n,P}^R$ in $H^*(\P\oH_{g,n,P},\Q)$ is given by
$$
d\cdot c_{\rm top} (\O(1)\otimes p^*(\oR/R))= d\cdot \xi^{\dim(\oR/R)}
$$
where $d$ is a rational number. Besides the cones $\oH_{g,n,P}^R$ and $\oH_{g,n,P}$ have the same Segre class thus 
$$
s_0= p_*\left(\xi^{{\rm rk}(\oH_{g,n,P}) -1}\right)=p_*\left([\P \oH_{g,n,P}^R]\xi^{{\rm rk} (\oH^R_{g,n,P})-1}\right)=d s_0,
$$
and the coefficient $d$ is equal to $1$.
\end{proof}
\begin{mypr}
The Segre class of $\oH_{g,n,P}$ is given by
\begin{eqnarray*}
\prod_{i=1}^{m} \frac{({p_i-1})^{p_i-1}}{(p_i-1)!} \cdot \frac{1-\lambda_1+\ldots+ (-1)^g \lambda_g}{\prod_{i=1}^{m}\left(1-(p_i-1) \psi_i\right)}.
\end{eqnarray*}
\end{mypr}

\begin{proof}
From the above lemma, we have
\begin{eqnarray*}
s_*(\oH_{g,n,P})&=& s_*(\oH^0_{g,n,P})\\
&=& c_*( \oH_{g,n+m} )^{-1} \cdot s_*\left(\bigoplus_{i=n+1}^{m} \mathcal{P}_{n+i}\right)\\
&=& c_*( \oH_{g,n+m}^\vee ) \cdot s_*\left(\bigoplus_{i=n+1}^{m} \mathcal{P}_{n+i}\right)\\
&=& \prod_{i=1}^{m} \frac{({p_i-1})^{p_i-1}}{(p_i-1)!} \cdot \frac{1-\lambda_1+\ldots+ (-1)^g \lambda_g}{\prod_{i=1}^{m}\left(1-(p_i-1) \psi_i\right)}.
\end{eqnarray*}
From the third line to the fourth we have used the fact that $c(\oH_g)^{-1}=c(\oH_{g}^\vee)$ (see \cite{Mum}).
\end{proof}

\subsection{Unstable base}\label{ssec:unstable}

Here we extend the definition of the spaces of stable differentials to differentials supported on an unstable base.

\begin{mydef} \label{Def:semistability}
A triple $(g, n, P)$ composed of a nonnegative integers $g$ and $n$ and a vector $P$ of positive integers is {\em semi-stable} if either:
\begin{itemize}
\item $2g-2+n+\ell(P)>0$ (in which case we also say that $(g,n,P)$ is {\em stable}),
\item or $g=0$, $n=1$ and $P=(p)$ with $p>1$;
\item or $g=0$, $n=0$, $P=(1,p)$ with $p>1$.
\end{itemize}
\end{mydef} 

We want to define the space $\oH_{g,n,P}$ for all semi-stable triple. However, the space $\oM_{0,2}$ is empty thus we cannot define the spaces $\oH_{0,1+1,(p)}$ and  $\oH_{0,2,(1,p)}$ as cones over a moduli space of curves. Still, we can define the cone structure of these two spaces over ${\rm Spec}(\C)$.

The space $\oH_{0,1+1,(p)}$ is defined as the complement of $\{u=0\}$ in the space of generalized principal parts defined in Section~\ref{ssec:principalparts}. In other words $\oH_{0,1+1,(p)}$ is the spectrum of the graded subalgebra of $\C[a_1,\ldots, a_{p-2}]$ generated by monomials with integral weights (where the weight of $a_j$ is $j/(p-1)$).

The space $\oH_{0,2,(1,p)}$ is the spectrum of the graded subalgebra of $\C[a_1,\ldots, a_{p-2}, r]$ generated by monomials with integral weights where $r$ (for residue) has weight $1$.

\subsection{Stable differentials on disconnected curves} \label{ssec:disconnected}

In the paper, we will need stable differentials supported on disconnected.  Let $q$ be a positive integer, and 
\begin{eqnarray*}
\mathbf{g}&=&(g_1,g_2,\ldots,g_q),\\
\mathbf{n}&=&(n_1,n_2,\ldots,n_q),\\
\mathbf{m}&=&(m_1,m_2,\ldots,m_q)
\end{eqnarray*}
be lists of nonnegative integers, and let
\begin{equation*}
\mathbf{P}=(P_j)_{\leq j\leq  q}=(p_{j,i})_{\leq j\leq  q, 1\leq i \leq  m_j}
\end{equation*}
be a list of vectors of positive integers of length $m_j$.

\begin{mydef}\label{def:stabletrip1} The triple $(\bg,\bn,\bP)$ is {\em stable} (or {\em semi-stable}), if the triple $(g_j,n_j,P_j)$ is  stable (or semi-stable) for all $1\leq j\leq q$,  (see Definition~\ref{Def:semistability}).
\end{mydef}
 
Unless otherwise state, we assume from now that $(\bg,\bn,\bP)$ is semi-stable.
\begin{mydef} The {\em space of stable differentials of type $(\bg,\bn,\bP)$} is the space
\begin{eqnarray*}
\oH_{\mathbf{g},\mathbf{n},\mathbf{P}}&=&\prod_{i=1}^q {\oH_{g_i,n_i,P_i}}.
\end{eqnarray*}
We define the {\em interior} of $\oH_{\mathbf{g},\mathbf{n},\mathbf{P}}$ as the open sub-stack $\H_{\mathbf{g},\mathbf{n},\mathbf{P}}\subset \oH_{\mathbf{g},\mathbf{n},\mathbf{P}}$ of differentials supported on smooth curves.
\end{mydef}

\begin{mydef} The {\em reduced base of type} $(\bg,\bn,\bP)$ (or of type $(\bg,\bn,\bm)$) is the space
$$
\oM_{\mathbf{g},\mathbf{n},\mathbf{m}}^{\rm red}= \!\!\!  \prod_{\begin{smallmatrix} j\text{ such that } \\ 2g_j-2+n_j+m_j>0 \end{smallmatrix}} \!\!\!\!\!\!\!  \oM_{g_j,n_j+m_j},
$$
if the product is non-empty and ${\rm Spec}(\C)$ otherwise.
\end{mydef} 

\begin{mypr}
The space of stable differentials of type $\mathbf{P}$ is a cone over $\oM_{\mathbf{g},\mathbf{n},\mathbf{m}}$. If the triple $(\bg,\bn,\bP)$ is stable then the Segre class is given by
\begin{equation*}
s\left(\oH_{\mathbf{g},\mathbf{n},\mathbf{P}}\right)= \prod_{j=1}^q s\left(\oH_{g_j,n_j+m_j,P_j}\right),
\end{equation*}
where $s\left(\oH_{g_j,n_j,P_j}\right)$ is the pull-back of the Segre class of $\oH_{g_j,n_j,P_j}$ to the product  $\prod_{j=1}^q \oM_{g_j, n_j+m_j}$ under the $j^{\rm th}$ projection.
\end{mypr}

\begin{proof}
The proof is straightforward because the space $\oH_{\mathbf{g},\mathbf{n},\mathbf{P}}$  is a product of cones. 
\end{proof} 

To handle the residues, we extend the definition of the space of residues $\oR$:
\begin{eqnarray}\label{eq:defres}
 \oR=\bigoplus_{j=1}^q \oR_j  &=& \{ (r_{j,i})_{j,i} \text{ such that } \sum_{i=1}^{m_j} r_{j,i} =0, \forall j \in [1,q]\} \subset \mathbb{C}^{m_1+\ldots+m_q}  .
\end{eqnarray}

\begin{mydef} Let $R$ be a vector subspace of $\oR$. The space $\oH_{\mathbf{g},\mathbf{n},\mathbf{P}}^R$ is the space of stable differentials with residues lying in $R$. 
\end{mydef}

\begin{mylem} Let $R$ be a linear subspace of $\oR$. The space $\oH_{\mathbf{g},\mathbf{n},\mathbf{P}}^R$ is a subcone of $\oH_{\mathbf{g},\mathbf{n},\mathbf{P}}^R$ of codimension ${\rm{dim}}(\oR)-{\rm{dim}}(R)$ and we have:
\begin{itemize}
\item  the cones $\oH_{\mathbf{g},\mathbf{n},\mathbf{P}}$ and $\oH_{\mathbf{g},\mathbf{n},\mathbf{P}}^R$ have the same Segre class;
\item  the Poincar\'e-dual class of $[\P\oH_{\mathbf{g},\mathbf{n},\mathbf{P}}^R]$ in $H^*(\P\oH_{\mathbf{g},\mathbf{n},\mathbf{P}},\Q)$ is given by
\begin{equation*}
\xi^{{\rm{dim}}(\oR)-{\rm{dim}}(R)};
\end{equation*}
\end{itemize}
\end{mylem}

\begin{proof}
The proof of Proposition~\ref{lemres} can be adapted immediately to the general case.
\end{proof}

\begin{mydef}\label{def:gentaut} Let $p:\oH_{\mathbf{g},\mathbf{n},\mathbf{P}} \to \oM_{\mathbf{g},\mathbf{n},\mathbf{m}}^{\rm red}$ be the projection to the base. The {\em tautological ring of $\P\oH_{\bg,\bn,\bP}$} is the sub-ring of $H^*(\P\oH_{\bg,\bn,\bP})$ generated by $\xi=c_1(\O(1))$ and pull-backs by $p$ of tautological classes from the base $\oM_{\mathbf{g},\mathbf{n},\mathbf{m}}^{\rm red}$.  We denote the this ring by $RH^*(\P \oH_{\mathbf{g},\mathbf{n},\mathbf{P}})$.
\end{mydef}

\subsection{Semi-stable graphs}\label{ssec:semi-stable} 

Let $\mathbf{g},\mathbf{n},\mathbf{m},$ and $\mathbf{P}$ be lists of genera, numbers of marked points without poles, numbers of marked poles and vectors of positive integers indexed by $j\in [\![1,q]\!]$ as in the previous Section. We assume that $(\bg,\bn,\bP)$ is semi-stable. 

In this section we define a combinatorial object called semi-stable graphs. We show here that the space $\oH_{\mathbf{g},\mathbf{n},\mathbf{P}}$ has a natural stratification according to semi-stable graphs and that semi-graphs allow to define some tautological classes.

\begin{mydef} 
A \textit{semi-stable graph} of type $(\mathbf{g},\mathbf{n},\mathbf{P})$ is given by the data
\begin{equation*}
(V, H, g : V \to \mathbb{N}, a : H \to V, i : H \to H,E,  \pi^0(V,E) \simeq [\![1,q]\!] , L\simeq \bigcup_{j=1}^q [\![1,n_j+m_j]\!]  ),
\end{equation*}
satisfying the following properties:
\begin{itemize}
\item $V$ is a vertex set with a genus function $g$. 
\item $H$ is a half-edge set equipped with a vertex assignment $a$ and an involution~$i$;
\item the edge set $E$ is defined as the set of length 2 orbits of $i$ in $H$ (self-edges at vertices are permitted);
\item The graph $(V,E)$ has $q$ labeled connected components;
\item for all $1\leq j\leq q$, the genus of the connected component labeled by $j$ is  defined by $\sum g(v) + \#(E_j) - \#(V_j) + 1$ and is equal to $g_j$;
\item $L$ is the set of fixed points of $i$ called {\em legs};
\item  for all $1\leq j\leq q$, there are $n_j+m_j$ legs on the $j^{\rm th}$ connected component and this set of legs is identified with the set $[\![1,n_j+m_j]\!]$;
\item for each vertex $v$ in $V$ belonging to the $j$-th component:
\begin{itemize}
\item let $n(v)$ be the number of legs adjacent to $v$ with label at most $n_j$;
\item let $m(v)$ be the number of legs adjacent to $v$ with label at least $n_j+1$;
\item let $P'(v)=(P_{j, m-n_j })_{m\mapsto v, m>n_j}$: it is the vector obtained from $P_j$  by keeping only the entries associated to the legs of the second type adjacent to $v$. We denote by $P(v)$ the concatenation of $P'(V)$ with the vector $(1,\ldots,1)$ of length equal to number of half-edges adjacent to $v$ that are not legs; 
\end{itemize}
\item for each vertex $v$, the triple $(g(v),n(v),P(v))$ is semi-stable.
\end{itemize}
\end{mydef}

We define the following lists indexed by the vertices of $\Gamma$:
\begin{eqnarray*}
\mathbf{g}_\Gamma=(g(v))_{v\in V}&,& \mathbf{n}_\Gamma=(n(v))_{v\in V},\\ \mathbf{m}_\Gamma=(m(v))_{v\in V}&,&\mathbf{P}_\Gamma=(P(v))_{v\in V}.
\end{eqnarray*} 
The triple $(\mathbf{g}_\Gamma,\mathbf{n}_\Gamma,\mathbf{P}_\Gamma)$ is semi-stable (it is implied by the last condition of the definition of a semi-stable graph). We consider the space $\oH_{\mathbf{g}_\Gamma,\mathbf{n}_\Gamma,\mathbf{P}_\Gamma}$. We denote by $\oR_\Gamma$ the space of residues of $\oH_{\mathbf{g}_\Gamma,\mathbf{n}_\Gamma,\mathbf{P}_\Gamma}$. We define the subspace $R_\Gamma\subset \oR$ by the equations
$$
r_h+r_{h'}=0
$$ 
for all edges $e=(h,h')$.
\begin{mynot} Let $\Gamma$ be a semi-stable graph we denote by $\oH_{\Gamma}$ the moduli space $\oH^{R_\Gamma}_{\mathbf{g}_\Gamma,\mathbf{n}_\Gamma,\mathbf{P}_\Gamma}$ and by
$$\zeta^\#_\Gamma: \oH_{\Gamma} \to \oH_{\mathbf{g},\mathbf{n},\mathbf{P}}$$
the natural closed morphism.
\end{mynot}

\begin{mypr}
The set of semi stable graphs is finite and the space $\oH_{\mathbf{g},\mathbf{n},\mathbf{P}}$ is stratified according to the semi-stable graphs; i.e, for all $x$ in $\oH_{\mathbf{g},\mathbf{n},\mathbf{P}}$ there exists a unique graph $\Gamma$ such that $x\in \zeta_\Gamma(\H_{\Gamma})$.
\end{mypr}
\begin{proof}
If we fix the datum $(\mathbf{g},\mathbf{n},\mathbf{P})$, then there are finitely many semi-stable graphs $\Gamma$ for  $(\mathbf{g},\mathbf{n},\mathbf{P})$  such that the graph $\Gamma$ is stable. Indeed, there are finitely many stabilization of $\Gamma$ and then the graph $\Gamma$ is determined by the choice of which set of marked points is on an unstable rational component (we recall that unstable rational bridges between components are not permitted because the triple $(0,0,(1,1))$ is not semi-stable).

Now for all semi-stable graphs the only possible unstable vertices are vertices of genus 0 with 2 marked points: a leg and a half-edge. Therefore for all stable graphs $\Gamma$ of type  $(\mathbf{g},\mathbf{n},\mathbf{P})$, there are finitely many semi-stable graphs $\Gamma'$ such that the stabilization of $\Gamma'$ is equal to $\Gamma$. Therefore there are finitely many semi-stable graphs.

Now, if $x$ is a point in $\oH_{\mathbf{g},\mathbf{n},\mathbf{P}}$ then if we denote by $\Gamma$ the dual graph of the underlying curve of $x$ then $x$ lies in $\zeta_\Gamma(\H_{\Gamma})$. This graph is uniquely determined.
\end{proof}

The space $\P\oH_\Gamma$ is a cone, thus it has  a tautological line bundle $\O(1)$. This line bundle is the pullback by $\zeta^{\#}_\Gamma$ of the tautological line bundle of $\P\oH_{\mathbf{g},\mathbf{n},\mathbf{P}}$. By abuse of notation we will write $\xi$ for the first Chern class of the tautological line bundle for both spaces. We have the following important proposition.

\begin{mypr}\label{contract} Let $\Gamma$ be semi-stable graph. The morphism ${\zeta^{\#}_\Gamma}_*: H^*(\P\oH_\Gamma,\Q)\to H^*(\P\oH_{\mathbf{g},\mathbf{n},\mathbf{P}},\Q)$ maps tautological classes to tautological classes.
\end{mypr}

\begin{proof} Let $\Gamma$ be a semi-stable graph. Let $k\geq 0$ and $\beta \in \oM_{\Gamma}^{\rm red}$. We need to prove that the class ${\zeta^{\#}_\Gamma}_*(\xi^{k}p^*(\beta))$ is tautological. We will prove this statement in three steps.

 \subsubsection*{Stable graphs.} We suppose first that $\Gamma$  is a stable graph. We recall that in this case we have defined a map $\zeta_\Gamma:\oM_\Gamma \to \oM_{\mathbf{g},\mathbf{n},\mathbf{m}}$. Then $\oH_\Gamma$ is the fiber product
\begin{center}
$\xymatrix{ 
 \oH_\Gamma\ar[d]_{p_\Gamma} \ar[r]^{\zeta^\#_\Gamma}  &\oH_{\mathbf{g},\mathbf{n},\mathbf{P}} \ar[d]^p \\
\oM_\Gamma \ar[r]_{\zeta_\Gamma}&  \oM_{\mathbf{g},\mathbf{n},\mathbf{m}}
}$
\end{center}
Let $\beta$ be a cohomology class in $H^*(\oM_\Gamma,\Q)$. We use the projection formula and the fact that $\oH_\Gamma$ is a fiber product to get ${\zeta^{\#}_\Gamma}_*(\xi^k \cdot p^*_\Gamma(\beta))=\xi^k p^* ({\zeta_\Gamma}_*(\beta))$. Therefore, if the class $\beta$ belongs to the tautological ring  $RH^*(\oM_\Gamma,\Q)$, then the class ${\zeta^{\#}_\Gamma}_*(\xi^k \cdot p^*_\Gamma(\beta))$ belongs to the tautological ring of $\oH_{\mathbf{g},\mathbf{n},\mathbf{P}}$. 

\subsubsection*{Graph with one main vertex.} Now we no longer assume that $\Gamma$ is stable. Let $1\leq j\leq q$ and $1\leq i\leq m_j$. Let $p_{i}$ be the $i^{\rm th}$ entry of $P_j$. Assume that $\Gamma$ is the following graph
\begin{center}
$\xymatrix@=1.5em{
x_{j,n_j+i} \ar@{-}[d]\\ 
 *+[Fo]{0} \ar@{-}[d] \\
 *+[Fo]{g_j}
}$
\end{center}
(we take the trivial graph for all the other connected components). We will prove that the class ${\zeta_\Gamma^\#}_*(1)$ lies in $RH^*(\P\oH_{\mathbf{g}, \mathbf{n}, \mathbf{m}, \mathbf{P}})$. We use the parametrization of the cone of principal parts at $x$
\begin{equation*}
\left[ \left(\frac{u}{z}\right)^{p_{i}-1}+ a_1 \left(\frac{u}{z}\right)^{p_{i}-2} + \ldots + a_{p_{i}-2}\left(\frac{u}{z}\right)\right] \frac{dz}{z}.
\end{equation*}
The stratum defined by $\Gamma$ is the vanishing locus of $u$. We have seen that $u^{p_{i}-1}$ is a section of the line bundle ${\rm Hom}(\O(-1),\mathcal{L}_i^{p_{i}-1})$. Therefore the vanishing locus of $u$ has Poincar\'e-dual class given by
$$
[{u=0}]=\frac{1}{p_{i}-1}\xi- \psi_i.
$$
By the same argument, if $\Gamma$ is the graph
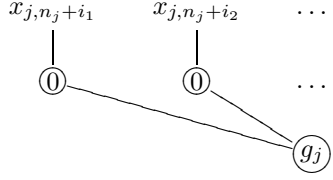
\begin{center}
$\xymatrix@=1.5em{
x_{j,n_j+i_1}\ar@{-}[d]& x_{j,n_j+i_2} \ar@{-}[d]  & \ldots \\ 
*+[Fo]{0}\ar@{-}[rrd]& *+[Fo]{0} \ar@{-}[rd]  & \ldots \\ 
&&*+[Fo]{g_j},
}$
\end{center}
where the set $\{i_k\}$ is a set of indices in $[\![1,m_{j}]\!]$. Then we have  
$${\zeta_\Gamma^\#}_*(1)=\prod_{k} \left(\frac{1}{p_{i_k}-1}\xi- \psi_{{i_k}}\right).$$
And more generally, for a class $\beta$ in $RH^*(\oM_{\mathbf{g}, \mathbf{n}, \mathbf{m}, \mathbf{P}}^{\rm red})$ and $k\in \N$, we have 
$${\zeta_\Gamma^\#}_*(\xi^k\beta)=\xi^k\beta \cdot \prod_{k} \left(\frac{1}{p_{i_k}-1}\xi- \psi_{{i_k}}\right) \in RH^*(\oH_{\mathbf{g}, \mathbf{n}, \mathbf{m}, \mathbf{P}}).$$

\subsubsection*{General unstable graph.} We combine the two previous arguments. Let $\Gamma$ be a general semi-stable graph. Let 
$\widehat{\Gamma}$ be the graph obtained by contracting all edges between stable vertices. We have $\oM^{\rm red}_{\mathbf{g},\mathbf{n},\mathbf{m}}= \oM^{\rm red}_{\widehat{\Gamma}}$ The space $\oH_{\Gamma}$ is the fiber product
\begin{center}
$\xymatrix{ 
 \oH_\Gamma\ar[d]_{p_\Gamma} \ar[r]  &\oH_{\widehat{\Gamma}}^{\rm red} \ar[d]^{p_{\widehat{\Gamma}}}\ar[r]^{\zeta^\#_{\widehat{\Gamma}}}& \oH_{\mathbf{g},\mathbf{n},\mathbf{m}} \\
\oM_\Gamma \ar[r]_{\zeta_\Gamma}& \oM^{\rm red}_{\mathbf{g},\mathbf{n},\mathbf{m}}.}
$
\end{center}
Thus ${\zeta_\Gamma^\#}_*(\xi^kp_{{\Gamma}}^*\beta)=
{\zeta_{\widehat{\Gamma}}^\#}_*
(\xi^k p_{\widehat{\Gamma}}^*({\zeta_\Gamma}_*\beta))$. Now $\widehat{\Gamma}$ has one stable vertex, and ${\zeta_\Gamma}_*\beta\in RH^*(\oM^{\rm red}_{\mathbf{g},\mathbf{n},\mathbf{m}})$ thus the class ${\zeta_\Gamma^\#}_*(\xi^kp_{{\Gamma}}^*\beta)$ is tautological.
\end{proof}

\section{Stratification of spaces of stable differentials}\label{sec:stratification}

The interior of space of stable differentials is stratified according to the  orders of the zeros of the differential. In this section we study the local parametrization of these strata and compute their dimension. 

\subsection{Definitions, notation} 

In the paper we will often consider the following set-up.

\begin{assumption}\label{assumption} The quadruple $(\bg,\bZ,\bP,R)$ is of the following type:
\begin{itemize}
\item $\mathbf{g} = (g_1, \dots, g_q)$, $\mathbf{Z} = (Z_1, \dots, Z_q)$, and $\mathbf{P} = (P_1, \dots, P_q)$  are lists of the same length $q\geq 1$;
\item for all $1\leq j\leq q$, $g_j$ is a positive integer, $Z_{j}$ is a vector of non-negative integers of length $n_j$ and $P_{j}$ is a vector of positive integers of length $m_j$;
\item we denote by $\bn=(n_1,\ldots, n_q)$  and $\bm=(m_1,\ldots, m_q)$;
\item the triple $(\bg,\bn,\bP)$ is semi-stable (in the sense of Definition~\ref{def:stabletrip1})
\item $R$ is a linear subspace of $\oR=\bigoplus_{j=1}^q \oR_j\simeq \bigoplus_{j=1}^q \C^{m_j-1}$ (defined as in~\eqref{eq:defres}).
\end{itemize}
\end{assumption}

 Let $(\bg,\bZ,\bP,R)$ be a quadruple satisfying Assumption~\ref{assumption}.
\begin{mynot}  We denote by 
\begin{equation*}
A_{\mathbf{g},\mathbf{Z},\mathbf{P}}^R \subset \H_{\mathbf{g},\mathbf{n},\mathbf{P}}^R
\end{equation*}
the locus of points $(C,(x_{j,i})_{1\leq j\leq q, 1\leq i\leq n_j+m_j},\alpha)\in \H_{\mathbf{g},\mathbf{n},\mathbf{P}}^R$ such that $C$ is smooth, and $\alpha$ is nonzero on each connected component and has a zero of order exactly $k_{j,i}$ at the $i$th point of the $j$th connected component for all $1\leq j\leq q$ and $1\leq i\leq n_j$. 

If there is no condition on the residues we will simply denote it by $A_{\mathbf{g},\mathbf{Z},\mathbf{P}}$.
\end{mynot}

\begin{mydef}
We say that $\mathbf{Z}$ is {\em complete} for $(\bg,\bP)$  if $Z_j$ is complete for $(g_j, P_j)$ for all $1\leq j\leq q$. 
\end{mydef}

\subsection{Standard coordinates}

In this section we describe how to parametrize differentials with prescribed singularities. We use the notation $\Delta_\rho=\{z\in \C: |z|<\rho\}$ for the disks of radius $\rho\in \R^{+}$ and $A_{\rho_1,\rho_2}=\{z\in \C: \rho_1<|z|<\rho_2\}$ for the annulus of parameters $0<\rho_1<\rho_2$.

\subsubsection{Standard coordinates}

 Let $\alpha$ be a meromorphic differential on a small disk $\Delta_\rho\subset \C$. We denote by $r$ the residue of $\alpha$ at 0. Then, there exists a conformal map $\varphi:\Delta_{\rho'}\to \Delta_\rho$ for $\rho'$ small enough, such that: $\varphi(0)=0$ and
$$
\varphi^*(\alpha)=\left\{ \begin{array}{l r} d(z^k)  &  \text{  if 0 is a zero of order $k-1$;} \\ r\frac{dz}{z} &\text{  if 0 is a pole of order 1;}\\ d(\frac{1}{z^k}) + r\frac{dz}{z} &\text{ if 0 is a pole of order $k+1$.} \end{array}\right. 
$$
The map $\varphi$ is unique up to multiplication of the coordinate $z$ by a $k$-th root of unity when $0$ is a zero of order $k-1$ or a pole of order $k+1$. The coordinate $z$ will be called the {\em standard coordinate}. 

More generally, if $U$ is an open neighborhood of $0$ in $\C^n$ and $\alpha_u$ is a holomorphic family of differentials on $\Delta_{\rho}$ such that  the order of $\alpha_u$ at $0$ is constant, then there exists a holomorphic map $\varphi: \widetilde{U} \times \Delta_{\rho'} \to \Delta_{\rho}$ such that $\varphi(u,\cdot)^*(\alpha_u)$ is in the standard form for some neighborhood of $0$, $\widetilde{U}$. Once again the map $\varphi$ is unique up to multiplication of the standard coordinate by a root of unity.

Now the following classical lemma describes the deformations of $d(z^k)$ (see~\cite{KonZor} for a proof):

\begin{mylem}\label{lem:univdiff}
Let $\rho>0$ and $U \subset \C^n$ be a domain containing 0.  Let $\alpha_u$ be a family of holomorphic differentials on $\Delta_\rho$ such that $\alpha_0$ has a zero of order $k-1$ at the origin. Then, there exists $\rho'>0$, a neighborhood of $0$ in $\C^{k-2}$, $\oZ$ and a conformal map
$$
\varphi: U \times \Delta_{\rho'} \to \Delta_{\rho}\times \oZ
$$
such that that $\varphi(u,\cdot)^*(\alpha_u)=d(z^{k}+a_{k-2}z^{k-2}\ldots+a_1z)$. The map $\varphi$ is unique up to multiplication of $z$ by a $k$-th root of unity.
\end{mylem}

The locus $z=0$ determines a section of the projection $U\times \Delta_\rho$ that does not depend on the choice of $k$-th root of unity. This section is called the {\em local center of mass of zeros}. 

Now we would like to generalize the above lemma to deformations of poles of order~$1$.

\begin{mydef}\label{def:standdef} Let $\rho>0$ and $U \subset \C^n$ be a domain containing 0. Let $\alpha$ be a differential on $\Delta_{\rho}$ in the standard form $d(z^k)$. A {\em standard deformation} of $\alpha$ is defined by a holomorphic function $\beta:U\times \Delta_\rho\to \C$ satisfying $\beta(0,z)=0$. A standard deformation associated to $\beta$ is the family of differentials on $\Delta_{\rho}$ parametrized by $U$
$$
\alpha_u=d(z^k)+\frac{\beta(u,z)}{z}dz.
$$
\end{mydef}

In general, there exists no standard coordinate for a standard deformation. However, the following proposition has been proved in~\cite{BCGGM} (see Theorem 4.3).
\begin{mypr}\label{pr:annulus}  We consider the annulus $A_{\rho_1,\rho_2}$ for any choice of $0<\rho_1<\rho_2<\rho$.

Chose a point $p\in A_{\rho_1,\rho_2}$ and $\zeta^\ell=\exp(\frac{2i\pi\ell}{k})$ a $k$-th root of unity. Chose a map $\sigma:U\to \Delta_\rho$ such that $\sigma(0)=\zeta^\ell p$. Then there exists a neighborhood $\widetilde{U}$ of 0 in $U$ and a holomorphic map $\varphi:\widetilde{U}\times A_{\rho_1,\rho_2}  \to \Delta_R$ such that 
$$\varphi_u^*(\alpha_u)=d(z^k)+\frac{\beta(u,0)}{z} dz,$$
and $\varphi(0,z)=\zeta^\ell z$ and $\varphi(u,p)=\sigma(u)$ for all $u\in \widetilde{U}$ and $z\in A_{\rho_1,\rho_2}$. For $\widetilde{U}$ small enough, the map $\varphi$ is unique.
\end{mypr}

\subsubsection{Neighborhood of strata}

Let $(\bg,\bZ,\bP,R)$ be a quadruple satisfying Assumption~\ref{assumption}. 

\begin{mylem}\label{deformation}
 There exists a neighborhood $V$ of $A_{\bg,\bZ,\bP}$ in $\H_{\bg,\bn,\bP}$ and a holomorphic retraction $\eta:V\to A_{\bg,\bZ,\bP}$ such that $\eta$ preserves the residues at the poles.
\end{mylem}

\begin{proof} The general statement follows immediately from the connected case. Indeed, $A_{\bg,\bZ,\bP}$ is locally isomorphic to $\prod_{j=1}^q A_{g_j,Z_j,P_j}$ therefore we can define the neighborhood $V$ and the retraction $\eta$ as the product of the $V_j$ and $\eta_j$ for all $1\leq j\leq q$. Therefore we will assume that $q=1$.

 Let $y_0=(C_0, x_1,\ldots,x_{n+m}, \alpha_0)$ be a point in $A_{g,Z,P}$. Let $n'$ be the number of zeros of $\alpha$ distinct from the marked points. We chose an ordering of these zeros $(\widetilde{x}_{1}, \ldots, \widetilde{x}_{n'})$ and we denote by $\widetilde{k}_i$  the order of $\alpha$ at $\widetilde{x}_{i}$ for all $1\leq i\leq n'$.

We denote by $d=\dim(A_{g,Z,P})$ and by $d'=\dim(\H_{g,n,P})$. A neighborhood of $y_0$ in $A_{g,Z,P}$ is of the form $U/ {\rm Aut}(y_0)$ where $U$ is a contractible domain of $\C^d$. A neighborhood of $U/ {\rm Aut}(y_0)$ in $\H_{g,n,P}$ is of the form $W/{\rm Aut}(y_0)$ where $W$ is a contractible domain of $\C^{d'}$.

For all $y=(C,\alpha,(x_{j,i}))$ in $U$ we denote by $P(y)\subset C$ the set of poles of $\alpha$ and by $Z(y)$ the set of zeros (marked or not).  For all $y$, the form $\alpha$ determines a class in the relative cohomology group  $H^1(C \setminus P(y), Z(y), \C)$.  Besides, we have a canonical identification of $H^1(C \setminus P(y), Z(y), \C)$ with $H^1(C_0 \setminus P(y_0), Z(y_0), \C)$ (this is the Gauss-Manin connection), therefore we have a holomorphic map
$$
\Phi_U: U\to H^1\left(C_0 \setminus P(y_0), Z(y_0), \C\right).
$$
This map can be described as follows. Let $(\gamma_1,\ldots, \gamma_d)$ be simple closed curves of $C_0\setminus (P(y_0)\bigcup Z(y_0))$  that form a basis of the relative homology group $H_1(C_0 \setminus P(y_0), Z(y_0), \Z)$. Then the map $\Phi_U$ is defined by
\begin{eqnarray*}
\Phi_U : U &\to&  H^1(C_0 \setminus P(y_0), Z(y_0), \C)\\
(C,\alpha, (x_i)) &\mapsto& \left ( \gamma \mapsto \int_{\gamma_i} \alpha \right)
\end{eqnarray*}
Where the cycles on $C_0\setminus (P(y_0)\bigcup Z(y_0))$ are identified with cycles on $C\setminus (P(y)\bigcup Z(y))$ by the Gauss-Manin connection. The map $\Phi_U$ is a local bi-holomorphism (see~\cite{Boi} for example). We call the map $\Phi_U$ a {\em period coordinates} chart. 

Now we will construct the following holomorphic maps
\begin{eqnarray*}
\Phi^1: W &\to& H^1(C_0 \setminus P(y_0), Z(y_0), \C),\\
\Phi^{2,i} : W &\to& \oZ^{k_i} \text{ for all $1\leq i \leq n$,}  \\ 
\Phi^{3,i} : W &\to& \widetilde{\oZ}^{\widetilde{k}_i} \text{ for all $1\leq i \leq n'$,}   
\end{eqnarray*}
where $\oZ_i$ is a domain of $\C^{k_i}$ containing of 0 for all $1\leq i\leq n$ and $\widetilde{\oZ}_i$ is a domain of $\C^{\widetilde{k}_i-1}$ containing of 0 for all $1\leq i\leq n'$.
\begin{itemize}
\item For all $1\leq i \leq n$, the map $\Phi^{2,i}$ is determined by a slight modification of Lemma~\ref{lem:univdiff} for marked differentials. We consider a tubular neighborhood $W\times \Delta_\rho\to C_W$ around the $i$-th section of the universal curve. There exists a $\rho'>0$ and a  neighborhood $\oZ_i$ of $0 \in \C^{k_i}$ with coordinates $(a_{i,1},\ldots,a_{i,k_i})$ and a map $\varphi: W\times \Delta_\rho \to \Delta_{\rho'}\times \oZ_i$ such that the marked point is at $z_i=0$ and
$$
\alpha_s= d(z_i^{k_i+1}+ a_{i,k_i} z_i^{k_i}+\ldots + a_{i,1} z_i)
$$
for each point $s$ of $W$. The map $\varphi$ is unique up to a multiplication of $z_i$ by a $(k_i+1)$-st root of unity. Thus we have defined a map from $W$ to $\oZ_i$ given by $\alpha_s\mapsto (a_{i,1},\ldots, a_{i,k_i)})$. 
\item For all $1\leq i \leq n'$, the map $\Phi^{2,i}$ is determined by Lemma~\ref{lem:univdiff}. We consider a tubular neighborhood $W\times \Delta_\rho\to C_W$ around the $i$-th section of the universal curve. There exists a $\rho'>0$ and a  neighborhood $\widetilde{\oZ}_i$ of $0 \in \C^{\widetilde{k}_i-1}$ with coordinates $(a_{i,1},\ldots,a_{i,\widetilde{k}_i-1})$ and a map $\varphi: W\times \Delta_\rho \to \Delta_{\rho'}\times \oZ_i$ such that 
$$
\alpha_y= d(z_i^{\widetilde{k}_i+1}+ \ldots + a_{i,1} z_i)
$$
for each point $y$ of $W$. The map $\varphi$ is once again unique up to a multiplication of $z_i$ by a $(\widetilde{k}_i+1)$-st root of unity. Thus we have defined a map from $W$ to $\oZ_i$ given by $\alpha_y\mapsto (a_{i,1},\ldots, a_{i,\widetilde{k}_i-1)})$. 

Besides, the point $z_i=0$ is called the center of mass of the differential. It does not depend on the choice of a root of unity, therefore we have a uniquely determined point $\widetilde{x}_i \in C$ for all $s$.
\item The map $\Phi^1$ is defined as $\Phi_U$ by the Gauss-Manin connection. For a point $y=(C,\alpha,x_1,\ldots, x_{n+m}) $ in $W$ we denote by $Z(y)=\{x_1,\ldots,x_n\} \cup \{\widetilde{x}_1,\ldots, \widetilde{x}_{n'}\}$ (the union of the marked points with the center of masses defined above).  Then the differential $\alpha$ defines a point in $H^1(C\setminus P(y), Z(y))$ which is once again canonically identified with $H^1(C_0 \setminus P(y_0), Z_(y_0))$.
\end{itemize}
We will prove that the map 
$$\Phi= \Phi^1 \times \left( \prod_{i=1}^n \Phi^{2,i} \right) \times \left( \prod_{i=1}^{n'} \Phi^{3,i}\right)
$$
 is a local bi-holomorphism (see~\cite{KonZor} 5.2, in the holomorphic case). The source and the target have the same dimension therefore we only need to check that that the differential of each component of $\Phi$ is surjective. For $\Phi^1$ this is obvious because $\Phi^1|_U = \Phi_U$ is a local bi-holomorphism. 

Let $1\leq i\leq n$ and let $\Delta_\rho$ be a disk in $C_0$ around $x_i$ such that $\alpha=d(w^{k_i+1})$. Up to a choice of a smaller $\oZ_i$, for all $(a_{i,1},\ldots, a_{i,k_i-1})\in \oZ_i$ we have $(z^{k_i+1}+ \ldots +a_{i,1} z)\neq 0$
 for all $\rho/2 <|z|<\rho$. Then we construct a family of curves $\mathcal{C}_i\to \oZ_i$ by gluing the two families of curves $\left(C_0\setminus \Delta_{\rho/2}\right)\times \oZ_i$ with $\Delta_\rho\times \oZ_i$ along the identification
$$
w=(z^{k_i+1}+ \ldots +a_{i,1} z)^{\frac{1}{({k_i}+1)}}
$$
(this family depends on the choice of a the $(k_i+1)$-st root). Now the differential $\alpha$ on $\mathcal{C}_i$ is determined by $\alpha_0$ on $\left(C_0\setminus \Delta_{\rho/2}\right)\times \oZ_i$ and by $d(z^{k_i+1}+ \ldots +a_{i,1} z)$ on $\Delta_\rho\times \oZ_i$. The two differentials agree by construction of the complex structure. Therefore the differential of $\Phi^{2,i}$ is surjective. The same argument holds for $\Phi^{3,i}$ for $1\leq i\leq n'$.

Now we set $\eta_W=\Phi_U^{-1} \circ  \Phi^1$. This retraction does not depend on the choice of the root of unity nor on the choice of ordering of the non-marked zeros. Indeed, it is defined by the inverse procedure of patching $d(w^{k_i+1})$ instead of $d(z^{k_i+1}+ \ldots +a_{i,1} z)$ for all $1\leq i\leq n$ (and for non-marked zeros). Therefore it does   not depend on the local identification of the relative homology group. Thus if we consider two maps $\eta_{W}$ and $\eta_{W'}$ (for neighborhoods of points $y_0$ and $y_0'$) then these two map  agree on $W\cap W'$.

Finally, the residues are preserved by $\eta$. Indeed for any choice of $y_0$, we can chose a basis $(\gamma_1,\ldots, \gamma_{d})$ of $H_1(C_0\setminus P(y_0), Z(y_0), \Z)$ such that   $\gamma_i$ is a small loop around the $(n+i)^{th}$ marked point for all $1\leq i\leq m-1$. The period of $\alpha$ around this loop is the residue of $\alpha$ at the $i$-th pole and is preserved by $\eta$.
\end{proof}

\begin{mycor}\label{resmap} The residue map restricted to $A^R_{\bg,\bZ,\bP}\to R$ is a submersion.
\end{mycor}

\begin{proof}
 Let $(C,x_1,\ldots,x_{n+m},\alpha)$ be a point of $A_{\bg,\bZ,\bP}^R$. Let $\mathbf{r}=(r_1,\ldots,r_m)$ be a vector in $R$. There exists a meromorphic differential $\varphi$ on $C$ with  at most simple poles at the $m$ last marked points with residues prescribed by $\mathbf{r}$. Let $\Delta$ be a disk of $\C$ centered at $0$ and parametrized by $\epsilon$. Let $\eta$ be the retraction map of Lemma~\ref{deformation}. The residues of $\eta(\alpha+\epsilon \varphi)$ at the poles are given by 
\begin{equation*}
{\rm{res}}_{x_{n+i}}(\alpha)+ \epsilon r_i.
\end{equation*}
Thus the vector $\mathbf{r}$ belongs to the image of the tangent space of $A^R_{\bg,\bZ,\bP}$ under the differential of the map $\rm{res}$. 
\end{proof}

\begin{remark}
Recenlty Gendron and Tahar studied the surjectivity of the residue maps for open strata in the space of meromorphic differentials (and also of higher order differentials -- see~\cite{GenTah}). Our statement that the residue map is a submersion does not imply surjectivity. However, the image of an algebraic submersion is always a Zarisky open set. Thus we can claim that the residue map is surjective on the {\em closure} of every nonempty stratum.
\end{remark}

\subsubsection{Neighborhood of strata with appearance of residues}

We consider a slightly more general set-up. Let $q\geq 2$ and $\bg,\bn,\bn',\bm$ be list of non-negative integers of length $q$.  Let $\bP=(P_1,\ldots,P_q)$ be a list of vectors of positive integers such that ${\rm length}(P_j)=m_j$ for all $1\leq j\leq q$ and  let $Z=(Z_1,\ldots, Z_q)$ be a list of vectors of nonnegative integers such that ${\rm length}(Z_j)=n_j+n_j'$. We assume that the triple $(\bg,\bn+\bn',\bP)$ is semi-stable (in the sense of Definition~\ref{def:stabletrip1}).

For all $1\leq j\leq q$, we denote by $P_j'=(p_1,\ldots,p_{m_j}, 1,\ldots, 1)$ the vector obtained from $p$ by adding $n_j'$ times $1$ and by $Z_j'=(k_1,\ldots,k_n)$ the vector obtained by erasing the last $n'$ entries of $Z$. 

The space $\H_{\bg,\bn+\bn',\bP}$ is embedded in $\H_{\bg,\bn,\bP'}$. We denote by $\oR$ and $\oR'$ the vector spaces of residues of $\H_{\bg,\bn+\bn',\bP}$  and $\H_{\bg,\bn,\bP'}$.  Let $R'$ be a vector subspace of $\oR'$. The vector space $\oR$ is a vector subspace of $\oR'$, and we denote by $R=\oR\cap R'$. We have the following series of embeddings
$$
A_{\bg,\bZ,\bP}^{R}\hookrightarrow A_{\bg,\bZ',\bP}^{R}\hookrightarrow A_{\bg,\bZ',\bP'}^{R'}.
$$

\begin{mypr}\label{pr:standardneihbor}
Let $y_0$ be a point in $A_{\bg,\bZ,\bP}^{R}$. Let $U$ be neighborhood of $y_0$ in $A_{\bg,\bZ,\bP}^{R}$. There exists a neighborhood $V$ of $y_0$ in $A_{\bg,\bZ',\bP'}^{R'}$ and a map
$$\phi: V  \overset{\sim}{\to} U \times \left(\prod_{\begin{smallmatrix} 1\leq j\leq q\\1\leq n_j'\end{smallmatrix}} \oZ_{j,i} \right)\times \oZ$$
where:
\begin{itemize}
\item $\oZ_{j,i}$ is a neighborhood of $0$ in $\C^{k_{j,n_j+i}}$ for all $1\leq j\leq q, 1\leq i\leq n'$ and $\oZ$ is a neighborhood of $0$ in $R'/R$;
\item if $\Delta_\rho$ is a disk and $s:U\times \Delta_\rho \to (\prod \oZ_{j,i})\times \oZ$ is a holomorphic map such that $s(u,0)=0$ then the family of differentials 
\begin{eqnarray*}
\tilde{s}:U\times \Delta_\rho &\to& V\\
(u,\epsilon) &\mapsto & \phi^{-1}(u,s(u,\epsilon))
\end{eqnarray*}  is a standard deformation of $d(z^{k_{j,n_j+i}+1})$ for all $1\leq j\leq q,1\leq i\leq n'$.
\end{itemize}
\end{mypr}

\begin{proof}
We have seen that a neighborhood of  $U$ in $A_{\bg,\bZ',\bP}^R$ is isomorphic to $U\times \prod_{j=1}^q \prod_{i=1}^{n_j+n_j'} \oZ_{j,i}$. For all $1\leq j \leq q$, and $1\leq i\leq n_j'$, the differential at the marked point $x_{j,n_j+i}$ is given by $d(z^{k_{n_j+i}} +a_1z^{k_{j,n_j+i}}+\ldots)$ (Lemma~\ref{lem:univdiff}). 

Now, for all $1\leq j\leq q$ and $1\leq i\leq n_j$,  we choose a meromorphic differential $\varphi_{j,i}$ with simple poles at the marked points in such a way that the vectors of residues $\mathbf{r}_{j,i}$ of $\varphi_{j,i}$ form a basis of $R'/R$. The residue map $A_{\bg,\bZ',\bP'}^{R'}\to R'$ is a submersion (Corollary~\ref{resmap}). Thus a neighborhood of $U\times \prod \oZ_{j,i}$ in $A_{\bg,\bZ',\bP'}^{R'}$ is naturally identified with a $U\times (\prod \oZ_{j,i})\times \oZ$ with $\oZ$ neighborhood of $0$ in $R'/R$. The identification is given by adding a linear combination of the $\varphi_{j,i}$'s. 

Both the  deformations of $U$ into $U\times \prod \oZ_{j,i}$ and the deformations of $U\times \prod \oZ_{j,i}$ into  $U\times (\prod \oZ_{j,i})\times \oZ$ are standard deformations at the marked point $x_{j,n_j+i}$  for all $1\leq j\leq q$ and $1\leq i\leq n_j$.
\end{proof}
The isomorphism $\phi$ is not unique. Our construction depends on the choice of standard coordinates at the $x_{j,n_j+i}$  for all $1\leq j\leq q$ and $1\leq i\leq n_j$ and on the choice of the differentials $\varphi_{j,i}$ with simple poles. However Proposition~\ref{pr:standardneihbor} implies the following corollary.

\begin{mycor}\label{pr:retraction}  Given $\phi$ satisfying the conditions of Proposition~\ref{pr:standardneihbor}. The morphism $\phi$ defines  a local retraction $\eta: V\to U$ such that $\eta \circ \tilde{s}={\rm Id}_U$ for any holomorphic section $s:U\times \Delta_\rho \to (\prod \oZ_{j,i})\times \oZ$.
\end{mycor}

\subsection{Dimension of the strata}\label{ssec:dim} Let $(\bg,\bZ,\bP,R)$ be quadruple satisfying Assumption~\ref{assumption}.

\begin{mydef}
A {\em completion} of $\bZ$ is a list of $q$ vectors of non-negative integers $Z_1'=(k_{1,1}',\ldots,k_{1,n_1'}'),\ldots, Z_q'=(k_{1,1}',\ldots,k_{1,n_q'}')$ such that:
\begin{itemize}
\item for all $1\leq j\leq q$,  $n_j'\geq n_j$;
\item  for all $1\leq j\leq q$ and $1\leq i\leq n$, we have $k_i'\geq k_i$. 
\item $\bZ'$ is complete for $(\bg,\bP)$.
\end{itemize}
We will say that the completion $\bZ'$ is {\em exterior} if for all $j$ and all $1\leq i\leq n_j$ we have $k_i'=k_i$. Finally we will denote by $\bZ_m$ the {\em maximal completion}, i.e. the exterior completion of $Z$ that satisfies $k_{j,i}'=1$ for all $j$ and $n_j+1\leq i\leq n_j'$. 
\end{mydef}

If $\bZ'$ is a completion of $\bZ$ we denote by $\pi: A_{\bg,\bZ',\bP}^R\to A_{\bg,\bZ,\bP}^R$ the forgetful map of marked point that are not accounted for by $\bZ$, i.e. the restriction 
of the forgetful map of marked points $\pi:\H_{\bg,\bn',\bP} \to \H_{\bg,\bn,\bP}$ to $A_{\bg,\bZ',\bP}^R$. We have the following straightforward lemma.

\begin{mylem}
We have
$$A_{\bg,\bZ,\bP}^R=\bigcup_{\bZ'}\pi(A_{\bg,\bZ',\bP}^R),$$
where the union is over all exterior completions of $Z$.
\end{mylem}

\begin{mylem}\label{lem:isoline}
If $q=1$ and the vector $Z$ is complete for $g$ and $P$, then the forgetful map of the differential $p:A^R_{g,Z,P}\to p(A^R_{g,Z,P})\subset \M_{g,n+m}$ is a line bundle minus the zero section. In particular $\P A^R_{g,Z,P}$ is isomorphic to its image.
\end{mylem}

\begin{proof}
Let $(C,x_1,\ldots,x_{n+m})$ be a point of ${\rm Im}(p)$. The curve $C$ is smooth and the divisor $\omega_C-\sum_{i=1}^n k_i (x_i)+\sum_{j=1}^m p_j (x_{n+j})$ is a principal divisor of degree 0. Therefore the fiber of $p$ over $(C,x_1,\ldots,x_{n+m})$ is given by the nonzero multiples of one differential with fixed orders of zeros and poles.
\end{proof}

\begin{mypr}
The space $A^R_{\bg,\bZ,\bP}$ is either empty or co-dimension 
$\sum_{j=1}^q |Z_j| +\dim(\oR/R)$ in $\H_{\bg,\bn,\bP}$.
\end{mypr}

\begin{proof}
First we assume that $q=1$ (connected case), $Z$ is complete and $R=\oR$ (no residue condition). The dimension of $\P A_{g,Z,P}$ is  equal to the dimension of its image in the moduli space of curves.  Then the image of $\P A_{g,Z,P}$ is of dimension $2g-2+n$ if $P$ is empty (see~\cite{Pol}) and $2g-3+n+m$ otherwise (see~\cite{FarPan}). By a simple count of dimension we can check that the proposition is valid in this specific case.

We no longer assume that $q=1$ (but we still assume that $\bZ$ is complete and $R=\oR$). Then the space $A_{\bg,\bZ,\bP}$ is birationally equivalent to $\prod_j A_{g_j,Z_j,P_j}$. Thus $\dim(A_{\bg,\bZ,\bP})= \sum \dim(A_{g_j,Z_j,P_j})$ and once again, the Proposition holds by a simple count of dimensions 

Now, we still assume that $\bZ$ is complete, however we no longer assume that $R=\oR$.  We have seen that the residue map $A^R_{\bg,\bZ,\bP}\to R$ is a submersion, therefore the dimension of $A^R_{\bg,\bZ,\bP}$ is equal to the dimension of $R$ plus the dimension of the fiber of the residue map at any point.  If we consider the case $R=\oR$, then we see that that dimension of the fiber at any point is $\dim  A_{\bg,\bZ,\bP} -\dim{\oR}$. Therefore the dimension of $A^R_{\bg,\bZ,\bP}$ is equal to $\dim  A_{\bg,\bZ,\bP}-(m-1)+ \dim(R)$. Thus the proposition is valid for all choices of $R$.

Now, let $\bZ$ be any vector. Let $\bZ'$ be an exterior completion of $Z$. The map $\pi:A^R_{\bg,\bZ',\bP} \to A^R_{\bg,\bZ,\bP}$ is quasi-finite. Indeed the preimage of a point $(C,x_1,\ldots,x_{n+m},\alpha)$ is finite: the points in the preimage correspond to the different orderings of the zeros that are not accounted for by $\bZ$. 

The proof of Lemma~\ref{deformation} implies that if $A_{\bg,\bZ',\bP}$ is not empty for some exterior completion then $A_{\bg,\bZ_m,\bP}$ is not empty: indeed we can always perturb a differential to ``break up'' a zero of order greater than 1.  By counting the dimensions, we have $\dim(A^R_{\bg,\bZ_m,\bP})>\dim(A^R_{\bg,\bZ',\bP})$ for all exterior completions $\bZ'\neq \bZ_m$. Therefore $\dim(A^R_{\bg,\bZ_m,\bP})=\dim(A^R_{\bg,\bZ,\bP})$ and the proposition is proved.
\end{proof}

\subsection{Fibers of the map $p:A_{\mathbf{g},\mathbf{Z},\mathbf{P}}^R \to \M_{\mathbf{g},\mathbf{n},\mathbf{m}}$}\label{ssec:fibers}

Let $(\mathbf{g},  \mathbf{Z},\mathbf{P}, R\subset \oR)$ be a quadruple satisfying Assumption~\ref{assumption}. Besides in all this section we assume that the triple $(\bg,\bn,\bP)$ is \underline{stable}.

If the context is clear, we denote by the same letter the map $p:\H_{\mathbf{g},\mathbf{n},\mathbf{P}}\to \M_{\mathbf{g},\mathbf{n},\mathbf{m}}$ and its restriction $p: A_{\mathbf{g},\mathbf{Z},\mathbf{P}}^R \to p (A_{\mathbf{g},\mathbf{Z},\mathbf{P}}^R)$. We denote by ${\rm Im}(p)= p (A_{\mathbf{g},\mathbf{Z},\mathbf{P}}^R)\subset \M_{\mathbf{g},\mathbf{Z},\mathbf{P}}$ its image.  

We recall that by definition $\oR=\bigoplus_{j=1}^q \oR_j \simeq \bigoplus_{j=1}^q \C^{\bm_j-1}$ (see Section~\ref{ssec:disconnected}). 

\begin{mynot}
Let $1\leq j\leq q$, we denote by ${\rm pr}_j: \oR\to \oR_j$ the projection onto $\oR_j$ along $\bigoplus_{j'\neq j} \oR_{j'}$. We denote by $R_j$ the space ${\rm pr}_j(R)$. 
\end{mynot}

\begin{remark} The linear relations that define the space $R$ may involve residues at poles of different connected components. Thus in general we have $R\cap \oR_j \subsetneq R_j$. 
\end{remark}

Let $1\leq j\leq q$. We  denote by $p_j$ the map from $A_{g_j,Z_j,P_j}^{R_j}$ to $\M_{g_j,n_j+m_j}$. Finally we denote by ${\rm Im}(p_j)$ the image of $p_j$. We have a natural embedding of $A_{\mathbf{g},\mathbf{Z},\mathbf{P}}^R$ into $\prod_{j=1}^q A_{g_j,Z_j,P_j}^{R_j}$ and of ${\rm Im}(p)$ into $\prod_{i=1}^n {\rm Im}(p_j)$. 

The purpose of this section is to state the condition $(\star\star)$ (see Notation~\ref{not:starstar}) that ensures that the projectivized morphism $p:\P A_{\mathbf{g},\mathbf{Z},\mathbf{P}}^R\to {\rm Im}(p)$ is birational. This will be needed in Section~\ref{ssec:div} to describe the boundary divisors of the stratum $A_{\mathbf{g},\mathbf{Z},\mathbf{P}}^R$. We will proceed in two steps: first we consider the case that $\bZ$ is complete and then a general $\bZ$.

\subsubsection*{Complete case} For now we assume that $\bZ$ is complete for $\bg$ and $\bP$. 

We have seen that the fact that $Z_j$ is complete for all $1\leq j\leq q$ implies that that $A_{g_j,Z_j,P_j}^{R_j}\to {\rm Im}(p_j)$ is a line bundle minus the zero section. We denote by $L_j$ the pull-back of this line bundle to ${\rm Im}(p)$. 

We define the {\em $j$-th  evaluation map of residues}  ${\rm ev}_j:L_j\to \oR_j$ as the morphism of vector bundles over  ${\rm Im}(p)$ given by the evaluation of the residues at the $j$-th connected component. We define  the {\em evaluation of residues} as the morphism of vector bundles: ${\rm ev}=\left(\bigoplus^q_{j=1} {\rm ev}_j\right):\bigoplus^q_{j=1} L_j\to \oR$.

\begin{remark}
The evaluation map (${\rm ev}$) and the residue map (${\rm res}$) are not defined on the same spaces. The first one is a morphism of vector bundles on the space ${\rm Im}(p)$ while the second one is defined as a morphism of vector bundles over $\P A_{\mathbf{g},\mathbf{Z},\mathbf{P}}^R$. If $q=1$, then $\P A_{\mathbf{g},\mathbf{Z},\mathbf{P}}^R$ is isomorphic to its image and the two morphisms are equal.
\end{remark}

\begin{mypr}\label{pr:fiber2}
Suppose that $\mathbf{Z}$ is complete. Then, the families
$$
p : A_{\mathbf{g},\mathbf{Z},\mathbf{P}}^R \to {\rm Im} (p)
$$ 
and 
$$
\widetilde{p} : {\rm ev}^{-1}(R) \cap \left(\prod_{j=1}^q   L^*_j \right) \to {\rm Im}(p)
$$ 
are isomorphic. If $q\geq 2$, the fiber of $p$ over a point is of dimension $1$ if and only if ${\rm ev}$ is injective and $R\cap {\rm ev}(\bigoplus_j L_j)$ is of dimension $1$.
\end{mypr}

\begin{proof} The proposition is straightforward for $q=1$. We suppose  from now on that $q\geq 2$. 

For a point $x \in {\rm Im}(p)$, the fiber of $p$ can be described as follows: it is the choice of a nonzero differential for each connected component such that the residues at the poles define a vector in $R$. Therefore the fiber over $x$  is the subset of points of $\prod L_j^*$ with residues in $R$. This fiber is given by ${\rm ev}^{-1}(R)\cap \prod_{j=1}^q L_i^*$.  

The fiber of ${\rm ev}^{-1}(R) \cap \prod_{j=1}^q L_j^*$ over $x\in {\rm Im}(p)$ is not empty. Indeed, suppose that for some $1\leq j\leq q$ the space ${\rm ev}^{-1}(R)$ is contained in $\{0\} \times \bigoplus_{j'\neq j} L_{j'}$, then the residue condition $R$ imposes that the differential on one of the component is zero. In which case, $x$ is not a point of ${\rm Im}(p)$. Therefore the dimension of  ${\rm ev}^{-1}(R) \cap \prod_{j=1}^q L_j^*$ is the same as the dimension of ${\rm ev}^{-1}(R)\cap \bigoplus_{j=1}^q L_j$.

The only point that remains to prove is: if the map ${\rm ev}$ is not injective then the fiber of $p$ is of dimension greater than 1. We assume that the map ${\rm ev}$ is not injective. 
Then one of the $L_j$'s is mapped to zero for some $1\leq j\leq q$: indeed for all $1\leq j\leq q$, the $j$-th component of ${\rm ev}$ is the composition of ${\rm ev}_j: L_j \to \oR_j$ 
with the inclusion of $\oR_j\to \oR$;  thus if a vector in $\bigoplus L_j$ with a non-zero $j$-th entry is mapped to zero in $\oR$ then the  generator of $L_j$ is mapped to zero in $\oR_j$ and $L_j$ is mapped to zero in $\oR$. 

Therefore we have
$$ {\rm ev}^{-1}(R)\cap \bigoplus_{j=1}^q L_j=L_j \oplus \left({\rm ev}^{-1}(R)\cap \bigoplus_{j'\neq j} L_{j'}\right). $$
We have seen that ${\rm ev}^{-1}(R)$ cannot be contained in $L_j\times \{0\}$, thus the second summands is of positive dimension and ${\rm ev}^{-1}(R)\cap \bigoplus_{j=1}^q L_j$ is of dimension greater than $1$.
\end{proof}

Let $\Sigma$ be the union of the  vector subspaces $R\cap \ker({\rm pr}_j)$ for $1\leq i\leq q$. If $R$ is of positive dimension, we denote by $\P \Sigma$ the image of $\Sigma$ in $\P R$. This is the locus of vectors of residues that vanish on at least one connected component. Suppose that all $R_j$ are of positive dimension, then $\Sigma\subsetneq R$ and there is a natural map $\rho:\P R\setminus \P \Sigma \to \prod_{j=1}^q \P R_j$ defined as the projection on each factor.

\begin{mynot}
We will say that the residue vector spaces $(\oR, R,(\oR_j)_{1\leq i\leq q})$ satisfy the condition $(\star)$ if either $q=1$  or the two following conditions holds:
\begin{itemize}
\item the space $R$ and the $R_j$'s are of positive dimension;
\item there exists an open and dense set $U$ in $\P R$ such that the restriction of  the natural map $\rho:\P R\setminus \P \Sigma\to \prod_{i=1}^q \P R_j$ to $U$ is finite.
\end{itemize} 
\end{mynot}

\begin{mypr}\label{pr:fiber3}
Suppose that $\mathbf{Z}$ is complete and that $q$ is at least $2$.  Then the fiber of $p$ over a generic point of ${\rm Im}(p)$ is of dimension $1$ if and only if $(\oR, R,(\oR_j)_{1\leq j\leq q})$ satisfy the condition $(\star)$.
\end{mypr}

\begin{proof}
We have already seen that if $R_j$ is reduced to the trivial space, then the map ${\rm ev}:\bigcup_{j=1}^q L_j\to \oR$ is not injective and the fibers of $p$ are all of dimension greater than 1 (see the proof of Proposition~\ref{pr:fiber2}). We assume that all $R_j$ are non trivial. For all $j$, we denote by $A_j^0\subset  A_{\mathbf{g},\mathbf{Z},\mathbf{P}}^R$ to be the locus of differentials with zero residues on the $j^{\rm th}$ component. The image of $A_j^0$ by the residue map lies in $R\cap \ker({\rm pr}_j)$ which is of positive codimension in $R$. Besides the residue map is a submersion, thus $\dim(A_j^0)<\dim(A_{\mathbf{g},\mathbf{Z},\mathbf{P}}^R)$. We will denote 
$$A'= A_{\mathbf{g},\mathbf{Z},\mathbf{P}}^R \setminus \bigcup_{j=1}^q A_j^0.$$
The locus $A'$ is dense in $A_{\mathbf{g},\mathbf{Z},\mathbf{P}}^R$. If we assume that the fibers of $p$ are generically of dimension 1, then $p(A')$ is also dense in ${\rm Im}(p)$. Therefore we only need to prove that a generic point of $p(A')$ has fibers of dimension $1$ if and only if condition ($\star$) is satisfied.

It is easy to check that the residue map sends $A'$ to $R\setminus \Sigma$. Therefore the locus $p(A')$ is the locus of points such that the map ${\rm ev}$ defined in the proof of Proposition~\ref{pr:fiber2} is injective. Thus a point of $p(A')$ has fibers of dimension 1 by $p$ if and only if $R\cap {\rm ev}(\bigoplus_j L_j)$ is of dimension $1$. Now, $R\cap {\rm ev}(\bigoplus_j L_j)$ is of dimension $1$ if and only if the preimage under $\rho$ of the point $(L_1,\ldots, L_q)\in  \prod_{j=1}^q \P R_j$ is composed of a unique point. 

Now the residue map is a submersion from $A_{\mathbf{g},\mathbf{Z},\mathbf{P}}^R$ to $R$. Therefore,  the map $\rho$ is finite on a dense open subset of $\P R\setminus \P \Sigma$ if and only if the fiber of $p$ is of dimension $1$ on a dense open set of ${\rm Im}(p)$. 
\end{proof}

\subsubsection{General case} We no longer assume that $\bZ$ is complete. We denote by $\bZ_m=(Z_{1,m},\ldots,Z_{q,m})$ the maximal completion of $\bZ$. Besides, we denote by $p_m: A_{\bg,\bZ_{m},\bP}^R \to {\rm Im}(p_m)$ the forgetful map of the differential.

\begin{mypr}\label{pr:fiber1} We suppose that $(\oR, R,(\oR_j)_{1\leq i\leq q})$ satisfy the condition $(\star)$. Then
we have $\dim({\rm Im}(p_m))=\dim({\rm Im}(p))$ if and only if for all $1\leq j\leq q$ we have $\dim(A_{g_j,Z_j,P_j}^{R_j})-1 \leq \dim(\M_{g_j,n_j+m_j})$.
\end{mypr}

\begin{proof} We proceed in two steps: first we assume that the base is connected and then we consider the general case. 

\subsubsection*{Connected case} We assume that $q=1$. In this case, the ``only if'' is trivial. Indeed $\P A_{g,Z,P}^R=\dim({\rm Im})(p_m)$ and  $\dim({\rm Im})(p)\leq \M_{g,n+m}$. 

We assume that the dimension of $\P A^R_{g,Z,P}$ is less than or equal to the dimension of $\M_{g,n+m}$. We have the following commutative diagram:
$$
\xymatrix{
A^R_{g,Z_m,P}\ar[r] \ar[d]_{p_m} & A^R_{g,Z,P} \ar[d]^{p}\\
{\rm Im}(p_m) \ar[r] & {\rm Im}(p),}
$$
where the horizontal arrows are the forgetful map of the zeros that are not accounted for by $Z$. We have seen that the image of $A^R_{g,Z_m,P}$ is dense in $A_{g,Z,P}$.  Therefore the image of ${\rm Im}(p_m)$ under the forgetful map of the points that are not accounted for by $Z$. is dense in ${\rm Im}(p)$. Then we have $\dim({\rm Im}(p_m))\geq \dim({\rm Im}(p))$. Now we will prove that $\dim(\P A^R_{g,Z,P})\leq \dim({\rm Im}(p))$. 

We consider the following two vector bundles over the moduli space of curves $\mathcal{M}_{g,n+m}$
\begin{eqnarray*}
K\M_{g,n}(P)&=& R^0\pi_*(\omega_C(\sum_{i=1}^m p_i \sigma_{n+i})),\\
E&=& \oR/R \oplus \left(\bigoplus_{i=1}^n J^{\rm hol}_{i,k_i}\right), 
\end{eqnarray*}
where $J^{\rm hol}_{i,k_i}$ is the vector space of holomorphic jets of order $k_i$ at the marked point $x_i$, i.e.
$$
J^{\rm hol}_{i,k_i}= R^0\pi_* (\omega(-k_i x_i)/ \omega).
$$
(beware the vector space of jets here is not the vector space of polar jets used in Section~\ref{ssec:stdiff}). We have a morphism $e:K\M_{g,n}(P)\to E$. The rank of $K\M_{g,n}(P)$ is $r_1=g-1+\sum p_i$ if $P$ is not empty and $r_1=g$ otherwise. The rank of $E$ is $r_2=\dim(\oR/R)+\sum k_i$. By assumption, we have
$$\dim(\P A^R_{g,Z,P})= \dim (\M_{g,n+m})+r_1-r_2-1\leq \dim (\M_{g,n+m}).$$
Let $\mathcal{E}\subset \mathcal{M}_{g,n+m}$ be the locus where $e$ is not injective.  We have $r_1\leq r_2+1$ thus the locus $\mathcal{E}$ is of codimension at most $r_2-r_1+1$ because it is the vanishing locus of $r_2-r_1+1$ minors of the map $e$. Therefore the locus $\mathcal{E}$ is of dimension greater than or equal to $\dim(\P A^R_{g,Z,P})=\dim({\rm Im}(p_m))$. 

To complete the proof, we show that ${\rm Im}(p)$ is open and dense in $\mathcal{E}$. Let $P'$ be a vector of $m$ positive integers such that $P'\leq P$. Let $Z'$ be a vector of $n$ nonnegative integers such that $Z'\geq Z$. The image of $\P A_{g,Z',P'}^R$ lies in $\mathcal{E}$. Conversely, the locus $\mathcal{E}$ is the union of all the ${\rm Im}(p')$ where $p'$ is the map from $\P A_{g,Z',P'}^R$ to $\mathcal{M}_{g,n+m}$ for $P'\leq P$ and $Z'\geq Z$. We have $\dim(\P A_{g,Z',P'}^R)<\dim(\P A_{g,Z,P}^R)\leq \dim(\mathcal{E})$ if $P'<P$ or $Z'>Z$.  Therefore all irreducible components of ${\rm Im}(p)$ have the same dimension as $\mathcal{E}$ and  $\dim({\rm Im}(p))=\dim({\rm Im}(p_m))$. 

\subsubsection*{Disconnected case} Suppose that there exists $1\leq j\leq q$ such that $\dim(\P A_{g_j,Z_j,P_j}^{R_j}) > \dim(\M_{g_j,n_j+m_j})$. Then the fibers of the map ${\rm Im}(p_{j,m})\to {\rm Im}(p_j)$ are of positive dimension. Thus for all points in ${\rm Im}(p)$ the fibers of the map $A_{\bg,\bZ,\bP}\to {\rm Im}(p)$ are of positive dimension.

Conversely, suppose that for all $1\leq j\leq q$, we have $\dim(\P A_{g_j,Z_j,P_j}^{R_j}) \leq \dim(\M_{g_j,n_j+m_j})$. Thus for all $1\leq j\leq q$, we have $\dim({\rm Im}(p_j))= \dim({\rm Im}(p_{j,m}))= \dim(\P A_{g_j,Z_j,P_j}^{R_j})$. Therefore, there exists a dense open subset $U_j\subset \P A_{g_j,Z_j,P_j}^{R_j}$ such that the morphism ${\rm Im}(p_{j,m})\to {\rm Im}(p_j)$ is finite over its image. Besides, the map $\P A_{\bg,\bZ,\bP}^R\to \P R$ and the maps $\P A_{g_j,Z_j,P_j}^{R_j}\to \P R_j$ are submersions. Thus, for all $1\leq j\leq q$, the image of $U_j$ under the residue map is an open subset of $\P R_j$ that we denote $\widetilde{U}_j\subset \P R_j$. 

Now we consider the morphism $\rho:\P R\setminus \P\Sigma \to \prod_{j} U_j$. We claim that the preimage of $\prod_j U_j$ under $\rho$ is a non empty open subset in $\P R$. Indeed, if we suppose that $\rho^{-1}(\prod_j U_j)$ is empty, then the image of $\P R\setminus \P\Sigma$ under $\rho$ is contained in a finite union of closed subsets of the form $(\P R_j\setminus \widetilde{U}_{j}) \times \prod_{j'\neq j} \P R_{j'}$ for some $1\leq j\leq q$. However, the space $\P R\setminus \P\Sigma$ is irreducible, thus its image under $\rho$ is contained in one such subspace. This would imply that the image of $\P R \setminus \P\Sigma\to \P R_j$ is contained in a closed subspace and this is not possible (because $R_j$ is the image of the projection of $R$ onto $\oR_j$). 

Putting everything together, the preimage of $\prod_j \widetilde{U}_j$ under the composition of morphisms $\P A_{\bg,\bZ,\bP}^R\to \P R\setminus \P \Sigma \to \prod_{j} R_j$ is an open and dense subspace $U$:
$$
\xymatrix{
U \subset \P A^R_{\bg,\bZ,\bP}\ar[r] \ar[d] & {\rm Im}(p) \ar[d]\\
\prod_{j=1}^q U_j \subset  \prod_{j=1}^q \P A^{R_j}_{g_j,Z_j,P_j} \ar[r] & \prod_{j=1}^q {\rm Im}(p_j) }
$$
The lower arrow is finite from $\prod_j {U_j}$ to its image. By construction the subspace $U$ is embedded in $\prod {U_j}$. Therefore the map $U\to {\rm Im}(p)$ is finite over its image and $\dim({\rm Im}(p))= \dim(\P A^R_{\bg,\bZ,\bP})=\dim({\rm Im}(p_m))$.
\end{proof}

\begin{mynot}\label{not:starstar}
We will say that $(\mathbf{g},\mathbf{Z},\mathbf{P},R)$ satisfies condition $(\star\star)$ if and only if the two following conditions are satisfied \begin{itemize}
\item the vector spaces $(\oR, R, (\oR_j)_{1\leq j\leq q})$ satisfy the condition $(\star)$;
\item for all $1\leq j\leq q$, we have $\dim(A_{g_j,Z_j,P_j}^{R_j})-1 \leq \dim(\M_{g_j,n_j+m_j})$.
\end{itemize}
\end{mynot}


\begin{mypr}\label{pr:fibergen}
The morphism $p:\P A_{\bg,\bZ,\bP}^R\to {\rm Im}(p)$ is birational if and only if $(\mathbf{g},\mathbf{Z},\mathbf{P},R)$ satisfies the condition $(\star\star)$.
\end{mypr}

\begin{proof}
Proposition~\ref{pr:fiber3} implies that ${\rm dim}({\rm Im}(p))=\dim(\P A_{\bg,\bZ,\bP}^R)$ if and only if $(\mathbf{g},\mathbf{Z},\mathbf{P},R)$ satisfies the condition $(\star\star)$. Therefore if $p:\P A_{\bg,\bZ,\bP}^R\to {\rm Im}(p)$ is birational then the condition $(\star \star)$ is satisfied. 

Conversely if $(\star \star)$ is satisfied, then there exists a dense open subspace $U$ in ${\rm Im}(p)$ such that for any point in $U$, the fiber of $p$ over this point is finite. Suppose that there are at least two points in the preimage of a marked curve $(C,(x_{j,i})_{j,i})\in {\rm Im}(p)$. Then there exist two non-proportional  meromorphic differentials $\alpha$ and $\alpha'$ supported on $C$ with orders of zeros and poles prescribed by $\bZ$ and $\bP$ and with the same residues at the poles. Any non zero linear combination of these two differentials is in $A_{\bg,\bZ,\bP}^R$ and in the pre-image of $(C,x_{j,i})$. This is a contradiction with the finiteness of the fibers of $p$ over $U$.
\end{proof}

\section{Boundary components of strata of stable differentials}\label{sec:boundary}

Let $(\bg, \bZ,\bP,R\subset \oR)$ be a quadruple satisfying Assumption~\ref{assumption}. 
\begin{mynot}
We respectively denote by $\overline{A}^R_{\mathbf{g},\mathbf{Z},\mathbf{P}}$ and $\P\overline{A}^R_{\mathbf{g},\mathbf{Z},\mathbf{P}}$ the Zariski closure of $A^R_{\mathbf{g},\mathbf{Z},\mathbf{P}}$ and $\P A^R_{\mathbf{g},\mathbf{Z},\mathbf{P}}$ in $\oH_{\bg,\bZ,\bP}$ and $\P\oH_{\bg,\bZ,\bP}$.
\end{mynot} 
In this section we describe the boundary components of $\overline{A}^R_{\mathbf{g},\mathbf{Z},\mathbf{P}}$. We will see that these can described with combinatorial objects called $\mathbf{P}$-admissible graphs. We also describe the subset of boundary divisors among these boundary components.

\subsection{Twisted graphs with level structures}\label{ssec:twisted}

 We introduce $\bP$-admissible graphs here and in the subsequent section, we explain how they correspond to strata of $\overline{A}^R_{\mathbf{g},\mathbf{Z},\mathbf{P}}$.

Let $\Gamma$ be a semi-stable graph of type $(\mathbf{g},\mathbf{n}, \mathbf{P}$). We denote by $H_e$ the set of half-edges of $\Gamma$ which are not legs. 

\begin{mydef}
A \textit{twist} on $\Gamma$ is a function
\begin{equation*}
I : H_e \to \mathbb{Z}
\end{equation*}
Satisfying the following conditions.
\begin{itemize}
\item If $h$ and $h'$ form an edge, then $I(h)+I(h')=0$.
\item Let $v$ and $v'$ be two vertices, and $\{(h_1,{h'}_1),\ldots,(h_n,{h'}_n)\}$ be the set of edges from $v$ to $v'$. Then either $I(h_j) = 0$ for all $1 \leq j \leq n$, or $I(h_j) > 0$ for all $1 \leq j \leq n$, or $I(h_j) < 0$ for all $1 \leq j \leq n$. We say that $v=v'$, or $v>v'$, or $v<v'$, depending on the above inequalities. 
\item The relation $\leq$ thus defined on vertices is transitive.
\end{itemize}
For shortness, a semi-stable graph endowed with a twist function will be called a {\em twisted graph}. If $(\Gamma,I)$ is a twisted graph, the above conditions  define a partial order on the set of vertices of $\Gamma$. 
\end{mydef}

\begin{mydef}
A \textit{level structure} on a twisted graph is a function:
\begin{equation*}
l: {\rm{Vertices}} \to \Z^-,
\end{equation*}
compatible with the partial order induced by the twist, i.e., for all vertices $v$ and $v'$,
\begin{equation*}
v = v'  \Rightarrow l(v) = l(v'),
\quad 
v < v'  \Rightarrow l(v) < l(v').
\end{equation*}
We impose that the image of $l$ is an interval containing all integers from $0$ to $-d$ and we call $d$ the \textit{depth} of the twisted graph. We will denote by $V^i$ the set of vertices of level $i$.
\end{mydef}
\begin{mydef}
An edge between vertices of the same level will be called an {\em horizontal edge}.
\end{mydef}

\begin{mydef}  A twisted graph with level structure is called $\mathbf{P}$-\textit{admissible} if all marked poles of order at least 2 belong to  vertices of level 0. For shortness  we will call such graphs \textit{admissible graphs}.
\end{mydef}

This definition of $\mathbf{P}$-admissibility implies in particular that unstable vertices can only be present at the level 0. In the sequel, we will see that $\bP$-admissible graphs represents loci in $\H_{\mathbf{g},\mathbf{n},\mathbf{P}}$ where the differential vanishes identically on the components of negative levels. As explained in the introduction, the appearance of unstable components on the level 0 ensures that the poles remains of fixed order. 

\begin{remark}
The reader should keep in mind that a stable differential cannot vanish identically on an unstable component. Indeed, otherwise there would be infinitely many automorphisms of the curve preserving the differential ; this would contradict the stability condition (see Definition~\ref{def:stablestack}).
\end{remark}

\begin{example} We represent in Figure~\ref{example} an example of admissible graph. Each vertex $v$ is represented by a circle containing the integer $g_v$. The marked poles and zeros are represented by legs. A leg corresponding to a pole (respectively a zero) of order $k$ is marked by $-k$ (respectively $+k$). The twists are indicated on each edge.
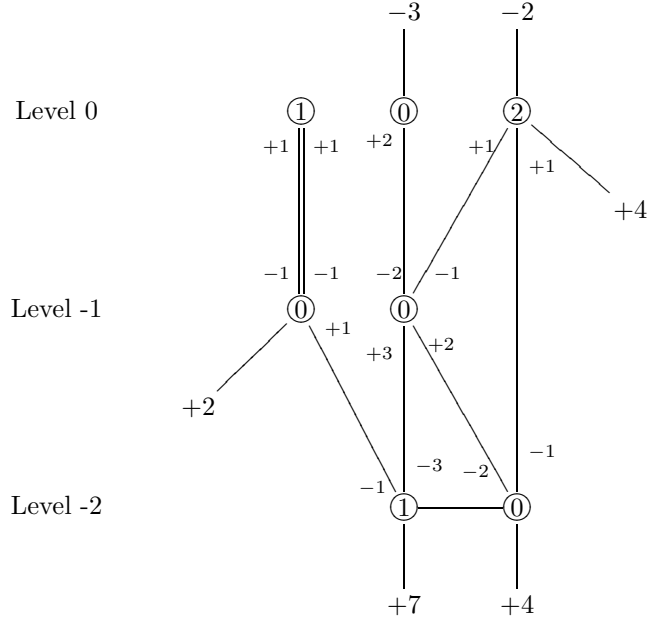
\begin{figure}[h!]
$$
\xymatrix{
&&&-3\ar@{-}[d]& -2 \ar@{-}[d]& \\
\text{Level 0}&& *+[Fo]{1}  \ar@{}[ddr] \ar@{=}[dd]^{\begin{smallmatrix}  +1\\ \\ \\ \\ \\ \\ \\ \\ -1 &&&&\;-2   \end{smallmatrix}}_{\begin{smallmatrix}+1\\ \\ \\ \\ \\  \\ \\ \\ -1\end{smallmatrix}} & *+[Fo]{0}  \ar@{-}[dd]^{\;\;\;\begin{smallmatrix}\; \; &+1 \\ \\ \\ \\ \\  \\ \\ \\ -1\end{smallmatrix}}_{\;\;\;\;\;\begin{smallmatrix}+2\\ \\  \\ \\ \\  \\\\ \\ \\ \\ \end{smallmatrix}}&*+[Fo]{2} \ar@{-}[rd] \ar@{-}[dddd]^{\begin{smallmatrix}+1\\ \\ \\ \\ \\ \\ \\ \\ \\ \\ \\ \\ \\ \\ \\ \\ \\ \\-1\end{smallmatrix}} \ar@{-}[ddl]\\
&&&&&+4\\
\text{Level -1}&&*+[Fo]{0} \ar@{}[dd]^{\;\;\begin{smallmatrix}+1\\ \\ \\ \\ \\ \\ \\ \\ \\  \\ & -1\end{smallmatrix}\;\;\;} \ar@{-}[rdd] & *+[Fo]{0} \ar@{-}[dd]^{\;\;\begin{smallmatrix}+2\\ \\ \\ \\ \\ \\ \\  \\ & -2\end{smallmatrix}}_{\begin{smallmatrix}+3\\ \\ \\ \\ \\  \\ \\ \\ \end{smallmatrix}}^{\begin{smallmatrix}\\ \\ \\ \\ \\  \\ \\ \\ -3\end{smallmatrix}} \ar@{-}[ddr]& \\
&+2\ar@{-}[ru]&&&\\
\text{Level -2}&&&*+[Fo]{1} \ar@{-}[r] &*+[Fo]{0}& \\
&&&+7 \ar@{-}[u]&+4\ar@{-}[u]&
}
$$

\caption{An example of admissible graph of genus 7 for the vectors $Z=(2,4,4,7)$ and $P=(-3,-2)$.\label{example}}
\end{figure}
\end{example}

\begin{mydef} Let $(\Gamma, I, l)$ be a semi-stable graph with a twist and a level structure. We say that $(\Gamma, I, l)$ is a {\em twisted stable graph}  if $\Gamma$ is a stable graph (in the sense of Definition~\ref{def:stgraph}). 
\end{mydef}

\begin{mydef}\label{def:realizable} Let $(\Gamma, I, l)$ be a semi-stable graph with a twist and a level structure. We say that $(\Gamma, I, l)$ is {\em realizable} if for all vertices $v$ of $\Gamma$ we have
\begin{equation}\label{ineq:real}
\sum_{(j,i)\mapsto v} k_{j,i} - \sum_{(j,n_j+i)\mapsto v} p_{j,i} + \sum_{h\mapsto v} I(h)-1 \leq 2g(v)-2
\end{equation}
where the sums are respectively over marked points corresponding to zeros, marked points corresponding to poles and half-edges adjacent to $v$. 
\end{mydef}

The following lemma will be needed later to compare the space of stable differentials and the incidence variety.
\begin{mylem}\label{lem:corgraph}
If $\mathbf{Z}$ is complete, then there exists a bijection between 
the set of realizable and admissible graphs and the set of realizable and twisted stable graphs.
\end{mylem}

\begin{proof}
 To an admissible graph we assign its stabilization.  The twists and levels on this graph are obtained by restriction of the former twists and level functions. 
 
From a twisted stable graph, we construct an admissible graph by adding an unstable vertex for each marked point corresponding to a pole of order $p$ greater than $1$ and  adjacent to a vertex of level $< 0$. This new vertex is of level 0 and the new edge between this vertex an the rest of the curve has twists given by $+p-1$ and $-p+1$.
\end{proof}

\begin{example}
Here is the stabilization of the admissible graph of Figure~\ref{example}. 
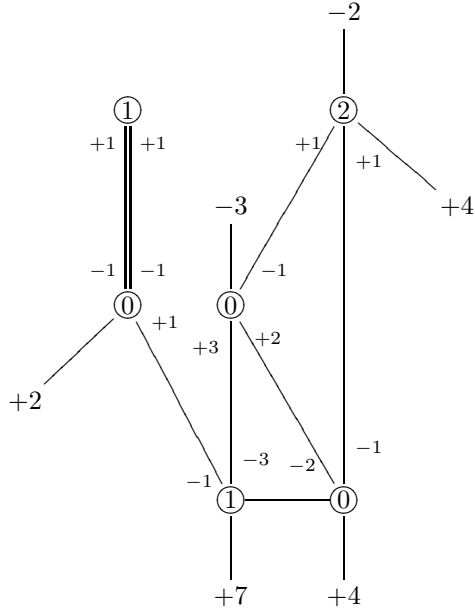
$$
\xymatrix{
&&& -2 \ar@{-}[d]& \\
& *+[Fo]{1}  \ar@{}[ddr] \ar@{=}[dd]^{\begin{smallmatrix}  +1\\ \\ \\ \\ \\ \\ \\ \\ -1 &&&&\;   \end{smallmatrix}}_{\begin{smallmatrix}+1\\ \\ \\ \\ \\  \\ \\ \\ -1\end{smallmatrix}} &  \ar@{}[dd]^{\;\;\;\begin{smallmatrix}\; \; &+1 \\ \\ \\ \\ \\  \\ \\ \\ -1\end{smallmatrix}}_{\;\;\;\;\;\begin{smallmatrix}\\ \\  \\ \\ \\  \\\\ \\ \\ \\ \end{smallmatrix}}&*+[Fo]{2} \ar@{-}[rd] \ar@{-}[dddd]^{\begin{smallmatrix}+1\\ \\ \\ \\ \\ \\ \\ \\ \\ \\ \\ \\ \\ \\ \\ \\ \\ \\-1\end{smallmatrix}} \ar@{-}[ddl]\\
&&-3\ar@{-}[d]&&+4\\
&*+[Fo]{0} \ar@{}[dd]^{\;\;\begin{smallmatrix}+1\\ \\ \\ \\ \\ \\ \\ \\ \\  \\ & -1\end{smallmatrix}\;\;\;} \ar@{-}[rdd] & *+[Fo]{0} \ar@{-}[dd]^{\;\;\begin{smallmatrix}+2\\ \\ \\ \\ \\ \\ \\  \\ & -2\end{smallmatrix}}_{\begin{smallmatrix}+3\\ \\ \\ \\ \\  \\ \\ \\ \end{smallmatrix}}^{\begin{smallmatrix}\\ \\ \\ \\ \\  \\ \\ \\ -3\end{smallmatrix}} \ar@{-}[ddr]& \\
+2\ar@{-}[ru]&&&\\
&&*+[Fo]{1} \ar@{-}[r] &*+[Fo]{0}& \\
&&+7 \ar@{-}[u]&+4\ar@{-}[u]&
}
$$
\end{example}

\subsection{Boundary strata associated to admissible graphs}\label{ssec:boundary}

Let $(\bg, \bZ,\bP,R\subset \oR)$ be a quadruple satisfying Assumption~\ref{assumption}.  Let $(\Gamma,I,l)$ be an admissible graph.  In this subsection, we assign to this admissible graph a stratum of abelian differentials $A_{\Gamma,I,l}\subset \oH_{\mathbf{g},\mathbf{n},\mathbf{P}}$ that lies in the closure of $A_{\mathbf{g},\mathbf{Z},\mathbf{P}}^R$. We build this stratum level by level.

To every level~$0$ vertex we assign a substack of the corresponding space of differentials. To every vertex of negative levels we assign a substack of the corresponding moduli space of curves. The product of these cycles will give us a substack of the space $\oH_\Gamma$ by putting an identically vanishing differential on every component of the curve of negative level.  Thus our input is $(\mathbf{Z}, R)$ and an admissible graph $(\Gamma, I, l)$ of type $\mathbf{g},\mathbf{n},\mathbf{P}$; our output is a collection of subspaces of the spaces of differentials (for level 0 vertices) and of the spaces of curves (for vertices of negative levels). 

\subsubsection*{Level 0 and -1.} We respectively denote by $q_0$ and $q_1$ the numbers of vertices of level $0$ and $-1$. Besides, we denote by $\bg_0$ and $\bg_1$ the lists of genera of vertices of level 0 and $-1$. We determine orders of zeros and poles as follows:
\begin{itemize}
\item  For all $1\leq j\leq q_0$, we construct the vector $P^0_j$ by taking the entries of $\bP$ for all marked poles on the $j$-th component and a $-1$ for each horizontal half-edge; we construct the vector $Z^0_j$ by taking the entries of $\bZ$ for all marked zeros carried by the $j$-th component and $I(h)-1$ for each half-edge $h$ to a deeper level. 
\item For all $1\leq j\leq q_1$, we construct the vector $P^1_j$ by taking $I(h)+1$ for all half-edges to level 0 and $1$ for all horizontal half-edges adjacent to to the $j$-th component; we construct the vector $Z^1_j$ by taking the entries of $\bZ$ for all marked zeros carried by the $j$-th component and $I(h)-1$ for each half-edge to a deeper level. 
\item We denote by $\bZ_i=(Z^i_1,\ldots, Z^{i}_{q_i})$ and $\bP_i:(P^i_1,\ldots, P^{i}_{q_i})$ for $i=0,1$.
\end{itemize}
Now we define the residue conditions as follows
\begin{itemize}
\item We denote by ${\rm hor}_0$ the number of horizontal half edges of level and by $\oR^H= \C^{{\rm hor}_0}$. We denote by $\oR^1$ the space of residues of the space of stable differentials $\oH_{\bg_1,\bn_1,\bP_1}$ (where $\bn_1$ is the determined by the length of entries of $\bZ_1)$. 
\item We define 
$$
{\rm proj}:  \oR  \oplus \oR^H \oplus  \oR^1\to \oR^1
$$
as the projection along $\oR \oplus \oR^H$.
\item We consider the vector subspace $R'=R\oplus \oR^H \oplus  \oR^1$, and  we define the vector subspace $\widetilde{R}$ of $R'$ by the following linear relations:
\begin{itemize}
\item $r_h+r_h'=0$ for all horizontal edges $(h,h')$;
\item for all vertex of level 0, we have
$$
\sum_{p\mapsto v} r_p + \sum_{\begin{smallmatrix} h \text{ horizontal} \\ h\mapsto v \end{smallmatrix}} r_h + \sum_{\begin{smallmatrix} h \text{ to level $-1$} \\ h\mapsto v \end{smallmatrix}} -r_h=0
$$
where the first sum is over marked poles adjacent to $v$, the second is over horizontal-edges, and the last one is over the edges  to level $-1$ (in this last sum $r_h$ is the value of the residue at the corresponding half-edge of level $-1$).
\end{itemize}
\item Finally we denote by $R^0={\rm ker}({\rm proj}) \cap \widetilde{R}$ and $R^1={\rm proj}( \widetilde{R})$.
\end{itemize}
With these data, we define the level 0 and $-1$ strata as
\begin{eqnarray*}
A_{\Gamma,I,l}^0&=&A_{\mathbf{g}_0,\bZ_0, \bP_0}^{R^0} \subset \oH_{\bg_0,\bn_0,\bP_0} \\
A^{1}_{\Gamma,I,l}&=& p(A_{\mathbf{g}_1,{\mathbf{Z}_1},{\mathbf{P}_1}}^{R^{1}})\subset \oM_{\mathbf{g}_1,\mathbf{n}_1,\mathbf{m}_1}=\prod_{v \in V^1} \oM_{g_v,n_v+m_v},
\end{eqnarray*} 
where  $p: \oH_{\mathbf{g}_1,\mathbf{n}_1,\mathbf{P}_1}\to \oM_{\mathbf{g}_1,\mathbf{n}_1,\mathbf{m}_1}$ is the forgetful map. 

\begin{example}
To illustrate the definition of $R^0$ and $R^1$, we compute all vector spaces for the following two graphs 
$$
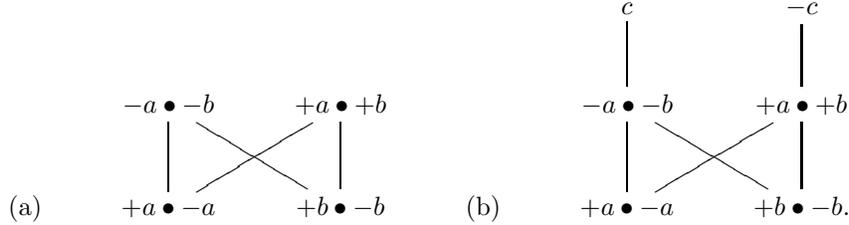
\xymatrix{
&&&&c \ar@{-}[d]& -c \ar@{-}[d]
\\
&-a \bullet -b \ar@{-}[rd] \ar@{-}[d] & +a \bullet +b \ar@{-}[ld] \ar@{-}[d] &&-a \bullet -b \ar@{-}[rd] \ar@{-}[d] & +a \bullet +b \ar@{-}[ld] \ar@{-}[d] \\
\text{(a)} &+a\bullet -a & +b \bullet -b &
 \text{(b)} &+a\bullet -a & +b \bullet -b.}
$$
On these two examples we have not represented the genera of the vertices and we have only represented the legs with poles (thus at level 0). In the first case $R=\oR=\{0\}$ (there are no poles). In the second case we assume that $R=\oR\simeq \C$ (we impose no condition on the residues). 

All letters stand for the value of the residue, i.e. for a coordinate in $\widetilde{\oR}\bigoplus \oR^1$ corresponding either to a half-edge or to a marked pole. In the following table we give the dimensions and equations of all sub-vector spaces of $\widetilde{\oR}$ and a presentation of $\widetilde{\oR}^1$ and $R^1$.
\begin{center}
\begin{tabular}{c|c|c}
{Vector space} & Example (a) &Example (b) \\
\hline\hline
$ \oR\oplus \oR^H\oplus \oR^1$&$\{0\} \oplus \{0\} \oplus \C^2$ & $\; \C\oplus  \{0\} \oplus\C^2$ \\
\hline $
R'$ & $= \oR\oplus \oR^H\oplus \oR^1$ &$= \oR\oplus \oR^H\oplus \oR^1$ \\
\hline $
\text{Relations from edges}$& \text{ none }  & \text{ none }  \\

\hline $
\text{Relations from vertices}$& $ \{a+b=0\} $& $\{c-a-b=0\}$ \\
\hline & & $ \{(a=\epsilon_1, b=\epsilon_2, $
\\ 
$
\widetilde{R}$&$\{(a=\epsilon, b=-\epsilon, \epsilon\in \C\}$& $c=-\epsilon_1-\epsilon_2, (\epsilon_1,\epsilon_2) \in \C^2\}$ \\
\hline
$R^1$& $\{(a=\epsilon, b=-\epsilon, \epsilon\in \C\}$& $ \{(a=\epsilon_1, b=\epsilon_2, (\epsilon_1,\epsilon_2) \in \C^2\}$ \\
\hline
$R^0$& $\{0\}$& $\{0\}$
\end{tabular}
\end{center}
%
\end{example}

\subsubsection*{Level $-\ell$.} Let $(\Gamma',I',l')$ be the graph obtained from $\Gamma$ by contracting edges between vertices of levels 0 through $-\ell+1$. The twist on $\Gamma$ restricts to $\Gamma'$ and the level structure is shifted. Vertices of levels 0 to $-\ell+1$ merge to level~0, level $-\ell$ vertices become level -1 vertices and so on.  Therefore we have the natural identification
$$
\prod_{\begin{smallmatrix} v\in V(\Gamma),\\ \ell(v)=-\ell \end{smallmatrix}} \oM_{g(v),n(v)}= \prod_{\begin{smallmatrix} v\in V(\Gamma'),\\ \ell(v)=-1 \end{smallmatrix}} \oM_{g(v),n(v)}
$$
and  we define $A^\ell_{\Gamma,I,l}$ as $A^1_{\Gamma',I',l'}$. 


\begin{example}
The contraction of level 0 and $-1$ of the admissible graph of Figure~\ref{example} gives the following admissible graph with two levels
$$
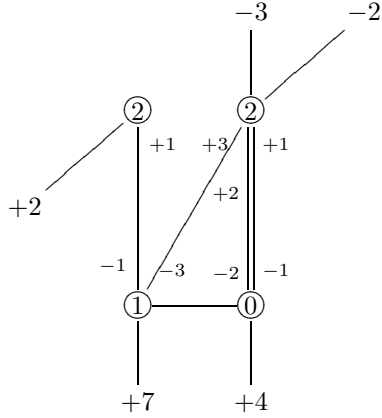
\xymatrix{
&&-3\ar@{-}[d]& -2 \ar@{-}[ld]& \\
&*+[Fo]{2} \ar@{-}[dd]_{\begin{smallmatrix}  \\ \\ \\ \\ \\ \\ \\ \\ -1   \end{smallmatrix}}^{\begin{smallmatrix} +1 \;\;\; & +3\;\; \\ \\ \\ \\ \\ \\ \\ \\  -3   \end{smallmatrix}}&*+[Fo]{2} \ar@{-}[ldd] \ar@{=}[dd]_{\begin{smallmatrix}  \\ \\ \\\\ +2 \\ \\ \\ \\ \\ -2  \end{smallmatrix}}^{\begin{smallmatrix}+1\\ \\ \\ \\ \\  \\ \\ \\ -1\end{smallmatrix}} \\
+2\ar@{-}[ru]&&&\\
&*+[Fo]{1} \ar@{-}[r] &*+[Fo]{0}& \\
&+7 \ar@{-}[u]&+4\ar@{-}[u]&
}
$$
If we assume here that $R=\oR\simeq \C$ then here we have $R^0=\{0\}$ while ${R^0}'=R$. 
\end{example}

\begin{mynot}\label{not:spacegraph}
Now that we have defined the $A^\ell_{\Gamma,I,l}$ for all levels, we denote
\begin{equation*}
A_{\Gamma,I,l}=\prod_{\ell \in \mathbb{Z}^-} A^{\ell}_{\Gamma,I,l}.
\end{equation*}
We have a natural morphism of $A_{\Gamma,I,l}\hookrightarrow \oH_{\mathbf{g},\mathbf{n},\mathbf{P}}$: the differential is nonzero only on the level 0 vertices and vanishes identically everywhere else.   We will call $A_{\Gamma,I,l}$ the {\em boundary stratum of type $(\mathbf{g},\mathbf{Z}, \mathbf{P}, R)$ associated to $(\Gamma,I,l)$}.
\end{mynot}

\begin{remark}
Note that the stratum $A_{\Gamma,I,l}$ is constructed from an admissible graph $(\Gamma, I,l)$ of type $(\mathbf{g},\mathbf{n},\mathbf{P})$, a space of residues $R \subset \oR$ and $q$ vectors of zeros $\mathbf{Z}$.
However, for simplicity, $R$ and $\mathbf{Z}$ do not explicitly appear in the notation.
\end{remark}

\begin{remark}
If $R=\oR$, then the construction of the space of residues is the translation of the global residue condition of~\cite{BCGGM}. For every level $-\ell$ and every vertex $v$ of level greater than $-\ell$ that does not contain a pole the following conditions holds. Let $h_1,\ldots, h_k$ denote the  half-edges adjacent to $v$ and part of an edge to a vertex of level $-\ell$. Then the sum of residues assigned to this set of half-edges is zero.

Our definition of the $R^i$ is more complicated to state because we need to take into account any vector subspace $R$ of $\oR$.
\end{remark}

\subsection{Stratification of $\overline{A}_{\mathbf{g},\mathbf{Z},\mathbf{P}}^R$}

Let $(\bg, \bZ,\bP,R\subset \oR)$ be a quadruple satisfying Assumption~\ref{assumption}.

\begin{mylem}\label{boundaries}
Let $(\Gamma,I,l)$ be an admissible graph of type $(\mathbf{g},\mathbf{Z},\mathbf{P},R)$. The locus $A_{\Gamma,I,l}$ lies in the closure of  $A_{\mathbf{g},\mathbf{Z},\mathbf{P}}^R$. Conversely if $y$ is  a point of $\overline{A}_{\mathbf{g},\mathbf{Z},\mathbf{P}}^R$ then there exists an exterior completion  $\mathbf{Z}'$ of $\mathbf{Z}$ and an admissible graph  $(\Gamma,I,l)$ of type $(\mathbf{g},\mathbf{Z}',\mathbf{P},R)$ such that $y$ lies in $\pi(A_{\Gamma,I,l})$, where $\pi:A_{\mathbf{g},\mathbf{Z}',\mathbf{P}}^R\to A_{\mathbf{g},\mathbf{Z},\mathbf{P}}^R$ is the forgetful map of the marked zeros that are not accounted for by $\mathbf{Z}$.
\end{mylem}

\begin{remark} The set of admissible and realizable graphs (see Definition~\ref{def:realizable}) is finite. Besides, if $(\Gamma,I,l)$ is an admissible graph, then the locus $A_{\Gamma, I, l}$ is empty  if $(\Gamma,I,l)$ is not realizable. Thus Lemma~\ref{boundaries} asserts that $\overline{A}_{\mathbf{g}\mathbf{Z},\mathbf{P}}^R$ is stratified by finitely many strata corresponding to admissible graphs.  
\end{remark}

Before proving it we will introduce the incidence variety compactification of~\cite{BCGGM}. 

\begin{mynot}
We suppose that $2g_j-2+n_j+m_j>0$ for all $1\leq j\leq q$. Then we denote by $K\oM_{\mathbf{g},\mathbf{n}}(\mathbf{P})$ the vector bundle
$$
R^0\pi_*\left(\omega\left( \sum_{j=1}^q \sum_{i=1}^{m_j} p_{j,i} \sigma_{j,n_j+i}  \right)\right),
$$
where $\pi: \oC_{\mathbf{g},\mathbf{n},\mathbf{m}}\to \oM_{\mathbf{g},\mathbf{n},\mathbf{m}}$ is the forgetful map, $\omega$ is the relative cotangent bundle and the $\sigma_{j,i}$'s are the sections of the universal curve (this generalize the notation~\ref{notKM} to the disconnected case).
\end{mynot}

As in Section~\ref{sec:stdiff}, there exists a natural morphism of cones 
$${\rm stab}:\oH_{\mathbf{g},\mathbf{n},\mathbf{P}}\to   K\oM_{\mathbf{g},\mathbf{n}}(\mathbf{P}).$$

\subsubsection{The image of $\overline{A}_{\mathbf{g},\mathbf{n},\mathbf{P}}^R$ under the morphism $\rm stab$}

\begin{mydef}  We denote by
$\Omega\M^{\rm inc}_{\mathbf{g}}(\mathbf{Z},\mathbf{P})^R \subset K\oM_{\mathbf{g},\mathbf{n}}(\mathbf{P})$ the image of ${A}_{\mathbf{g},\mathbf{n},\mathbf{P}}^R$ under the morphism $\rm stab$.  The {\em incidence variety} for the tuple  $(\mathbf{g}, \mathbf{Z},\mathbf{P}, R)$ is the closure of $\Omega\M^{\rm inc}_{\mathbf{g}}(\mathbf{Z},\mathbf{P})^R$ in $ K\oM_{\mathbf{g},\mathbf{n}}(\mathbf{P})$.
\end{mydef}

 The morphism ${\rm stab}$ induces a map from $A_{\mathbf{g},\mathbf{n},\mathbf{P}}^R$ to $\Omega\M^{\rm inc}_{\mathbf{g}}(\mathbf{Z},\mathbf{P})^R$. We will use the same notation for the morphism ${\rm stab}$ and its restriction
$${\rm stab}: \overline{A}_{\mathbf{g},\mathbf{n},\mathbf{P}}^R\to \overline{\Omega}\M^{\rm inc}_{\mathbf{g}}(\mathbf{Z},\mathbf{P})^R.
$$
\begin{mypr}\label{pr:correspondence}
We suppose that $\mathbf{Z}$ is complete. The map ${\rm stab}: \overline{A}_{\mathbf{g},\mathbf{n},\mathbf{P}}^R\to \overline{\Omega}\M^{\rm inc}_{\mathbf{g}}(\mathbf{Z},\mathbf{P})^R
$ is an isomorphism.
\end{mypr}

\begin{remark}
Beware that this statement is valid only under the hypothesis that $\mathbf{Z}$ is complete. Otherwise the map ${\rm stab}$ may have fibers of positive dimension and/or not be surjective.
\end{remark}

\begin{proof}
In Section~\ref{sec:stdiff} we proved that the following square  is cartesian
$$
\xymatrix{\oH_{\mathbf{g},\mathbf{n},\mathbf{P}} \ar[r]^{\Phi_{j,i}} \ar[d]& \ar[d] \bigoplus_{\begin{smallmatrix} 1\leq j \leq q\\ 1\leq i \leq m_j \end{smallmatrix} } \oP_{j,n_j+ i} \\
 K\oM_{\mathbf{g},\mathbf{n}}(\mathbf{P}) \ar[r]^{{\rm proj}_{j,i}}& \bigoplus_{\begin{smallmatrix} 1\leq j \leq q\\ 1\leq i \leq m_j \end{smallmatrix} }J_{j,n_j+i},
}
$$
where $\oP_{j,n_j+i}$ is the cone of principal parts of order $p_{j,i}$ at the $i$-th marked point of $j$-th connected component and $J_{j,n_j+i}$ is the vector bundle of polar jets of order $p_{j,i}$.  We recall that we have defined the spaces \begin{eqnarray*}
 \widetilde{\oP}_{j,n_j+i}&=& (\oP_{j,n_j+i}\setminus \mathcal{A}_{j,n_j+i})\cup \text{ the zero section}\\
\widetilde{J}_{j,n_j+i}&=& (J_{j,n_j+i}\setminus \{ \text{leading term}=0\}) \cup \text{ the zero section}.
\end{eqnarray*}
 We have seen that the map $\Phi_{j,i}$ maps $\oP_{n_j+i}$ to $\widetilde{J}_{n_j+i}$ and that the restriction of $\phi_{i,j}$ to $\widetilde{\oP}_{j,n_j+i}\to \widetilde{J}_{j,n_j+i}$ is an isomorphism (see Lemma~\ref{lem:bij1}).  Thus, the morphism $\oH_{\bg,\bn,\bP}\to K\oM_{\bg,\bn}(\bP)$ is an isomorphism from the preimage of $\bigoplus \widetilde{\mathcal{P}}_{j,n_j+i}$ to the preimage of $\bigoplus \widetilde{J}_{j,n_j+i}$. 
 
 The spaces $\overline{A}_{\mathbf{g},\mathbf{n},\mathbf{P}}^R$ and $\overline{\Omega}\M^{\rm inc}_{\mathbf{g}}(\mathbf{Z},\mathbf{P})^R$ are defined as Zariski closure of open sub-space  of $\oH_{\bg,\bn,\bP}$ and $K\oM_{\bg,\bn}(\bP)$. Therefore we will prove that for all $1\leq j\leq i$ and $1\leq i\leq m_j$, the image of  $\overline{A}_{\mathbf{g},\mathbf{n},\mathbf{P}}^R$ (respectively $\overline{\Omega}\M^{\rm inc}_{\mathbf{g}}(\mathbf{Z},\mathbf{P})^R$)
 under $\Phi_{j,i}$ (respectively ${\rm proj}_{j,i}$) is included in $\widetilde{\oP}_{j,n_j+i}$ (respectively $\widetilde{J}_{j,n_j+i}$) to deduce the proposition.

Let us consider a differential $(C,\alpha)$ in $\overline{A}_{\mathbf{g},\mathbf{n},\mathbf{P}}^R$ and one of the marked points $x_{j,n_j+i}$ corresponding to a pole. There are two possibilities.
\begin{itemize}
\item The point $x_{j,n_j+i}$ belongs to a stable irreducible component of level 0. In which case the principal part belongs to $\oP_{n_j+i}\setminus \mathcal{A}_{n_j+i}$;
\item The point $x_{j,n_j+i}$ belongs to an unstable rational component. In this case the differential restricted to this rational component is necessarily given by $dw/w^{p_{j,i}}$ (the marked point is at 0 and the node at $\infty$). Indeed, this follows from the assumption that $\bZ$ is complete:  suppose that $\alpha$ has a zero outside the node ; then let $B\to \overline{A}_{\mathbf{g},\mathbf{n},\mathbf{P}}^R$ be a irreducible family of differentials with a special point $b_0\in B$ whose image is the class $[(C,\alpha)]$ and the image of $B\setminus\{b_0\}$ lies in ${A}_{\mathbf{g},\mathbf{n},\mathbf{P}}^R$. Then there exists a neighborhood $U$ of $b_0$ such that the differential parametrized by $U$ has an unmarked zero (this follows from Lemma~\ref{lem:univdiff}).  This is contradictory with the assumption that $\mathbf{Z}$ is complete (all zeros of differentials in ${A}_{\mathbf{g},\mathbf{n},\mathbf{P}}^R$ are at marked points). Therefore the principal part is equal to 0.
\end{itemize}
Therefore the image of $\overline{A}_{\mathbf{g},\mathbf{n},\mathbf{P}}^R$ under $\Phi_{j,i}$ is included in $\widetilde{\oP}^{j,n_j+i}$. Now, let us consider a differential in $\overline{\Omega}\M^{\rm inc}_{\mathbf{g}}(\mathbf{Z},\mathbf{P})^R$, and one of the marked points $x_{j,n_j+i}$ corresponding to a pole.  Once again, there are two possibilities.
\begin{itemize}
\item The point $x_{j,n_j+i}$ belongs to an irreducible component of level 0. In this case the differential has a pole of order exactly $p_{j,i}$ at this marked point and the jet at $x_{j,n_j+i}$ is in $\widetilde{J}_{j,n_i+j}$;
\item The point $x_{j,n_j+i}$ 
belongs to an irreducible component of level $-\ell<0$. Then the differential vanishes identically on this component and the jet at $x_{j,n_j+i}$ is 0. 
\end{itemize}
Therefore the image of $\overline{\Omega}\M^{\rm inc}_{\mathbf{g}}(\mathbf{Z},\mathbf{P})^R$ under ${\rm proj}_{j,i}$ is included in $\widetilde{J}_{j,n_j+i}$. This completes the proof.
\end{proof}

\subsubsection{The image of the $A_{\Gamma, I,l}$ under the morphism~$\rm stab$}

To complete the description of the map ${\rm stab}$ we describe the image of the strata defined by admissible graphs.

\begin{mynot}
Suppose that $\mathbf{Z}$ is complete and $(\Gamma, I,l)$ is a realizable stable twisted graph. Let $(\Gamma', I', l')$ be  the corresponding admissible graph. We denote by  $\Omega\M^{\rm inc}_{\Gamma', I', l'}$ the locus ${\rm stab}(A_{\Gamma, I,l})\subset K\oM_{\mathbf{g},\mathbf{n}}(\mathbf{P}).$
\end{mynot}

\subsubsection{Stratification of $ \overline{\Omega}\M^{\rm inc}_{\mathbf{g}}(\mathbf{Z},\mathbf{P})^R$} Recall  the main result of \cite{BCGGM}.

\begin{mylem}\label{boundariesinc} (Theorem 1.3 of~\cite{BCGGM}) Suppose that $\mathbf{Z}$ is complete and that the triple $(g_j, n_j, P_j)$ is stable for all $1\leq j\leq q$. Let $(\Gamma,I,l)$ be a stable graph. The locus ${\Omega}\M^{\rm inc}_{\Gamma,I,l}$ lies in the closure of  ${\Omega}\M^{\rm inc}_{\mathbf{g}}(\mathbf{Z},\mathbf{P})^R$. 
Conversely the space $\overline{\Omega}\M^{\rm inc}_{\mathbf{g}}(\mathbf{Z},\mathbf{P})^R$  is the union of the ${\Omega}\M^{\rm inc}_{\Gamma,I,l}$ for all stable graphs $(\Gamma,I,l)$.
\end{mylem}

\begin{remark}
The statement here is slightly more general than Theorem 1.3 of~\cite{BCGGM}. Indeed it takes into account possible disconnected basis and general choices of vector subspace $R\subset \oR$. However all arguments in the proof of~\cite{BCGGM} can be adapted {\em mutatis mutandis} to get the general statement above.
\end{remark}

\begin{proof}[Proof of Lemma~\ref{boundaries}] Suppose that $\mathbf{Z}$ is complete and that the triple $(g_j, n_j, P_j)$ is stable for all $1\leq j\leq q$. Then, using Lemma~\ref{boundariesinc} and Proposition~\ref{pr:correspondence} we automatically get 
$$\overline{A}^R_{\mathbf{g},\mathbf{Z},\mathbf{P}}= \bigcup A_{\Gamma, I, l}
$$
where the union is taken over all admissible graphs. Therefore we only need to prove that the statement of Lemma~\ref{boundaries} is still valid if we allow unstable base curves and non complete lists of vectors $\mathbf{Z}$.\bigskip

\noindent{\em Unstable basis.} We assume that $\mathbf{Z}$ is complete but we no longer impose that the base curves are stable. Then on a rational component with two points the only possible configuration is $P=(p)$ and $Z=(p-2)$. This is a closed point in $\oH_{0,1,(p)}$ Thus the statement of Lemma~\ref{boundaries} is still valid if we consider unstable basis.
\bigskip

\noindent{\em Non complete $\mathbf{Z}$.} We no longer impose that $\mathbf{Z}$ is complete. The space ${A}^R_{\mathbf{g},\mathbf{Z},\mathbf{P}}$ is the union of the $\pi({A}^R_{\mathbf{g},\mathbf{Z}',\mathbf{P}})$ for all exterior completions $\mathbf{Z}'$ of $\mathbf{Z}$ ($\pi$ being the forgetful map of the zeros which or accounted for by $\mathbf{Z}$). Therefore we have
$$
\overline{A}^R_{\mathbf{g},\mathbf{Z},\mathbf{P}}=\bigcup \pi(\overline{A}^R_{\mathbf{g},\mathbf{Z}',\mathbf{P}})=\bigcup  \pi(A_{\Gamma, I,l}),
$$
where the last union is over all possible completions and admissible graphs.
\end{proof}

\subsection{Description of boundary divisors}\label{ssec:div}

Let $(\bg, \bZ,\bP,R\subset \oR)$ be a quadruple satisfying Assumption~\ref{assumption}.  In the proof of the main theorem, we will be interested in the vanishing loci of sections of certain line bundles over $\overline{A}_{\mathbf{g},\mathbf{Z},\mathbf{P}}^R$. That is why we need to understand the boundary divisors of $\overline{A}_{\mathbf{g},\mathbf{Z},\mathbf{P}}^R$. The purpose of this section is to determine the set of admissible graphs which are associated to strata of codimension 1, i.e. to divisors.

\subsubsection{Bi-colored graphs}

\begin{mylem}\label{depth} Let $(\Gamma,I,l)$ be an admissible graph. The codimension of $\overline{A}_{\Gamma,I,l}$ in $\overline{A}_{\mathbf{g},\mathbf{Z},\mathbf{P}}^R$ is greater than or equal to the depth of the level structure~$l$.
\end{mylem}

\begin{proof}
Let $(\Gamma,I,l)$ be an admissible graph of depth $d$. Let $(\Gamma',I',l')$ be the admissible graph obtained by merging the levels 0 and $-1$. The locus $A_{\Gamma,I,l}$ lies in the closure of $A_{\Gamma',I',l'}$. Indeed this follows from Lemma~\ref{boundaries} applied to the stratum $A^0_{\Gamma',I',l'}$: the sub-graph of $(\Gamma, I, l)$ obtained by keeping only vertices of level 0 and -1 determines a boundary stratum of  $A^0_{\Gamma',I',l'}$. Thus $A_{\Gamma,I,l}$ is of dimension at most $\dim(A_{\Gamma,I,l})-1$. Therefore, every time we merge two levels we decrease the codimension at least by~1.
\end{proof}

\begin{mylem}\label{horizontal} Let $(\Gamma,I,l)$ be an admissible graph of depth 1. The codimension of $\overline{A}_{\Gamma,I,l}$ in $\overline{A}_{\mathbf{g},\mathbf{Z},\mathbf{P}}^R$ is greater than the number of horizontal edges.
\end{mylem}

\begin{proof} We can independently merge vertices along horizontal edges (See ``classical plumbing'' in~\cite{BCGGM}). At every merging, we decrease the codimension by at least~1.
\end{proof}

It follows from Lemmas~\ref{depth} and~\ref{horizontal} that a nontrivial admissible graph corresponding to a divisor of $A_{\mathbf{g},\mathbf{Z},\mathbf{P}}^R$ is necessarily of depth at most 1. Moreover, if it is of depth $1$  then it has no horizontal edges.

We recall from Section~\ref{ssec:boundary} that the boundary stratum associated to a graph of depth $1$ is equal to $p(A_{\mathbf{g}_1,\mathbf{Z}_1,\mathbf{P}_1}^{ R^1})\times A_{\mathbf{g}_0,\mathbf{Z}_0,\mathbf{P}_0}^{ R^0}$, where $p$ is the map from $A_{\mathbf{g}_1,\mathbf{Z}_1,\mathbf{P}_1}^{ R^1}$ to the moduli space of curves $\oM_{\mathbf{g}_1,\mathbf{n}_1,\mathbf{m}_1}$.

\begin{mynot} We denote by $\bic(\mathbf{g},\mathbf{Z},\mathbf{P},R)$ the set of realizable and admissible graphs with two levels and no horizontal edges. We will call such graphs {\em bi-colored graphs}. 

We say that a bi-colored graph $(\Gamma, I,l)$ satisfies condition $(\star\star)$ if $(\bg_1,\bZ_1,\bP_1,R^1)$ satisfies the condition $(\star\star)$ (see Notation~\ref{not:starstar}). We denote by $\D(\mathbf{g},\mathbf{Z},\mathbf{P},R)$ the set of bi-colored graphs satisfying condition $(\star\star)$.  
\end{mynot} 

\begin{remark} Elements of $\bic(\mathbf{g},\mathbf{Z},\mathbf{P},R)$  are twisted graphs with level structures. However, the level structure of a bi-colored graph is completely determined by the twists. This is why we will denote by $(\Gamma,I)$ the elements of $\bic(\mathbf{g},\mathbf{Z},\mathbf{P},R)$.
\end{remark}
 
 \begin{mypr}\label{conddiv} Let $(\Gamma,I)$ be a bi-colored graph in $\bic(\mathbf{g},\mathbf{Z},\mathbf{P},R)$. The locus $A_{\Gamma,I}$ is of co-dimension 1 in $\overline{A}_{\bg,\bZ,\bP}^R$ if and only if $(\Gamma,I)$ belongs to $\D(\mathbf{g},\mathbf{Z},\mathbf{P},R)$.
\end{mypr}

\begin{proof}
 Let $(\Gamma,I)\in \bic(\mathbf{g},\mathbf{Z},\mathbf{P},R)$. The proposition follows easily from the equation
\begin{equation}\label{eq:dimbound}
\dim(A^{R^0}_{\bg_0,\bZ_0,\bP_0})+ \dim(A^{R^1}_{\bg_1,\bZ_1,\bP_1})= \dim(A^{R}_{\bg,\bZ,\bP})
\end{equation}
Indeed $A_{\Gamma,I}$ is of co-dimension 1 in $A^{R}_{\bg,\bZ,\bP}$ if and only if $\dim(p(A^{R^1}_{\bg_1,\bZ_1,\bP_1}))=\dim(\P A^{R^1}_{\bg_1,\bZ_1,\bP_1})$, i.e. if and only if $(\bg_1,\bZ_1,\bP_1,R^1)$ satisfy condition $(\star\star)$ (see Proposition~\ref{pr:fiber1}).

Let us prove equation~\ref{eq:dimbound}. We assume first that $\bZ$ is complete for $(\bg,\bP)$, the dimension of $A_{\bg,\bZ,\bP}^R$ is given by
$\left(\sum_{j=1}^q  (2g_j-1+n_j)\right)+\dim(R).$ Therefore we have 
\begin{eqnarray*}
\dim(A^{R}_{\bg,\bZ,\bP})- \dim(A^{R^0}_{\bg_0,\bZ_0,\bP_0}) \!\!\! \! &-& \!\!\! \! \dim(A^{R^1}_{\bg_1,\bZ_1,\bP_1}) \\
&=& \left(\sum_{j=1}^q  (2g_j-1+n_j)\right) + \dim(R) -\dim(R^1\oplus R^0)\\
&& -\left(\sum_{v \in V^0}  (2g_v-1+n_v) + \sum_{v\in V^1}  (2g_v-1+n_v)\right)\\
&=& 2 h^1(\Gamma)- q  + {\rm Card}(V(\Gamma)) - {\rm Card}(E(\Gamma))
\\ && + \dim(R) -\dim(R^1\oplus R^0)\\
&=& h^1(\Gamma) +\dim(R)-\dim(R^1\oplus R^0).
\end{eqnarray*}
Thus we will prove that $\dim(R^1\oplus R^0)= \dim(R)+h^1(\Gamma)$. 

Let us recall the construction of $R^0$ and $R^1$. In absence of horizontal edges, we consider the vector space $\oR\oplus \oR^1$ and the projection ${\rm proj}: \oR\oplus \oR^1\to \oR^1$ along $\oR$. We also consider the vector subspace $\widetilde{R}\subset R\oplus \oR^1\subset \oR\oplus \oR^1$ defined by the  linear relations
$$
\sum_{h\in H(\Gamma),h\mapsto v} r_h=0
$$
for all vertices $v$ of level 0 (the sum is over all residues at half-edges adjacent to $v$). We defined $R^0={\ker}({\rm proj}) \cap \widetilde{R}$ and $R^1={\rm proj}( \widetilde{R})$. Thus $\dim(R^0)+\dim(R^1)=\dim(\widetilde{R})$. Therefore we need to prove that $\dim(\widetilde{R})=\dim(R)+h^1(\Gamma)$.

To prove this equality we use the graph $\Gamma'$ obtained from $\Gamma$ by adding one vertex per marked pole and one edge between this vertex and the vertex that carries the marked pole. We consider the spaces $C_0=\C^{V(\Gamma')}$ and $C_1=\C^{E(\Gamma')}$. We have the chain complex $d: C_1\to C_0$.

The space $R$ is a subspace of $C_0$: indeed, the space $\oR$ is a subspace of  the subspace of $\oR$ spanned by the vertices in $V(\Gamma')\setminus V(\Gamma)$.  The space $\widetilde{R}$ is naturally identified with $d^{-1}(R)$. Therefore $\dim(\widetilde{R})= \dim(R)+ \dim({\rm ker}(d))= \dim(R)+h^1(\Gamma)$. Q.E.D.

If $\bZ$ is not complete, then we consider $\bZ_m$ the maximal completion of $\bZ$. Then equation~\eqref{eq:dimbound} still holds by:
\begin{eqnarray*}
\dim(A_{\bg,\bZ,\bP}^R)= \dim(A_{\bg,\bZ_m,\bP}^R)&=&\dim(A^{R^0}_{\bg_0,\bZ_{0,m},\bP_0})+ \dim(A^{R^1}_{\bg_1,\bZ_{1,m},\bP_1})\\
&=&\dim(A^{R^0}_{\bg_0,\bZ_{0,},\bP_0})+ \dim(A^{R^1}_{\bg_1,\bZ_1,\bP_1}).
\end{eqnarray*}
\end{proof}

\subsubsection{Classification of boundary divisors}
 
 \begin{mynot}
Let $1\leq j\leq q$ and $1\leq i\leq \ell(Z_{j})$. We denote by $\mathbf{Z}_{j,i}$ the list of vectors obtained from $\mathbf{Z}$ by increasing the $i^{\rm th}$ coordinate of $Z_{j}$ by one.
 \end{mynot} 
 
\begin{mypr}\label{listdivisors} Let $\mathbf{Z}'$ be a completion of $\mathbf{Z}$ and let $(\Gamma,I,l)$ be an admissible graph such that $D=\pi(\overline{A}_{\Gamma,I})$ is a divisor of   $\overline{A}_{\mathbf{g},\mathbf{Z},\mathbf{P}}^R$ (where $\pi$ is the forgetful map of the points), then $D$ is necessarily of one of the four kinds:
\begin{enumerate}
\item the stratum $\overline{A}_{\Gamma,I}$ for $(\Gamma,I)\in \D(\mathbf{g},\mathbf{Z},\mathbf{P},R)$;
\item the locus $\overline{A}_{\mathbf{g},\mathbf{Z}_{j,i},\mathbf{P}}^R$ for some label $(j,i)$ corresponding to a marked point which is not a pole;
\item the locus $\overline{A}_{\Gamma,I,l}$ for a $\mathbf{P}$-admissible graph of depth $0$ with a unique horizontal edge;
\item the locus $\overline{A}^{R'}_{\mathbf{g},\mathbf{Z},\mathbf{P}}$ for the vector subspace $R'\subset R$ defined by the condition: ${\rm res}_{x_{j,n_j+i}}=0$ for a choice of $j$ and $i$ such  the point $x_{j,n_j+i}$ corresponds to a pole of order at most $-1$. 
\end{enumerate}
 \end{mypr}

 \begin{proof} Let $\mathbf{Z}'$ be a completion of $\mathbf{Z}$. If $\mathbf{Z}'$ is not the maximal completion then $\dim(A_{\mathbf{g},\mathbf{Z}',\mathbf{P}}^R)<\dim(A_{\mathbf{g},\mathbf{Z},\mathbf{P}}^R)$. The only possible admissible graph is the trivial and we obtain a divisor of  type 2. 
 
We suppose now that $\mathbf{Z}'=\mathbf{Z}_m$, then $(\Gamma,I,l)$ is of depth less than or equal to $1$ by Lemma~\ref{depth}.  If  $(\Gamma,I,l)$ is of depth 0 then $(\Gamma,I,l)$ has at most one horizontal edge (type 3).  If $(\Gamma,I,l)$ is of depth $1$ then either all or none of the edges of $(\Gamma,I,l)$ are contracted under the forgetful map of the marked points which are not accounted for by $\mathbf{Z}$ (otherwise this graph does not satisfy condition $(\star\star)$). If none of the edges are contracted, then $D$ is a divisor of type 1. If all edges are contracted then we get a divisor of type 2 or 4 (depending on whether there is a leg corresponding to a pole of order 1 on a level -1 vertex or not).
 \end{proof}
 
 \begin{mypr}\label{intersectiondiv}
 Let $D_1$ and $D_2$ be two divisors obtained from an admissible graph as in Proposition~\ref{listdivisors}. Then $D_1$ and $D_2$ have no common irreducible components.
 \end{mypr}
 
 \begin{proof} The divisors $D_1$ and $D_2$ can be of one of the four types described in Proposition~\ref{listdivisors}.  We will prove this proposition by considering every possible cases.
  
{\em Type 1/type 1.} Let $(\Gamma,I)$ and $(\Gamma',I')$ in  $\D(\mathbf{g},\mathbf{Z},\mathbf{P},R)$ such that $A_{\Gamma,I}$ and $A_{\Gamma',I'}$ have a common irreducible component $D$. The component $D$ determines a semi-stable graph by taking the dual graph of a any point of $D\cap {A}_{\Gamma,I}$, therefore $\Gamma=\Gamma'$. Moreover, the vertices of $\Gamma$ with identically zero differentials are the vertices of level $-1$. Therefore the level structure (or more precisely the signs of the twists) are the same for $(\Gamma,I)$ and $(\Gamma',I')$. Now the twist at an edge is determined by the vanishing order of the differential at the corresponding node on the component of level 0 for any point in $D\cap {A}_{\Gamma,I}$. Therefore $(\Gamma,I)=(\Gamma',I')$. Thus divisors of type 1 have no common irreducible components.

{\em Types 2 and 4.} The underlying generic curve of the divisors of type 2 or 4 is a curve without singularities, therefore divisors of type 2 or 4 do not intersect divisors of type 1 or type 3. Now the differentials of the generic differentials of two divisors of type 2 have different vanishing order at two of the marked points (either a marked zero or a marked pole of order $-1$).

{\em Type 3.} Two divisors of type 3 are distinguished by the toplogical types of a generic curve. Besides, a divisor of type 3 is distinguished from a divisor of type 1 because none of the components carries a vanishing differential in a divisor of type 3.
\end{proof}

\section{Computation of classes of strata}\label{sec:indfor}

Let $(\bg,\bZ,\bP,R\subset \oR)$ be a quadruple satisfying Assumption~\ref{assumption}. The purpose of this section is to prove the following generalization of Theorem~\ref{main} stated in the introduction.

\begin{myth}\label{maingen}
Let $(\bg, \bZ,\bP,R)$ be a quadruple satisfying Assumption~\ref{assumption}. The Poincar\'e-dual class of $\P\overline{A}_{\mathbf{g},\mathbf{Z},\mathbf{P}}^R\in H^*(\P\oH_{\mathbf{g},\mathbf{n},\mathbf{P}},\Q)$ is tautological (in the sense of Definition~\ref{def:gentaut}) and is explicitly computable. 
\end{myth}
 Theorems~\ref{main},~\ref{mainbis}, and~\ref{mainter} will be deduced from Theorem~\ref{maingen} at the end of the section. The most technical result involved in the proof of Theorem~\ref{maingen} is the induction formula for the classes $[\P\overline{A}_{\mathbf{g},\mathbf{Z},\mathbf{P}}^R]$ (see Section~\ref{ssec:induction}).

\subsection{A meromorphic function on $\overline{A}_{\bg,\bZ,\bP}^R$}\label{ssec:function}

Let $1\leq j\leq q$ and $1\leq i\leq n_j$.  Let $k_{i,j}$ be the $i^{\rm th}$ entry of $Z_j$. We consider the line bundle:
$$
\mathcal{O}(-1)\otimes\mathcal{L}_{j,i}^{k_{j,i}+1}\bigg|_{A_{\mathbf{g},\mathbf{P},\mathbf{Z}}^R}\simeq {\rm{Hom}}\left(\mathcal{O}(-1),\mathcal{L}_{j,i}^{k_{j,i}+1}\right)\bigg|_{A_{\mathbf{g},\mathbf{P},\mathbf{Z}}^R},
$$
where $\mathcal{L}_{j,i}$ is the cotangent line bundle to the $i$-th marked point of $j$-th connected component. Let $s_{j,i}$ be the holomorphic section of the line bundle ${\rm{Hom}}(\mathcal{O}(-1),\mathcal{L}_{j,i}^{k_{j,i}+1})_{|A_{\mathbf{g},\mathbf{P},\mathbf{Z}}^R}$ that maps a differential to its $(k_{j,i}+1)$-st order term at the $i$th marked point of the $j$th connected component.

\begin{mylem}\label{lem:multiplicity1}
The section $s_{j,i}$ vanishes with multiplicity $1$ along $\P\overline{A}^R_{\mathbf{g},\mathbf{Z}_{j,i},\mathbf{P}}$. 
\end{mylem}

\begin{proof}
Let $y_0=(C,\alpha, Z(y_0)\cup P(y_0))$ be a point of $ A_{\mathbf{g},\mathbf{Z}_{j,i},\mathbf{P}}^{R}$ where we denote by $P(y_0)\subset C$ be the set of poles of $\alpha$ and $Z(y_0)\subset C$ be the set of marked zeros of $C$. Besides we denote by $Z'(y_0)\subset C$ be the set of non-marked zeros. 

Let $W/{\rm Aut}(y_0)$ be a contractible neighborhood of $y_0$. Up to a choice a smaller $W$, in the proof of Lemma~\ref{deformation}, we constructed the 3 following maps:
\begin{eqnarray*}
\Phi^1: W &\to& H^1(C\setminus P(y_0), Z(y_0)\cup Z'(y_0), \C),\\
\Phi^{2,x} : W &\to& \oZ^{k_x} \text{ for all $x\in Z(y_0)$,}  \\ 
\Phi^{3,x} : W &\to& \widetilde{\oZ}^{k_x-1} \text{ for all $x\in Z'(y_0)$,}   
\end{eqnarray*}
where $k_x$ is the order of $\alpha$ at $x$ (be it a marked or non-marked zero) and where $\oZ^k$ is a domain in $\C^k$ containing $0$. These maps are not uniquely determined, however, we saw in the proof of Lemma~\ref{deformation} that the map 
$\Phi_1 \times \prod_{x\in Z(y_0)} \Phi^{2,x} \times \prod_{x\in Z'(y_0)} \Phi^{3,x}$ is a local biholomorphism.  

Now we consider the marked point $x_{j,i}$. We denote
$$\Phi^{(j,i)}= \Phi^{2,x_{j,i}}, \text{ and } \widehat{\Phi}^{(j,i)}=\prod_{x\in Z_{y_0}\setminus \{x_{j,i}\}}  \Phi^{2,x}.$$ 
We recall that the map $\Phi^{(j,i)}$ is defined as follows: for all points $s$ in a neighborhood of  $y_0$, the differential representing $y$ is given in neighborhood of the marked point $x_{j,i}(y)$ by 
$$
\alpha=\left( z^{k_{j,i}+1} + a_{k_{i,j}} z^{k_{j,i}}+ \ldots +a_0\right) dz
$$
(the marked point being at $z=0$), then we define $\Phi_2(y)=(a_0,\ldots, a_{k_{i,j}})\in \oZ^{k_{j,i}+1}$ (this definition is unique up to choice of $(k_{j,i}+2)$-nd root of unity).

Then with this parametrization we have 
\begin{eqnarray*}
W\cap A_{\mathbf{g},\mathbf{Z}_{j,i},\mathbf{P}}^{R}&=& \left(\Phi^{j,i}\times \widehat{\Phi}^{j,i}\right)^{-1} \left(  \prod_{x\in Z(y_0)} \{0\} \right), \\
W\cap A_{\mathbf{g},\mathbf{Z},\mathbf{P}}^{R}&=& \left(\Phi^{j,i}\times \widehat{\Phi}^{j,i}\right)^{-1} \left( (0,\ldots, 0,\epsilon) \times \!\!\!\!\!\!\!\prod_{x\in Z(y_0)\setminus \{x_{j,i}\} } \!\!\!\!\!\!\! \{0\} \right).
\end{eqnarray*}
 In other words, the coordinate $a_{k_{j,i}}$ is a transverse parameter to the divisor $A_{\mathbf{g},\mathbf{Z}_{j,i},\mathbf{P}}^{R}$ in $A_{\mathbf{g},\mathbf{Z},\mathbf{P}}^{R}$. We obviously have $s_{j,i}(y)=a_{k_{j,i}}$. Therefore the vanishing order of $s_{j,i}$ along $\P A_{\mathbf{g},\mathbf{Z}_{j,i},\mathbf{P}}^{R}$ is equal to $1$.
\end{proof}

\begin{mynot} We denote by $\bic(\mathbf{g},\mathbf{P},\mathbf{Z},R)_{j,i}\subset \bic(\mathbf{g},\mathbf{P},\mathbf{Z},R)$ the subset of bi-colored graphs such that the $i^{\rm th}$ marked point of the $j^{\rm th}$ connected component belongs to a level -1 vertex and we denote 
$
\D(\mathbf{g},\mathbf{P},\mathbf{Z},R)_{j,i}$ the intersection of $\bic(\mathbf{g},\mathbf{P},\mathbf{Z},R)_{j,i}$ and  $\D(\mathbf{g},\mathbf{P},\mathbf{Z},R)$.
\end{mynot}

\begin{mylem}\label{lem:vanishinglocus}
The divisors contained in the vanishing locus of $s_{j,i}$ are exactly the divisors corresponding to admissible graphs in $\D(\mathbf{g},\mathbf{P},\mathbf{Z},R)_{j,i}$ and the divisor $\P\overline{A}^R_{\mathbf{g},\mathbf{Z}_{j,i},\mathbf{P}}$. No two of these divisors have a common irreducible component.
\end{mylem}

\begin{proof} 
It is a consequence of Propositions~\ref{listdivisors} and~\ref{intersectiondiv}.
\end{proof}

\subsection{Induction formula}\label{ssec:induction}

Let $(\bg, \bZ,\bP,R\subset \oR)$ be a quadruple satisfying Assumption~\ref{assumption}. Let $1\leq j\leq q$ and $1\leq i\leq n_j$. We recall that we denote by $\mathbf{Z}_{j,i}$ the list of vectors obtained from $\mathbf{Z}$ by increasing $k_{j,i}$ by~1. Besides, as in the previous section, we denote by $\mathcal{L}_{j,i}$ the cotangent line to the $i$th marked point on the $j$th connected component of the curve and by  $\psi_{j,i}= c_1(\mathcal{L}_{j,i}) \in H^2(\P\oH_{\mathbf{g},\mathbf{n},\mathbf{P}},\Q)$.

\subsubsection{Multiplicity of $(\Gamma,I)$}

\begin{mydef}\label{def:multiplicity} Let $(\Gamma,I)\in \bic(\mathbf{g},\mathbf{P},\mathbf{Z},R)$. The \textit{multiplicity} of $(\Gamma,I)$ is defined as
\begin{equation*}
m(I)=\prod_{h \to V^0}  I(h),
\end{equation*}
where the product runs over the half-edges which are not legs, pointing to vertices of level 0. The {\em least common multiple} and the {\em group of roots} of the twist are
\begin{eqnarray*}
L(I)&=&{\rm{LCM}}\left(\{I(h)\}_{h \to V^0}\right),\\
G_{I}&=&\left(\prod_{h \to V^0}\mathbb{Z}_{I(h)}\right)\bigg/\Z_{L(I)}.
\end{eqnarray*}
\end{mydef}

\subsubsection{Locus of generic points.} Let $(\Gamma,I) \in \D(\mathbf{g},\mathbf{Z},\mathbf{P},R)$.  We recall that
\begin{equation*}
A_{\Gamma,I}=p( A_{\mathbf{g}_1,\mathbf{Z}_1,\mathbf{P}_1}^{R^1}) \times A_{\mathbf{g}_0,\mathbf{Z}_0,\mathbf{P}_0}^{R^{0}},
\end{equation*}
where $p:A_{\mathbf{g}_1,\mathbf{Z}_1,\mathbf{P}_1}^{R^1} \to \oM_{\mathbf{g}_1,\mathbf{n}_1,\mathbf{m}_1}$ is the forgetful map. The condition $(\star \star)$ ensures that there exists an open dense locus $A_{1}^{\rm gen}\subset A_{\mathbf{g}_1,\mathbf{Z}_1,\mathbf{P}_1}^{R^1}$ such that the map $p:A_{1}^{\rm gen} \to p(A_{1}^{\rm gen})$ has fibers of dimension 1 (see Proposition~\ref{pr:fibergen}). Then we set
\begin{equation*}
A_{\Gamma,I}^{\rm{gen}}=A_1^{\rm{gen}}\times A_{\mathbf{g}_0,\mathbf{Z}_0,\mathbf{P}_0}^{R^{0}}.
\end{equation*}
This open locus of generic points will be important for us because the map 
$$p: A_1^{\rm gen} \times A_{\mathbf{g}_0,\mathbf{Z}_0,\mathbf{P}_0}^{R^{0}}\to A_{\Gamma,I}^{\rm gen}= p(A_1^{\rm gen})\times A_{\mathbf{g}_0,\mathbf{Z}_0,\mathbf{P}_0}^{R^{0}} $$
is a line bundle minus the zero section. 

\begin{mynot}\label{not:linegeneric}
We denote by $p:\oN_{\Gamma,I} \to A_{\Gamma,I}^{\rm gen}$ this line bundle.
\end{mynot} 

\subsubsection{Induction formula} We finally have all elements to state the main result of the paper.

\begin{myth}\label{ind} In $H^*(\P\oH_{\mathbf{g},\mathbf{n},\mathbf{P}},\Q)$ we have
\begin{equation}\label{eqn:ind1}
[\P\overline{A}_{\mathbf{g},\mathbf{Z}_{j,i},\mathbf{P}}^{R}] = (\xi+(k_{j,i}+1)\psi_{j,i}) \cdot [\P\overline{A}_{\mathbf{g},\mathbf{Z},\mathbf{P}}^{R}] \;  - \!\!\!\! \!\!\!  \sum_{(\Gamma,I) \in \D(\mathbf{g},\mathbf{P},\mathbf{Z},R)_{j,i}} \!\!\!\! \!\!\!  \!\!\!  m(I) \; [\P\overline{A}_{\Gamma,I}]
\end{equation}
if $2g_j -2 + n_j + m_j >0$, or 
\begin{equation}\label{eqn:ind2}
[\P\overline{A}_{\mathbf{g},\mathbf{Z}_{j,1},\mathbf{P}}^{R}] = \frac{p-k-2}{p-1} \xi \cdot [\P\overline{A}_{\mathbf{g},\mathbf{Z},\mathbf{P}}^{R}]
\end{equation}
 if $g_j=0$, $Z_j = (k)$, $P_j = (p)$.
\end{myth} 

\begin{proof}[Proof of~\eqref{eqn:ind1}] 
As in Section~\ref{ssec:function}, we consider the line bundle ${\rm{Hom}}(\mathcal{O}(-1),\mathcal{L}_{j,i}^{k_{j,i}+1}) \to \P \overline{A}_{\mathbf{g},\mathbf{Z},\mathbf{P}}^{R}$. Its first Chern class is equal to $\xi + (k_{j,i}+1) \psi_{j,i}$. Moreover, this line bundle has a global section $s_{j,i}$ which maps a differential to its $(k_{j,i}+1)^{\rm{st}}$-order term at the marked point $(j,i)$. In Lemma~\ref{lem:multiplicity1} we showed that $s_{j,i}$ vanishes along $\P \overline{A}_{\mathbf{g},\mathbf{Z}_{j,i},\mathbf{P}}^{R}$ with multiplicity~1. In Lemma~\ref{lem:vanishinglocus} we showed that the remaining vanishing loci of $s_{j,i}$ are supported on the $\P\overline{A}_{\Gamma,I}$ for $(\Gamma, I)$ of $\D(\mathbf{g},\mathbf{Z},\mathbf{P},R)_{j,i}$. Therefore we deduce that
\begin{equation*}
\left(\xi+(k_{j,i}+1)\psi_{j,i}\right) \cdot [\P\overline{A}_{\mathbf{g},\mathbf{Z},\mathbf{P}}^{R}] = [\P\overline{A}_{\mathbf{g},\mathbf{Z}_{j,i},\mathbf{P}}^{R}]  + \Z,
\end{equation*}
where $\Z$ is a cycle supported on the union of $\P \overline{A}_{\Gamma,I}$ for $(\Gamma,I)\in \D(\mathbf{g},\mathbf{P},\mathbf{Z},R)_{j,i}$. 

Now we claim that the vanishing order of $s_{j,i}$ along the locus $\P A_{\Gamma,I}$ is equal to $m(I)$ (see Definition~\ref{def:multiplicity}). Lemma~\ref{tech} below implies this statement and thus Equation~(\ref{eqn:ind1}).
\end{proof}

\begin{mylem}\label{tech}
Let $(\Gamma,I)$ be a divisor graph in  $\D(\mathbf{g},\mathbf{Z},\mathbf{P},R)_{j,i}$. Let $y_0 \in \P A_{\Gamma,I}^{\rm{gen}}$. Let $\Delta$ be an open disk in $\C$ containing 0 and parametrized by $\epsilon$. There exists an open neighborhood $U$ of $y_0$ in $\P A_{\Gamma,I}^{\rm{gen}}$ together with a map $\iota:U\times \Delta\times G_I \to \P \oH_{\mathbf{g},\mathbf{n},\mathbf{m},\mathbf{P}}$ satisfying:
 \begin{itemize}
 \item the restriction $\iota|_{U\times 0\times g}$ is the identity on $U$ for all $g\in G_I$;
 \item the image of the restriction $\iota|_{\epsilon \neq 0} $ lies in the open stratum $\P A_{\mathbf{g},\mathbf{Z},\mathbf{P}}^{R}$;
 \item for all $g\in G_I$, the section $s_{j,i}$ restricted to $\iota(U\times \Delta\times g)$ vanishes along $\iota(U\times 0 \times g)$ with multiplicity $L(I)$;
 \item the map $\iota:U\times \Delta\times G_I\to \P\overline{A}_{\mathbf{g},\mathbf{Z},\mathbf{P}}^R$ is a degree $1$ parametrization of a neighborhood of $U$ in $\P\overline{A}_{\mathbf{g},\mathbf{Z},\mathbf{P}}^R$.
 \end{itemize}
\end{mylem}

The proof of Theorem~\ref{ind} immediately follows from Lemma \ref{tech} because the vanishing order of $s_{j,i}$ along $\P \overline{A}_{\Gamma,I}$ is equal to 
$$
L(I)\cdot {\rm{Card}}\left(G_I\right) = m(I).
$$

\begin{proof}[Proof of Lemma \ref{tech}]  We prove the lemma in two steps: first we will prove the first three points of the lemma and then we will prove that $\iota$ is a parametrization of degree 1 of a neighborhood of $U$ in $A_{\mathbf{g},\mathbf{Z},\mathbf{P}}^R$. 

\subsubsection*{Proof of the first three points.} For the sake of clarity we will successively prove the first three points at three levels of generality: first for a divisor graph with one edge, then for divisor graph with $R^1=\{0\}$ and finally in full generality.

\subsubsection*{Bi-colored graph with one edge.} For the moment we place ourselves in the simplest case: $(\Gamma,I)$ is an admissible graph with two vertices, one at level 0 and one at level $-1$. We suppose that there is only one edge with a twist given by $k>0$. Let $y_0$ be a point of $\P A^{\rm gen}_{\Gamma,I}$. Let $U$ be an open neighborhood of $y_0$ in $\P A^{\rm gen}_{\Gamma,I}$. A point $y$ of $U$ is given by
\begin{equation*}
([C^0],[C^1],\overline{x}^0,\overline{x}^1,[\alpha^{0}]),
\end{equation*}
where $C^0$ and $C^1$ are the curves corresponding to the two vertices of the graph; $\overline{x}^0$ and $\overline{x}^1$ are their marked point sets; $\alpha^0$ is a differential on the curve~$C^0$ and $[\alpha^0]$ its equivalence class under the $\C^*$-action. More precisely, we denote by $\alpha^0(y)$ a nonvanishing section of the line bundle $\O(-1)$ over~$U$. (Also recall that on $C^1$ the differential vanishes identically.)

The condition that $y \in A^{\rm gen}_{\Gamma,I}$ implies that the curve $C^1$ carries a {\em unique} meromorphic differential $\alpha^1$ with zeros and poles of prescribed multiplicities at the marked points, up to a scalar factor. Let $\alpha^1(y)$ be a nonvanishing section of the line bundle $\oN_{\Gamma,I}$, i.e., a choice of the scalar factor for each point $y$.

At the neighborhood of the node, the curves $C^1$ and $C^0$ have standard coordinates $z$ and $w$ such that $\alpha^0=d(z^k)$ and $\alpha^1=d(\frac{1}{w^{k}})$. The local coordinates $z$ and $w$ are unique up to the multiplication by a $k^{th}$ root of unity. We fix one such choice in a uniform way over~$U$. We define a family of curves $C(y,\epsilon)$ over $U\times \Delta$ by smoothing the node between $C^0$ and $C^1$ via the equation $zw=\epsilon$, where $\epsilon$ is the coordinate on the disc $\Delta$ and $z,w$ are as above. The differentials $\alpha^0$ and $\epsilon^k \alpha^1$ automatically glue together into a differential on $C(y,\epsilon)$.

The deformation that we have constructed does not depend on the choice of standard coordinates $z$ and $w$. For instance, if we multiply $z$ by a $k$th root of unity $\zeta$, the equation of the deformation becomes $zw=\zeta \epsilon$, which is isomorphic to the original deformation under a rotation of the disc $\Delta$. 

The section $s_{j,i}$  vanishes with multiplicity $k$ along the locus defined by $\epsilon=0$: indeed we have explicitly
\begin{equation*}
s_{j,i}(y,\epsilon)=\epsilon^k \cdot \alpha_1(y).
\end{equation*}

\subsubsection*{Bi-colored graph $(\Gamma,I)$ with $R^1=\{0\}$.}  We suppose now that the space $R^1$ is trivial (residues at the nodes between vertices of level 0 and -1 are equal to 0). A point $y$ in $U$ still determines
$$
([C^0],[C^1],\overline{x}^0,\overline{x}^1,[\alpha^0], [\alpha^1])
$$
where $\alpha^0$ and $\alpha^1$ are sections of $\O(-1)$ and $\mathcal{N}_{\Gamma,I}$ as in the previous paragraph.

 Let $e$ be an edge of $\Gamma$. We denote by $k_e$ the positive integer equal to $|I(h)|$ for any of the two half-edges of $e$. Let $z_e$ and $w_e$ be choices of standard coordinates in a neighborhood of the node corresponding to $e$: i.e. $\alpha^0=d(z_e^{k_e})$ and $\alpha^{1}=d(1/w_e^{k_e})$.  This choice of standard coordinates being fixed for all edges, we choose, on top of that, $\zeta_e$ a $k_e$-th root of unity for each edge~$e$.  

We define a family of curves $C(y,\epsilon)$ over $U\times \Delta$ by smoothing the node corresponding to an edge $e$ of $\Gamma$ via the equation $z_ew_e=(\zeta_e \epsilon)^{L(I)/k_e}$ where $\epsilon$ is the coordinate on the disc $\Delta$. The differentials defined by $\alpha^0$ and by $\epsilon^{L(I)} \alpha_1$ automatically glue together into a differential on $C(y,\epsilon)$.

A multiplication of $\epsilon$ by a $L(I)$-th root of unity $\zeta$ gives an isomorphic deformation. Thus two choices of roots $(\zeta_e)_{e\in \rm Edges}$ and $(\zeta_e')_{e\in \rm Edges}$ give isomorphic deformation if  $\zeta_e'=\zeta^{L(I)/k_e} \zeta_e$ for all edges. The vanishing multiplicity of $s_{j,i}$ along the locus defined by $\epsilon=0$ is equal to $L(I)$.

\subsubsection*{General bi-colored graph $(\Gamma,I)$.}  We no longer impose restrictions on $R^1$. We still define
\begin{equation*}
([C^0],[C^1],\overline{x}^0,\overline{x}^1,\alpha^{0},\alpha^1),
\end{equation*}
as above. Moreover we define the section $r$ 
\begin{equation*}
r(y) = (r_e(y))_{e\in \rm Edges},
\end{equation*}
where $r_e(y)$ is the residue of $\alpha_1$ at the node of $C^1$ corresponding to the edge $e$. For every edge $e$, we fix a choice of standard coordinates of  $z_e$ and $w_e$ in a neighborhood of the node corresponding to $e$, i.e., coordinates satisfying $\alpha^0=d(z_e^{k_e})$ and $\alpha^1=d(1/w_e^{k_e})+ \frac{r_e(y)dw_e}{w_e}$.

Using Proposition~\ref{pr:standardneihbor}, we get a family of differentials $(\widetilde{C}^0,\overline{x}^0,\widetilde{\alpha}^0)$ parametrized by $U\times \Delta$ such that:
\begin{itemize}
\item when $\epsilon=0$, we have $(C^0,\overline{x}^0,\alpha^0)=(\widetilde{C}^0,\overline{x}^0,\widetilde{\alpha}^0)$;
\item the zeros of the differential which are not at the marked points corresponding to nodes are of fixed orders;
\item the differential $\widetilde{\alpha}^0$ has at most simple poles at the nodes of $\widetilde{C}^0$ and the residue at the node corresponding to the edge~$e$ is equal to $-\epsilon^{L(I)} r_e(y)$;
\item the vector of residues at the poles of $\widetilde{\alpha}^0$ lies in $R$;
\item for each node corresponding to an edge $e$ with a twist $k_e$, the family of differentials defined by $U\times \Delta$ is a standard deformation of $d(z_e^{k_e})$ (see Definition~\ref{def:standdef}).
\end{itemize}
We use the fact that the family parametrized by $U\times \Delta$ is a standard deformation of $d(z_e^{k_e})$ to  apply Proposition~\ref{pr:annulus}. At each node $e$ the differential $\widetilde{\alpha}_0$ can be written in the form $d(z_e^{k_e})-\epsilon^{L(I)}r(u)\frac{dz_e}{z_e}$ in any annulus contained in a neighborhood of the node. Therefore we can still glue the two components together along this annulus with the identification $z_ew_e=\zeta_e\epsilon^{L(I)/k_e}$ for any choice of the $k_e$-th root of unity $\zeta_e$. The end of the proof is the same as for divisor graphs with trivial residue conditions.

 \subsubsection*{Proof of the fourth point.} Now we will prove that the map $\iota:U\times \Delta\times G_I\to \P\overline{A}_{\mathbf{g},\mathbf{Z},\mathbf{P}}^R$ is a degree $1$ parametrization of a neighborhood of $U$ in $\P\overline{A}_{\mathbf{g},\mathbf{Z},\mathbf{P}}^R$. 
 
First we prove that the image $\iota(U\times \Delta\times G_I)$ covers {\em entirely} a neighborhood of $U$ in $A_{\mathbf{g},\mathbf{Z},\mathbf{P}}^R$. Let $y_0=(C=C_0 \cup C_1,\overline{x}_0,\overline{x}_1,\alpha_0)$ be a point in $A_{\Gamma,I}^{\rm gen}$. Let $\widetilde{\iota}:\Delta\to \overline{A}_{g,Z,P}$ be a family of differentials such that $\widetilde{\iota}(0)=y_0$ and  $\widetilde{\iota}(\epsilon)\in A_{\mathbf{g},\mathbf{Z},\mathbf{P}}^R$ for $\epsilon\neq 0$. We denote by $\pi:\mathcal{C}\to \Delta$ the induced family of curves and by $\alpha$ the induced family of differentials on the fibers of $\mathcal{C}\to \Delta$. 

Let $e$ be a node of $C$ with a twist of order $k_e$. Let $\gamma_e$ be a simple loop in the curve $C_0$ around the node $e$. Let $W_e$ be a neighborhood of $\gamma_e$ in $\mathcal{C}$ such that $W_e\cap \pi^{-1}(\epsilon)$ is an annulus for any $\epsilon$ small enough. Now, the differential $\alpha_0$ is given by $d(z_e^{k_e})$ in a standard coordinate. Thus the differential $\alpha|_{\pi^{-1}(\epsilon)}$ is given by $d(z_e^{k_e})+\phi(\epsilon,z_e) dz_e $ and we denote by $r_e(\epsilon)$ the integral of $\phi(\epsilon,z_e)dz_e$ along $\gamma_e$. We consider the differential $\alpha_e(\epsilon)=dz_e+ \phi(\epsilon,z_e) dz_e- r_e(\epsilon) \frac{dz_e}{z_e}$. We fix a point $p$ in the annulus  $W_e\cap \pi^{-1}(\epsilon)$, the function $f:z\mapsto (\int_{p}^{z} \alpha_e)^{1/k_e}$ is uniquely determined for small values of $\epsilon$. This determines a coordinate (that we will still denote $z_e$) such that $\alpha_0={z_e}^{k_e}dz_e- \varphi(\epsilon,z_e)\frac{dz_e}{z_e}$ with $\varphi$ holomorphic and thus a standard deformation of $\alpha^0$.  Proposition~\ref{pr:annulus} implies that there exists a coordinate $z_e$ on this annulus such that $\alpha|_{\pi^{-1}(\epsilon)}=d(z_e^{k_e})+r_e(\epsilon) \frac{dz_e}{z_e}$.

We fix $\epsilon$ small enough so that the coordinates $z_e$ are defined for all edges $e$. We cut the curve $\pi^{-1}(\epsilon)$ along simple loops contained in $W_e$. This gives two (possibly disconnected) curves with boundary $C^{\rm open}_0$ and $C^{\rm open}_1$. We ``plug'' the holes of $C^{\rm open}_0$ with disks parametrized by the coordinate $z_e$ and the holes of  $C^{\rm open}_1$ with disks with coordinate $1/z_e$. This determines two curves $C_0(\epsilon)$ and $C_1(\epsilon)$. On both $C_0$ and $C_1$, the local chart used to ``plug'' the holes allow us to define differentials $\alpha_0(\epsilon)$ and $\alpha_1(\epsilon)$. 

The differential $\alpha_1(\epsilon)$ has a pole of order $k_e+1$ at $w_e=0$; thus $(C_1,\overline{x}_1,\alpha_1)(\epsilon)$ is an element of $A_{\mathbf{g}_1,\mathbf{Z}_1,\mathbf{P}_1}^{R^1}$. Now, at the level 0, we use Proposition~\ref{pr:retraction}: in a neighborhood of $y_0$ we can apply the retraction $\eta$. The point $\eta((C_0,\overline{x}_0,\alpha_0)(\epsilon))$ is a point of $A_{\mathbf{g}_0,\mathbf{Z}_0,\mathbf{P}_0}^{R^0}$. Therefore we define 
$$
y(\epsilon)=(\eta(C_0,\overline{x}_0,\alpha_0),(C_1,\overline{x}_1,\alpha_1))(\epsilon)\in A_{\Gamma,I}^{\rm gen}.
$$
For all $\epsilon$ in a neighborhood of $0$, the point $\widetilde{\iota}(\epsilon)$ lies in the deformation of $y(\epsilon)$ by the family $\iota$ restricted to $y(\epsilon)\times \Delta \times g$ for some $g\in G_I$ (in fact here $g=1$ because of the choices of the parameters around $y_0$ that we have fixed).

To finish the proof of the fourth point, we need to prove that the parametrization is of degree $1$. For this, we once again use the retraction $\eta$ defined in Proposition~\ref{pr:retraction}. We have $\eta \circ \iota= {\rm Id}_U$, thus we only need to prove that for all $y \in U$, the family $\iota$ restricted to $y\times \Delta \times G_I$ is of degree $1$. We consider this family in the moduli space of curves, i.e let 
\begin{eqnarray*}
\iota':\Delta \times G_I &\to & \oM_{\mathbf{g},\mathbf{n},\mathbf{m}}\\
\epsilon\times g&\mapsto& p(\iota(y,\epsilon,G_I)).
\end{eqnarray*} 
This family is of degree one. Indeed the stack $\oM_{\Gamma}$ is regularly imbedded  in $\oM_{\mathbf{g},\mathbf{n},\mathbf{m}}$ and its normal bundle is the direct sum of the $T_h\otimes T_{h'}$ for all edges $e=(h,h')$ of $\Gamma$. Thus the family $\iota'$ is given by the family:
\begin{eqnarray*}
\iota': \Delta \times G_I &\to& \bigoplus_{(h,h')\in \rm Edges} T_h\otimes T_{h'}\\
\left(\epsilon, (\zeta_e)_{e \in \rm Edges}\right)&\mapsto& \left(\zeta_e\epsilon^{L(I)/k_e}\right)_{e\in \rm Edges},
\end{eqnarray*}
which is of degree 1.
\end{proof}

\begin{proof}[Proof of Formula~(\ref{eqn:ind2})]
We have seen that the space of differentials on an unstable component is a weighted projective space parametrized by
\begin{equation*}
\left[ w^{p-1}+ a_1 w^{p-2} + \ldots + a_{p-2}w \right]  \frac{dw}{w},
\end{equation*}
where the weight of $a_j$ is $\frac{j}{p-1}$. The fact that the order of the point $x$ is $k_{j,i}$ is equivalent to the vanishing of the terms $a_{p-2},\ldots,a_{p-k_{j,i}-3}$. Therefore, the class of $[\P\overline{A}_{\mathbf{g},\mathbf{Z}_{j,i},\mathbf{P}}^{R}]$ is the closure of the vanishing locus $a_{p-k_{j,i}-2}$. Moreover we can easily check that $a_{p-k_{j,i}+1}^{p-1}$ is a global section of $\O(-1)^{p-k_{j,i}+1}$.
\end{proof}

\subsection{Class of a boundary divisor}\label{ssec:class}

Let $(\bg, \bZ,\bP,R\subset \oR)$ be a quadruple satisfying Assumption~\ref{assumption}.  We want to compute the Poincar\'e-dual class of the locus associated to an element of $\D(\mathbf{g},\mathbf{P},\mathbf{Z},R)$. 

\subsubsection{Decomposition of the morphism $A_{\Gamma,I}\to \oH_{\bg,\bZ,\bP}$} 
Let $(\Gamma,I)$ be an admissible graph in $\bic(\mathbf{g},\mathbf{P},\mathbf{Z},R)$ (this graph may be a divisor or not). 
We recall that the semi-stable graph $\Gamma$ determines a stratum 
$$
\zeta_\Gamma^{\#}:\oH_{\Gamma}=\oH^{R_\Gamma}_{\mathbf{g}_\Gamma,\mathbf{n}_\Gamma,\mathbf{P}_\Gamma}\to \oH_{\mathbf{g},\mathbf{n},\mathbf{P}}
$$
(see Section~\ref{ssec:semi-stable}). 

\begin{remark} Beware that $\oH_{\Gamma}$ is purely determined by $\Gamma$ (and not $I$). However the twist $I$ determines the components that are of level $-1$ and the $\bP$-admissibility condition implies that these component do not carry marked poles.  Therefore the poles on these components are of order at most $-1$ and only at the marked points that will be mapped to the branches of nodes.
\end{remark}

We define the linear subspace $R^\dagger_{\Gamma} \subset R_\Gamma$, as the space of vectors in $R_\Gamma$ defined by the condition: all residues of poles on components of level $-1$ (or equivalently at the nodes) vanish.

Now, we have seen that $(\Gamma,I)$ defines the space
$$
A_{\Gamma,I}\simeq A_{\mathbf{g}_0,\mathbf{Z}_0,\mathbf{P}_0}^{R^{0}} \times  p(A_{\mathbf{g}_1,\mathbf{Z}_1,\mathbf{P}_1}^{R^1})\hookrightarrow \oH_{\mathbf{g}_0,\mathbf{n}_0,\mathbf{P}_0} \times  \oM_{\mathbf{g}_1,\mathbf{n}_1,\mathbf{m}_1}
$$
(see Notation~\ref{not:spacegraph} for the definitions of $\bg_i,\bZ_i,\bP_i,$ and $R^i$). We denote by $\widetilde{A}_{\Gamma,I}$ the space on the right hand side. 

With this notation, we have the following isomorphism
$$
\oH^{R^\dagger_{\Gamma}}_{\mathbf{g}_\Gamma,\mathbf{n}_\Gamma,\mathbf{m}_\Gamma,\mathbf{P}_\Gamma}\simeq \oH_{\mathbf{g}_0,\mathbf{n}_0,\mathbf{P}_0} \times \left( \prod_{v \in V^1} \oH_{g_v,n_v+m_v} \right),
$$
where in the second product, $n_v$ and $m_v$ are the length of $Z_v$ and $P_v$ respectively and we recall that $p_v:\oH_{g_v,n_v+m_v}\to \oM_{g_v,n_v+m_v}$ is the Hodge bundle. Indeed a differential in $\oH^{R^\dagger_{\Gamma}}_{\mathbf{g}_\Gamma,\mathbf{n}_\Gamma,\mathbf{m}_\Gamma,\mathbf{P}_\Gamma}$ is a differential on the normalization of a curve in $\zeta_\Gamma^{\#}(\oH_\Gamma)$ such that the differential has no residues (thus no poles) at the branches of a node. Therefore the restriction of this differential to components of level $0$ is a point in $\oH_{\bg_0,\bn_0,\bP_0}$ (without residue condition at the marked poles), and its restriction  to the component of level $-1$ is an holomorphic differential (thus a point in the product of the Hodge bundles). 

All in all we have the following sequence of embeddings:
$$
{A}_{\Gamma,I} \hookrightarrow \widetilde{A}_{\Gamma,I} \hookrightarrow \oH^{R^\dagger_{\Gamma}}_{\mathbf{g}_\Gamma,\mathbf{n}_\Gamma,\mathbf{m}_\Gamma,\mathbf{P}_\Gamma}\hookrightarrow \oH_\Gamma,
$$
where the second one is given by the zero section embedding of $\oM_{\bg_1,\bn_1,\bm_1}$ in the Hodge bundle. All these embeddings are compatible with the $\C^*$-action therefore we get the sequence of embeddings
$$
\P A_{\Gamma,I} \hookrightarrow \P\widetilde{A}_{\Gamma,I} \hookrightarrow \P\oH^{R^\dagger_{\Gamma}}_{\mathbf{g}_\Gamma,\mathbf{n}_\Gamma,\mathbf{m}_\Gamma,\mathbf{P}_\Gamma}\hookrightarrow \P\oH_\Gamma.
$$
From here, we will compute the Poincar\'e-dual cohomology class of $\P\overline{A}_{\Gamma,I}$ in $H^*(\P\oH_\Gamma,\Q)$  by computing successively the class of each of these sub-stacks in $H^*(\P\oH_\Gamma,\Q)$.

\subsubsection{The Poincar\'e-dual class of $\P\widetilde{A}_{\Gamma,I}$}

We denote by
$$
d_{\Gamma}={\rm dim} (R_{\Gamma})-{\rm dim} (R^\dagger_{\Gamma}).
$$
The Poincar\'e-dual class of $\P\oH^{R^\dagger_{\Gamma}}_{\mathbf{g}_\Gamma,\mathbf{n}_\Gamma,\mathbf{P}_\Gamma}$ in $H^*(\P\oH_\Gamma,\Q)$ is equal to $\xi^{d_{\Gamma}}$ (see Lemma~\ref{lemres}).

 Now we consider the morphism  $\psi: \P\oH^{R^\dagger_{\Gamma}}_{\mathbf{g}_\Gamma,\mathbf{n}_\Gamma,\mathbf{P}_\Gamma} \to \oM_{\bg_1,\bn_1,\bm_1}$ (it is the composition of a projection the forgetful map of the differential). The restriction of a differential differential in $\P\oH^{R^\dagger_{\Gamma}}_{\mathbf{g}_\Gamma,\mathbf{n}_\Gamma,\mathbf{P}_\Gamma}$ gives rise to morphism of vector bundles
\begin{center}
$
\xymatrix{
\Psi: \mathcal{O}(-1) \ar[r] \ar[d] & \psi^*\left(\bigoplus_{v \in V_1} \oH_{g_v,n_v+m_v} \right)\ar[ld] \\
\P\oH^{R^\dagger_{\Gamma}}_{\mathbf{g}_\Gamma,\mathbf{n}_\Gamma,\mathbf{m}_\Gamma,\mathbf{P}_\Gamma}.
}$
\end{center}
The morphism $\Psi$ can equivalently be seen as a section of $\mathcal{O}(1)\otimes \psi^*\left(\bigoplus_{v \in V_1} \oH_{g_v,n_v+m_v} \right)$.  The vanishing locus of $\Psi$ is the locus of differentials whose restriction to level -1 components is identically zero, i.e. $\P\widetilde{A}_{\Gamma,I}$ (with the reduced closed substack structure). The Poincar\'e-dual class of this locus in $H^*(\P\oH_\Gamma,\Q)$ is then given by
$$ \xi^{d_\Gamma}\cdot \prod_{v\in V_1} (\xi^{g_v}+\lambda_1 \xi^{g_v-1}+\ldots+ \lambda_{g_v}).$$

\subsubsection{The Poincar\'e-dual class of $\P\overline{A}_{\Gamma,I}$} We have the natural isomorphism:
$$
\P\widetilde{A}_{\Gamma,I}\simeq \P\oH_{\mathbf{g}_0,\mathbf{n}_0,\mathbf{P}_0} \times \oM_{\mathbf{g}_1,\mathbf{n}_1,\mathbf{m}_1}.
$$
We denote by $\Phi_0$ and $\Phi_1$ the projections on both factors.
\begin{mydef}
The class $a_{\Gamma,I} \in H^*(\P\oH_{\mathbf{g},\mathbf{n},\mathbf{P}},\Q)$ is defined by
$$
\frac{1}{|{\rm Aut}(\Gamma,I)|}
{\zeta^\#_{\Gamma}}_*\left(\xi^{d_{\Gamma}} \cdot \Phi_1^*(p_*[ \P\overline{A}_{\mathbf{g}_1,\mathbf{Z}_1,\mathbf{P}_1}^{R^1}]) \cdot \Phi_0^*[\P \overline{A}_{\mathbf{g}_0,\mathbf{Z}_0,\mathbf{P}_0}^{R^0}] \prod_{v\in {V^1}} \left(\xi^{g_v}+\lambda_1 \xi^{g_v-1}+\ldots+ \lambda_{g_v}\right)\right).
$$
where ${\rm Aut}(\Gamma,I)$ is the group of automorphism of $\Gamma$ preserving the twists at the edges.
\end{mydef}

\begin{mypr}\label{classboundary}
Let $(\Gamma,I)\in \bic(\mathbf{g},\mathbf{P},\mathbf{Z},R)$. We have:
\begin{enumerate}
\item if $(\Gamma,I)$ is divisor graph then $a_{\Gamma,I}=[\P \overline{A}_{\Gamma,I}]$;
\item If $(\Gamma,I)$ is not a divisor graph then $a_{\Gamma,I}=0$;
\item  if $[\P  \overline{A}_{\mathbf{g}_0,\mathbf{Z}_0,\mathbf{P}_0}^{R^0}]$ and $[\P\overline{A}_{\mathbf{g}_1,\mathbf{Z}_1,\mathbf{P}_1}^{R^1}]$ are tautological and can be explicitly computed then so is $a_{\Gamma,I}$. 
\end{enumerate}
\end{mypr}

\begin{proof}[Proof of the first and second points.]
If $(\Gamma,I)$ is a divisor graph then $p:\P A_{\mathbf{g}_1,\mathbf{Z}_1,\mathbf{P}_1}^{R^1}\to {\rm Im}(p)$ is of degree 1, thus $p_*[\P A_{\mathbf{g}_1,\mathbf{Z}_1,\mathbf{P}_1}^{R^1}]=[p(\P A_{\mathbf{g}_1,\mathbf{Z}_1,\mathbf{P}_1}^{R^1})]$. Therefore, by construction $a_{\Gamma,I}$ is the Poincar\'e-dual class of $\P \overline{A}_{\Gamma,I}$. 

If $(\Gamma,I)$ belongs to $\bic(\mathbf{g},\mathbf{P},\mathbf{Z},R)\setminus \D(\mathbf{g},\mathbf{P},\mathbf{Z},R)$ then the fibers of $p:\P A_{\mathbf{g}_1,\mathbf{Z}_1,\mathbf{P}_1}^{R^1} \to {\rm Im}(p)$ are of positive dimension and $p_*[ \P\overline{A}_{\mathbf{g}_1,\mathbf{Z}_1,\mathbf{P}_1}^{R^1}]=0$.
\end{proof}

\begin{proof}[Proof of the third point]  We assume that $[\P\overline{A}_{\mathbf{g}_0,\mathbf{Z}_0,\mathbf{P}_0}^{R^0}]$ and $[\P\overline{A}_{\mathbf{g}_1,\mathbf{Z}_1,\mathbf{P}_1}^{R^1}]$ are tautological and can be explicitly computed.

 The projections $\Phi_1$ is equal to the composition of the forgetful map from $\oH_{\Gamma}$ to $\oM^{\rm red}_{\Gamma}$ with the projection to the vertices of level $-1$. Thus by definition, if $\beta$ is a tautological class of $\oM_{\mathbf{g}_1,\mathbf{Z}_1,\mathbf{m}_1}$ then $\Phi_1^*\beta$ is a tautological class of $H^*(\P\oH_{\mathbf{g},\mathbf{Z},\mathbf{P}},\Q)$. Besides, if  $[\P\overline{A}_{\mathbf{g}_1,\mathbf{Z}_1,\mathbf{P}_1}^{R^1}]$ is tautological and be explicitly computed then so is $p_*[\P\overline{A}_{\mathbf{g}_1,\mathbf{Z}_1,\mathbf{P}_1}^{R^1}]$: indeed the Segre class of $\oH_{\mathbf{g}_1,\mathbf{n}_1,\mathbf{P}_1}$ is a tautological class of $\oM_{\mathbf{g}_1,\mathbf{m}_1,\mathbf{P}_1}$. 

 The map $\Phi_1$ is equivariant with respect to the $\C^*$-action, thus we have $\Phi_1^{-1}(c_1(\O(1))=c_1(\O(1)$. Besides the following diagram commutes:
$$
\xymatrix{
\P\tilde{A}_{\Gamma,I} \ar[r]^{\Phi_0} \ar[d] & \ar[d]_p \P\overline{A}_{\mathbf{g}_0,\mathbf{n}_0,\mathbf{P}_0}\\
\oM_{\Gamma}^{\rm red} \ar[r] & \oM^{\rm red}_{\mathbf{g}_0,\mathbf{n}_0,\mathbf{m}_0}
}
$$
Thus, if $\beta$ is a tautological class of $\oM^{\rm red}_{\mathbf{g}_0,\mathbf{n}_0,\mathbf{m}_0}$, then the class $\Phi_0^*(p^*(\beta))$ is a tautological class of $\P\oH_\Gamma$ and thus a tautological class of $H^*(\P\oH_{\mathbf{g},\mathbf{n},\mathbf{P}},\Q)$.
\end{proof}

We can already remark that Proposition~\ref{classboundary} implies the following 
\begin{mycor}\label{divbic}
The following equality holds:
$$
\sum_{(\Gamma,I) \in \bic(\mathbf{g},\mathbf{P},\mathbf{Z},R)_{j,i}} \!\!\!\! m(I) \; a_{\Gamma,I}=\!\!\!\! \sum_{(\Gamma,I) \in \D(\mathbf{g},\mathbf{P},\mathbf{Z},R)_{j,i}} \!\!\!\! m(I) \; a_{\Gamma,I}.
$$
\end{mycor}

\begin{proof}
It follows from the fact that if $(\Gamma,I) \in \bic(\mathbf{g},\mathbf{P},\mathbf{Z},R)_{j,i}\setminus\D(\mathbf{g},\mathbf{P},\mathbf{Z},R)_{j,i}$ then $a_{\Gamma,I}=0$.
\end{proof}

\subsection{Proof of Theorems~\ref{main}, \ref{mainbis}, and~\ref{mainter}}

We have all ingredients to prove Theorem~\ref{maingen} (see the beginning of the Section). 

\begin{proof}[Proof of Theorem~\ref{maingen}]
For a list $\bZ=(Z_1,\ldots, Z_q)$ of vectors of non-negative integers we denote $|\mathbf{Z}|=\sum_{j=1}^q |Z_j|$. We prove Theorem~\ref{maingen} by induction on $|\mathbf{Z}|$. 

\subsubsection*{Base of the induction: $|\mathbf{Z}|=0$.} Let $(\bg,\bZ,\bP,R)$ be a quadruple satisfying~\ref{assumption} and such that $|\mathbf{Z}|=0$. is trivial then $A^{R}_{\mathbf{g}, \mathbf{Z},\mathbf{P}}$ is dense in $\oH_{\mathbf{g}, \mathbf{n}, \mathbf{P}}$. Therefore 
$$[\P A^{R}_{\mathbf{g}, \mathbf{Z},\mathbf{P}}]=[\P\oH^R_{\mathbf{g}, \mathbf{n}, \mathbf{P}}]=\xi^{\dim(\oR)-\dim(R)},$$
 by Lemma~\ref{lemres}.

\subsubsection*{Induction.} Now, let $(\bg,\bZ,\bP,R)$ be a quadruple satisfying~\ref{assumption} and such that $|\mathbf{Z}|>0$. The induction Formulas~(\ref{eqn:ind1}) and ~(\ref{eqn:ind2}) of Theorems~\ref{ind} express the class  $\left[\P\overline{A}^R_{\mathbf{g}, \mathbf{Z}, \mathbf{P}}\right]$ in terms of a class with a smaller sum of the order of zeros and a sum over all bi-colored graph (by Corollary~\ref{divbic}). We only need to prove that the class $a_{\Gamma,I}$ is tautological for any $(\Gamma,I)\in \D(\mathbf{g}, \mathbf{Z}, \mathbf{P},R)$. 

The vectors of zeros $\mathbf{Z}_0$ and $\mathbf{Z}_1$ of the levels 0 and $-1$ satisfy $|\mathbf{Z}_i|<|\mathbf{Z}|$. Therefore the classes $[\P\overline{A}_{\mathbf{g}_1,\mathbf{Z}_1,\mathbf{P}_1}^{R^1}]$ and  $[ \P\overline{A}_{\mathbf{g}_0,\mathbf{Z}_0,\mathbf{P}_0}^{R^0}]$ can be computed and are tautological. Using Proposition~\ref{classboundary}, this implies that the class $a_{\Gamma,I}$ is tautological and can be computed.
\end{proof}

Theorems~\ref{main}, \ref{mainbis}, and~\ref{mainter} stated in Section~\ref{ssec:results} are straightforward corollaries of Theorem \ref{maingen}. 

\begin{proof}[Proof of Theorems~\ref{main}, \ref{mainbis}, and~\ref{mainter}] Theorem \ref{main} is the special case of Theorem \ref{maingen} for a connected and stable curves. Theorem \ref{mainter} is a consequence of \ref{main} and Proposition \ref{cone} (the Segre class of the spaces of stable differential is tautological).

To prove Theorem \ref{mainbis}, we recall that we denote by $\widetilde{\pi}_n: \P\oH_{g,n} \to \P\oH_{g}$, the forgetful map of points.  The bundle $\oH_{g,n}$ is the pull-back of $\oH_g$ by $\pi_n$, then $\xi\in H^*(\P\oH_{g,n},\Q)$ is the pull-back of $\xi\in H^*(\P\oH_{g},\Q)$. Therefore the push-forward of a tautological class of $RH^*(\P\oH_{g,n},\Q)$ by $\pi_{n}$ is in $RH^*(\P\oH_{g},\Q)$ and can be explicitly computed.

If $Z=(k_1,\ldots,k_n)$ is complete, the map $\widetilde{\pi}_n$ restricted to $\P A_{g,Z}$ is finite of degree ${\rm{Aut}}(Z)$ onto $\P\mathcal{H}[Z]$. We have 
\begin{equation*}
\left[\P\oH[Z]\right]= \frac{1}{{\rm{Aut}}(Z)} \cdot \widetilde{\pi}_{n*} \left[ \P \overline{A}_{g,Z}\right],
\end{equation*}
and the class $\left[\P\oH[Z]\right]$ is tautological and can be explicitly computed.
\end{proof}

\section{Examples of computation}\label{sec:compu}

We give two examples of computation: the first one is a computation in the projectivize Hodge bundle (we forget the marked points), the second is a computation in the moduli space of curves (we forget the differential).

\subsection{The class $[\P\oH_{g}(3)]$}\label{ssec:compu1}

We consider here $g>2$ and $Z=(3,1\ldots,1)$. We have seen in the introduction the computation of $[\P A_{g,(2)}]$. Therefore, in order to compute $[\P A_{g,(3)}]$ we need to  list the divisor graphs contributing to $[\P\overline{A}_{g,(3)}]- (\xi+3\psi_1) [\P\overline{A}_{g,(2)}]$.

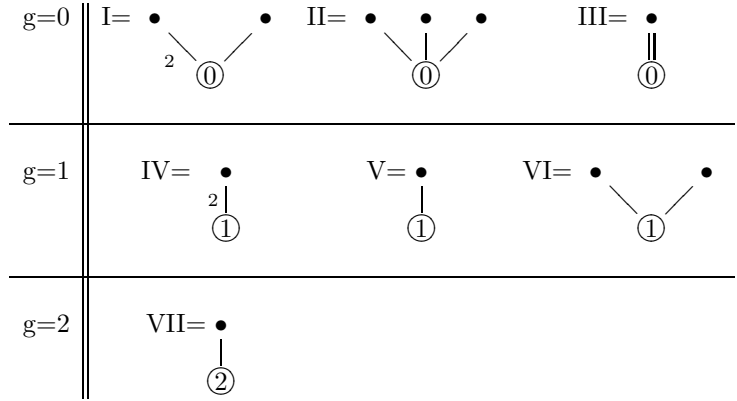
\begin{figure}[h!]
\centering
\begin{tabular}{c||ccc}
g=0 &
I=  $\xymatrix@=1em{
{\bullet} \ar@{-}[rd]_{2}&&{\bullet}\ar@{-}[ld] \\ 
&*+[Fo]{0} &
}$
&
II= $\xymatrix@=1em{
{\bullet} \ar@{-}[rd]_{}&{\bullet}\ar@{-}[d]&{\bullet}\ar@{-}[ld] \\ 
&*+[Fo]{0} &
}$
&
III=  $\xymatrix@=1em{
{\bullet}\ar@{=}[d] \\ 
*+[Fo]{0} 
}$
\\
\\
\hline
\\
g=1&
IV= $\xymatrix@=1em{
{\bullet}\ar@{-}[d]_2 \\ 
*+[Fo]{1} 
}$ &
V=$\xymatrix@=1em{
{\bullet}\ar@{-}[d] \\ 
*+[Fo]{1} 
}$&
VI= $\xymatrix@=1em{
{\bullet} \ar@{-}[rd]_{}&&{\bullet}\ar@{-}[ld] \\ 
&*+[Fo]{1} &
}$
\\
\\
\hline
\\
g=2&
VII=$\xymatrix@=1em{
{\bullet}\ar@{-}[d] \\ 
*+[Fo]{2} 
}$
&
\end{tabular}
\caption{List of boundary terms in $[\P\overline{A}_{g,(3)}]- (\xi+3\psi_1) [\P\overline{A}_{g,(2)}]$.}\label{list3}
\end{figure}
We have represented vertices of level -1 with their genera and the  vertices of level 0 by bullets (the sum will run over all possible distributions of the genera of  vertices of level 0). The marked point always belong to the unique vertex of level $-1$. The twists are represented by one number because the level structure already implies the sign of the twist on each half-edge. Finally we only represented the twists of absolute value greater than 1. 

After push-forward by the forgetful map of the marked point, we get the following formula for the class $[\P\oH(3,1,\ldots,1)]\in H^*(\P \oH_{g},\Q)$:
\begin{eqnarray*}
[\P\oH(3,1,\ldots,1)]&=&(12g-12)\ \xi^2+\left(11\kappa_1- \delta - \delta_{sep}- 5 \ \xymatrix@C=1em@R=0.5em{
*+[Fo]{1} \ar@{-}[r] & {\bullet}
}\right)\xi\\&+&\left(6\kappa_2-\xymatrix@C=1em@R=0.5em{
{\bullet}^{\psi_e} \ar@{-}[r] & {\bullet}
} -1/12 \; \; \; \;\;\;\xymatrix@C=1em@R=0.5em{
 *+[Fo]{0}  \ar@{-}@(ul,dl)[]  \ar@{-}[r] & {\bullet}
}\right).
\end{eqnarray*}

We explain the notation of the above expression. If the graph is not decorated, then the notation stands for the push forward of the fundamental class of $\oM_{\Gamma}$ under $\zeta_\gamma$. If a graph is decorated with classes $P_v$ in $\oM_{g(v),n(v)}$ for each vertex then the notation stands for $\zeta_{\gamma *}(\prod P_v)$. These classes are either $\psi_i$ for a marked point, $\psi_e$ for an half-edge or $\lambda_i$ and $\kappa_i$ for a vertex. In the above expression there is only one decoration $\psi_e$ on a half-edge.  

\begin{remark} For $g=3$, we can compute $p_*[\P\overline{A}_{3,(3)}]\in H^0(\oM_{3},\Q)\simeq \Q$, where $p$ is the forgetful map of the differential. We get $p_*[\P\overline{A}_{g,(3)}]=24$, the number of ordinary double points of a general quartic plane curve. In genus $3$, we can also compute $p_*(\pi_*[\P\overline{A}_{3,(2,2)}])=2\times 28$, i.e. two times the number of bi-tangents to a general quartic plane curve.
\end{remark}

\subsection{The class $[\oM_3(4)]$}\label{ssec:compu2}

Here $g=3$ and $Z=(4)$. We will compute the class $\oM_3(4)=\pi_*[\P\overline{A}_{3,(4)}]\in H^4(\oM_{3,1})$. We will not give the details of the computation however we have
\begin{eqnarray*}
[\oM_3(4)]&=&\lambda_2-10\psi_1\lambda_1+35\psi_1^2 
- 5 \; \xymatrix@C=1em@R=0.5 em{
 *+[Fo]{0} \ar@{-}[d] \ar@{=}[r] & *+[Fo]{2} - \\
&
}
\; - \; \xymatrix@C=1em@R=0.5em{
*+[Fo]{1} \ar@{-}[d] \ar@{=}[r] & *+[Fo]{1} - \\
&} 
+6 \; \xymatrix@C=1em@R=0.5em{
*+[Fo]{1} \ar@{-}[d] \ar@{-}[r] & *+[Fo]{1} - \ar@{-}[r] & *+[Fo]{1} - \\
&
}
\\
&+& \; \xymatrix@C=1em@R=0.5em{
*+[Fo]{1}  \ar@{-}[r] & *+[Fo]{1} \ar@{-}[d] \ar@{-}[r] & *+[Fo]{1} - \\
&
}
+6 \; \xymatrix@C=1em@R=0.5em{
*+[Fo]{1} \ar@{-}[d] \ar@{-}[r] & *+[Fo]{2}  \\
\lambda_1&
}
-34 \; \xymatrix@C=1em@R=0.5em{
*+[Fo]{1} \ar@{-}[d] \ar@{-}[r] & *+[Fo]{2}  \\
\psi_1&
}
-11 \; \xymatrix@C=1.6em@R=0.5em{
*+[Fo]{1}  \ar@{-}[d] \ar@{-}[r]^{\;\;\;\;\psi_e}  & *+[Fo]{2} \\
&
}\\
&+&  \; \xymatrix@C=1em@R=0.5em{
*+[Fo]{1}  \ar@{-}[r] & *+[Fo]{2} \ar@{-}[d] \\
\lambda_1&
}
-10 \; \xymatrix@C=1em@R=0.5em{
*+[Fo]{1}  \ar@{-}[r] & *+[Fo]{2}\ar@{-}[d]  \\
& \psi_1
}
- \; \xymatrix@C=1.6em@R=0.5em{
*+[Fo]{1} \ar@{-}[r]^{\;\;\;\;\psi_e}  & *+[Fo]{2} \ar@{-}[d] \\
&
}
\end{eqnarray*}
We explain the notation of the above expression. The legs on the graphs stands for the only marked point. We have decorated graph with classes $P_v$ in $\oM_{g(v),n(v)}$ for each vertex. These classes are either $\psi_1$ (for the marked point), $\psi_e$ for an half-edge or $\lambda_1$ for a vertex. 

We recall that $\H_3(4)$ has two connected components (hyperelliptic and odd). In this case one can compute $[\oM_3(4)^{\rm hyp}]$ by using the work of Faber and Pandharipande (see~\cite{FabPan}). This way one can also compute $[\oM_3(4)^{\rm odd}]=[\oM_3(4)]-[\oM_3(4)^{\rm hyp}]$. In general, it is possible to compute the class of the hyperelliptic component but we do not know how to compute separately the classes of odd and even components for $g\geq 4$.

Felix Janda has compared this expression with the expression of Conjecture B. The two expressions agree modulo tautological relations (see Section~\ref{ssec:appli} for presentation of the conjecture).

If we forget the marked point, then we get a class in ${\rm Pic}(\oM_g)\otimes \Q$. Using the string and dilaton equations and Mumford's formula for $\kappa_1$ we get
\begin{eqnarray*}
\pi_*[\oM_3(4)]&=&0-10\times 4  \;  \lambda_1 + 35 \;  \kappa_1 
- 5 \; \delta_{\rm nonsep}
\; - \; 0
+6 \cdot 0 \; 
\\
&+& \; 0
+6 \; \cdot 0
-34   \;  \delta_{\rm sep}
-11 \;  \delta_{\rm sep}\\
&+& 0
-10 \times 3 \;  \delta_{\rm sep}
- \;  \delta_{\rm sep}\\
&=& 380 \; \lambda_1 - 40 \; \delta_{\rm nonsep}-100 \; \delta_{\rm sep}.
\end{eqnarray*}
The expression agrees with the formula of Scott Mullane (see~\cite{Mul}). 

\section{Relations in the Picard group of the strata}\label{sec:picard}

We fix the notation for all the section. Let $g, n, m\geq 0$ such that $2g-2+n+m>0$. Let $Z=(k_1,\ldots,k_n)$ and $P=(p_1,\ldots,p_n)$ be vectors of positive integers such that $|Z|-|P|=2g-2$. In this section we consider the space $\oM_g(Z-P)\subset \oM_{g,n+m}$ (see Section~\ref{ssec:results} for definitions). The purpose  is to define several natural classes in $\pic(\oM_g(Z-P))\otimes{\Q}$ and to compute relations between these elements. Namely there are two types of classes which arise naturally:
\begin{itemize}
\item Divisors associated to admissible graphs (see Sections~\ref{ssec:boundary} and~\ref{ssec:div});
\item Intersections of $\oM_g(Z-P)$ with the tautological classes of $A_1(\oM_{g,n})$.
\end{itemize}

\subsection{Classes defined by admissible graphs}

We consider the moduli space of stable differentials $\oH_{g,n,P}$ and the locus $\overline{A}_{g,Z,P}\subset \oH_{g,n,P}$. We recall that $p:\oH_{g,n,P}\to \oM_{g,n+m}$ is the forgetful map. We have seen that $\overline{A}_{g,Z,P}$ admits a stratification  indexed by admissible graphs (see Lemma~\ref{boundaries}). Here, we will describe the set of admissible graphs $(\Gamma,I,l)$ such that $p(\overline{A}_{\Gamma,I,l})$ is a divisor in $\oM_g(Z-P)=p(\overline{A}_{g,Z,P})$.

The map $p:\P A_{g,Z,P}\to \mathcal{M}_g(Z-P)$ is an isomorphism (see Lemma~\ref{lem:isoline}). Thus, if $p(\P\overline{A}_{\Gamma,I,l})$ is  a divisor in $\oM_g(Z-P)$ then $\P\overline{A}_{\Gamma,I,l}$ is a divisor in $\P\overline{A}_{g,Z,P}$. We saw that an admissible graph $(\Gamma,I,l)$ defines to a divisor of  $\overline{A}_{g,Z,P}$ if and only it is of one of the three following types (see Section~\ref{ssec:div}):
\begin{enumerate}
\item the admissible graph of depth $0$ with one vertex and one edge;
\item an admissible graph of depth $0$ with two vertices and one edge;
\item a bi-colored graph that satisfies the condition $(\star\star)$.
\end{enumerate}
\begin{mypr}~\label{pr:listPic}
Let $(\Gamma,I,l)$ be an admissible graph.  The locus $p(\P\overline{A}_{\Gamma,I,l})$ is a divisor of $\oM_g(Z-P)$ if and only if:
\begin{itemize}
\item or $(\Gamma,I,l)$ is of the type 1 above ;
\item or $(\Gamma,I,l)$ is a bi-colored graph with one vertex of level $-1$, one stable vertex of level 0 and possibly other semi-stable vertices of level 0.
\end{itemize}
\end{mypr}
We call {\em irreducible divisor} the divisor of $\oM_g(Z-P)$ of the first type. We denote this divisor by $D_0$ (with the reduced structure). 

 In the second case, the stabilization of the graph $\Gamma$ determines a unique stable twisted graph of depth $1$, $(\Gamma',I')$ (we no longer write the level structure which is uniquely determined by $I$). Conversely, a twisted stable graph of depth 1 with two vertices, we can uniquely determine an admissible graph satisfying the condition of Proposition~\ref{pr:listPic} by putting all the poles on the component of level $-1$ on unstable rational components of level 0 (see Lemma~\ref{lem:corgraph} and Example~\ref{ex:contraction} below). 
 
\begin{mydef} 
A {\em simple bi-colored graph} is a  twisted stable graphs of depth 1 with two vertices.  We denote by ${\rm SB}(Z,P)$ the set of simple bi-colored graphs. If $(\Gamma,I)$ is a simple bi-colored graph, we denote by $D_{\Gamma,I}$ the corresponding divisor in $\oM_g(Z-P)$ (with the reduced structure) and by $a_{\Gamma,I}$ its class in $\pic(\oM_g(Z-P))\otimes{\Q}$.
\end{mydef}

The class $i_*a_{\Gamma_I}$ (where $i$ is the closed immersion of $p(\overline{A}_\Gamma,I)$ in $\oM_{g,n+m}$) in the moduli space of curves is simply given by:
$$
\zeta_{\Gamma*}\left([\oM_{g_0}(Z_0-P_0)], [\oM_{g_1}(Z_1-P_1)]\right),
$$
where $g_0$ and $g_1$ are the genera of the vertices of level $0$ and $-1$ and the vectors $Z_0$, $P_0$, $Z_1$ and $P_1$ are the vectors encoding the orders of zeros and poles at the marked points and half-edes induced by $Z, P$ and the twist $I$.

\begin{example}\label{ex:contraction} We illustrate this correspondence between simple bi-colored graphs and boundary divisors. We consider $g=3$, $Z=(2,6)$ and $P=(-2,-2)$ and the admissible graph
\begin{figure}
$$\xymatrix@C=0.5em{
-2 \ar@{-}[d]& -2 \ar@{-}[d] & +2 \ar@{-}[ld] &&&&&&& -2 \ar@{-}[d] & +2 \ar@{-}[ld]\\ 
*+[Fo]{0}\ar@{-}[rd]& *+[Fo]{1} \ar@{=}[d]  & +6 \ar@{-}[ld]  &&\ar[rrr]&&&\ar[lll]&-2 \ar@{-}[rd] & *+[Fo]{1}\ar@{=}[d]& +6 \ar@{-}[ld]    \\ 
&*+[Fo]{1}&&&&&&&& *+[Fo]{1}.
}$$
\caption{Example of the correspondance between admissible and stable graphs.}\label{example2}
\end{figure}
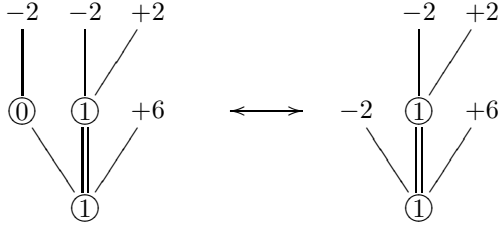
(on this example we take the twists equal to $1$ on all edges). On this example, the class $i_*a_{\Gamma,I}$ in the moduli space of curves will be given by 
$$\zeta_{\Gamma*}\left([\oM_{1}(+2,+0,+0,-2)], [\oM_{1}(+6,-2,-2,-2)]\right).$$
\end{example}

\begin{proof}[Proof of Proposition~\ref{pr:listPic}]
Let $(\Gamma,I)$ be an amissible graph of depth at most 1 with several stable components of level 0. Then the fiber of $p$ over a generic point of $p(A_{g,Z,P})$ is of dimension greater than one.
That is why divisors of type 2 are not mapped to divisors while the map $p$ restricted to $D_0$ is indeed of degree one onto its image.

Now we consider an admissible graph of depth $1$ with one stable vertex of level 0. Then the graph satisfies condition $(\star\star)$ if and only it has one vertex of level $-1$. 

Finally, we consider an admissible graph $(\Gamma, I, l)$ of depth 1 and   with no stable vertex of level 0. The projectivized stratum $\P A_{\Gamma, I, l}\subset \P\oH_{g,n,P}$ is empty. Indeed, $Z$ is complete for $g$ and $P$ thus the differential on each unstable component with a marked pole of order $p$ is given by $dz/z^p$. Therefore $A_{\Gamma, I,l}$ is a substack of the zero section of the cone $\oH_{g,n,P}\to  \oM_{g,n+m}$ (see Section~\ref{ssec:properties} for the description of the zero section). 
\end{proof}

\subsection{Classes defined by residue conditions}

We recall that $\oR$ is the vector space of residues, i.e. the subspace of $\C^{m}$ defined by $\{(r_1,\ldots,r_m)/ r_1+\ldots+r_m=0\}$.  Let $R\subset \oR$ be vector subspace of codimension $1$.  We define the following class in the rational Picard group of $\oM_g(Z-P)$: 
$$
\delta^{\rm res}_{R}=p_*(\P \overline{A}_{g,Z,P}^{R}).
$$

\begin{mynot} Let $1\leq i<j\leq n+m$. We denote ${\rm SB}(Z,P)_i$ (respectively ${\rm SB}(Z,P)^i$) the set of simple bi-colored graphs such that the leg corresponding to $i$ is adjacent to the vertex of level $-1$ (respectively to the vertex of level $0$). We denote ${\rm SB}(Z,P)^j_i={\rm SB}(Z,P)_i\cap {\rm SB}(Z,P)^j$.

Let $R\subset \oR$ is a vector subspace. For a simple bi-colored graph, we denote by  $R^0\subset \oR$ be the vector space defined by the linear conditions $\{r_i= 0 \}$ for all $1\leq i\leq m$ such that the leg of index $n+i$ is at level $-1$.  We denote by ${\rm SB}(Z,P)_R$ the set simple bi-colored graphs such that the space $R$ contains $R^0$. 
\end{mynot}

\subsection{Classes defined by intersection}

Let $\beta$ be a tautological class in $\pic(\oM_{g,n+m})\otimes{\Q}$. The class $\beta$ determines a class in $\pic(\oM_g(Z-P))\otimes{\Q}$ by taking $i^*\beta$ where $i$ is the closed immersion of of  $\oM_g(Z-P)$ into $\oM_{g,n+m}$.  If $\beta$ is either $\lambda_1, \kappa_1$ or a $\psi$-class then we will denote by the same letter its pull-back to $\pic(\oM_g(Z-P))\otimes{\Q}$ if the context is clear.

The last class  that we will consider is the push-forward of the $\xi$-class that we denote:
$$
\overline{\xi}=p_*(\xi\cdot [\P\overline{A}_{g,Z,P}]).
$$

\begin{myth}\label{th:rel}
The following relations holds in ${\rm Pic}(\oM_g(Z-P))\otimes \Q$:
\begin{enumerate}
\item for all $1\leq i\leq n$:
$$
\overline{\xi} + (k_i+1)\psi_1=\!\! \sum_{(\Gamma,I)\in {\rm SB}(Z,P)_i}\!\!\!\!  m(I) a_{\Gamma,I} ;
$$
\item for all $1\leq i,j\leq n$:
$$
(k_i+1)\psi_i - (k_j+1)\psi_j=\!\!\!\!\!\!\!\! \sum_{(\Gamma,I)\in {\rm SB}  (Z,P)_i^j} \!\!\!\!\!\!\! m(I) a_{\Gamma,I}   - \!\!\!\!\!\!\sum_{(\Gamma,I)\in {\rm SB}(Z,P)_j^i} \!\!\!\!\!\!\!  m(I) a_{\Gamma,I};
$$
\item  for all $R\subset \oR$ vector subspace of codimension $1$:
$$\overline{\xi}=\delta^{\rm res}_{R} \; + \!\! \sum_{(\Gamma,I)\in {\rm SB}(Z,P)_R} \!\!\!\!  m(I) a_{\Gamma,I}  ;$$
\item if $m=0$ then 
$$
 \lambda_1+ \kappa_Z  \overline{\xi}=\frac{1}{12}\; \delta + \!\!\!\!\!   \sum_{(\Gamma,I)\in {\rm SB}(Z)}\!\!\! \!\!\!\!  2\overline{m}(I,\Gamma) a_{\Gamma,I} ,
$$
where $\delta$ is the boundary divisor of $\oM_{g,n}$,
\begin{eqnarray*}
\kappa_Z&=&\frac{1}{12}  \sum_{i=1}^n \frac{k_i (k_i+2)}{k_i+1} \\
\text{and   }\; \overline{m}(I,\Gamma)&=& \frac{m(I)}{12}\left( -m(I)+ \sum_{i\mapsto v^1} \frac{k_i (k_i+2)}{k_i+1} \right).
\end{eqnarray*}
 the second sums goes over all legs adjacent to the vertex of level $-1$.
\end{enumerate}
\end{myth}

\subsubsection{Relations~(1) and (2) and Double Ramification cycles}

The second relation of Theorem~\ref{th:rel} is a direct consequence of the first one: we write $(k_i+1)\psi_i - (k_j+1)\psi_j= (\overline{\xi}+ (k_i+1)\psi_i)-(\overline{\xi}+ (k_j+1)\psi_j)$. However, we chose  to write Relation~(2) in this form for two reasons:
\begin{itemize}
\item first because it involves only classes  defined directly in the moduli space of curves;
\item the second motivation is related to the Conjectures A and B.  Indeed the classes ${\rm H}_g(Z)$ and $[\oM_g(Z-P)]$   (see Section~\ref{ssec:appli} for definitions) are supposed to be generalizations of Double Ramification cycles. In~\cite{BurShaSpiZvo}, the authors proved several identities between intersection of $\psi$-classes with Double Ramification cycles. One consequence of the relations proven in~\cite{BurShaSpiZvo} is the existence a universal $\psi$-class  over the Double Ramification Cycles (independent of the choice of a marked point). For strata of differentials the following corollary gives a candidate for this universal $\psi$-class.
\begin{mycor}
The following class in ${\rm Pic}(\oM_g(Z-P))\otimes \Q$
$$
(k_i+1)\; \psi_i - \!\!\! \!\!\!\!\!\sum_{(\Gamma, I) \in {\rm SB}(Z,P)_i} \!\!\!\!\!\!\!\!\! m(I) a_{\Gamma,I} 
$$
is independent of the choice of $1\leq i \leq n$. 
\end{mycor}
\end{itemize} 

\begin{proof}[Proof of Relation~(1).] It  is a direct consequence of the induction formula (see Theorem~\ref{ind}). We consider $Z_i$, the vector obtained from $Z$ by increasing the $i$-th entry by $1$ and $R=\oR$ (no residue condition), then we get: 
$$
 (\xi+(k_{i}+1)\psi_{i}) \cdot [\P\overline{A}_{g,Z,P}] = [\P\overline{A}_{g,Z_{j},P}]  \;  + \!\!\!\!\!\!\!\! \sum_{(\Gamma,I) \in \bic(g,Z,P)_{i}} \!\!\!\!\!\!\!\! m(I) \; a_{\Gamma,I}.
$$
We remark that $|Z_j|-|P|>2g-2$ thus $ [\P\overline{A}_{g,Z_{j},P}]=0$. Now we apply the push forward by $p$ to this expression. In the sum of the right-hand side only the simple bi-colored graphs will contribute and we indeed get
$$\overline{\xi} + (k_i+1)\psi_1=\!\! \sum_{(\Gamma,I)\in {\rm SB}(Z,P)_i}\!\!\!\!  m(I) a_{\Gamma,I}.
$$
\end{proof}

\subsubsection{Relation~(3)}

To prove the third relation, we need a generalization of the induction formula. Let $R\subset \oR$ be a vector subspace of co-dimension 1. We recall that an admissible bi-colored graph defines a space of residue conditions $R^0\subset \oR$ (see Section~\ref{ssec:boundary} for the construction of $R^0)$. We define $\bic(g,Z,P)_R\subset \bic(g,Z,P,\oR)$ as the subset of bi-colored graphs such that $R^0\subset R$.
\begin{mypr}\label{indres} The following equality holds in $H^*(\P\oH_{g,n,P};\Q)$
$$
[\P \overline{A}_{g,Z,P}^R]=\xi [\P \overline{A}_{g,Z,P}]- \!\!\!\!\! \!\!\!\!\! \sum_{(\Gamma,I)\in \bic(g,Z,P)_R} \!\!\!\!\! \!\!\!\!\! m(I)a_{\Gamma,I}.
$$
\end{mypr}

\begin{remark}
We could have stated this proposition in a larger generality (unstable disconnected base) but it will not be useful here. 
\end{remark}

\begin{proof}
The proof is the same as the proof of Theorem~\ref{ind}. We consider the line bundle $\O(1)\simeq \O(-1)^\vee$ restricted to $\P \overline{A}_{g,Z,P}$ with its section 
\begin{eqnarray*}
s:\O(-1)&\to & \C\\
\alpha &\mapsto& \oR/R
\end{eqnarray*} 
defined as the composition of the residue map $\O(-1)\to \oR$ and the projection $\oR\to \oR/R$. The vanishing locus of the section $s$ is the union of  $\P \overline{A}^R_{g,Z,P}$ and of the divisors $\P\overline{A}_{\Gamma,I}$ for all $(\Gamma,I)\in \bic(g,Z,P)_R$. 

Now the vanishing order of $s$ along  $\P \overline{A}^R_{g,Z,P}$ is $1$ because the residue map is a submersion. The vanishing order of $s$ along $\P\overline{A}_{\Gamma,I}$ is $1$ because Lemma~\ref{tech} remains valid if we replace  the section $s_{i,j}$ by the section $s$ and the set of graphs $\D(\mathbf{g},\mathbf{Z},\mathbf{P},R)_{ij}$ by the set of graphs $\D(g,Z,P)_R$. 
\end{proof}  

\begin{proof}[Proof of Relation~(3).]  Relation (3) is a direct consequence of Proposition~\ref{indres}. It suffices to use apply the push-forward by the forgetful map  $p$. 

\end{proof}

\subsubsection{Relation~(4) and the work of Eskin-Kontsevich-Zorich}~\label{ssec:KonZor} Let $g\geq 2$ and let $Z=(k_1,\ldots, k_n)$ be a partition of $2g-2$. Before proving Relation~(4), let us mention that Konstevich proved that 
$$
\lambda_1 = - \kappa_Z \overline{\xi} + \gamma
$$
where $\gamma$ is a class supported on the boundary of $\oM(Z)$. From this relation, he deduced an equation relating two numerical invariants of strata of differentials: the sum of the Lyapunov exponents and the Siegel-Veech constants (the complete proof of this equation was achieved in~\cite{EskKonZor}).  Relation~(4) gives an explicit expression for the class $\gamma$.

\begin{proof}[Proof of Relation~(4).]
Let $Z'$ be the vector equal to $(k_1,\ldots,k_n,0)$. If $\pi: \oM_{g,n+1}\to \oM_{g,n}$ is the forgetful map of the last marked point, then we have $\overline{A}_{g,Z'}=\pi^{-1}(\overline{A}_{g,Z})$. We use the induction formula to obtain the relation:
$$
(\xi+\psi_{n+1})[\P A_{g,Z'}]= 0 + \!\!\!\!\!\! \sum_{\bic(g,Z)_{n+1}} \!\!\!\!\!\! m(I) a_{\gamma,I}
$$ 
We multiply this formula by $\psi_{n+1}$ to get
\begin{equation}\label{eqn:psixi}
\xi\psi_{n+1} [\P A_{g,Z'}]+ \psi_{n+1}^2 [\P A_{g,Z'}] = \!\!\!\!\!\! \sum_{\bic(g,Z')_{n+1}} \!\!\!\!\!\! m(I) \psi_{n+1} a_{\Gamma,I}.
\end{equation}
Now we apply $(p_*)\circ(\pi_*)$ to this formula (we forget the last point and then the differential). We study each term separately.

\subsubsection*{Contribution of $\xi\psi_{n+1} [\P A_{g,Z'}]$} The classes $\xi$ and $[\P A_{g,Z'}]$ are pull back by $\pi$ thus
\begin{eqnarray*}
p_*\left(\pi_*(\psi_{n+1} \xi [\P A_{g,Z'}])\right)&=& p_*\left(\pi_*(\psi_{n+1}) \xi [\P A_{g,Z}]\right)\\
&=& \kappa_0  p_*(\xi [\P A_{g,Z}])\\
&=& (2g-2+n) \overline{\xi}
\end{eqnarray*}
by the projection formula.
\subsubsection*{Contribution of $\psi_{n+1}^2 [\P A_{g,Z'}]$} Still by the projection formula we have:
\begin{eqnarray*}
p_*\left(\pi_*(\psi^2_{n+1}  [\P A_{g,Z'}])\right)&=& p_*\left(\pi_*(\psi^2_{n+1})  [\P A_{g,Z}]\right)\\
&=& \kappa_1\\
&=&12\lambda_1 - \delta + \sum_{i=1}^{n} \psi_i.
\end{eqnarray*}
Now we use the first relation to write:
$$
\sum_{i=1}^{n} \psi_i= -\left(\sum_{i=1}^n \frac{1}{k_i+1}\right) \overline{\xi} +\sum_{i=1}^{n} \left(\sum_{(\Gamma,I)\in {\rm BS}(g,Z)_i} \frac{m(I)}{k_i+1} a_{\Gamma,I}\right).
$$

\subsubsection*{Contribution of $\sum_{\bic(g,Z')_{n+1}}m(I) \psi_{n+1} a_{\Gamma,I}$}

Let $(\Gamma,I)$ be a bi-colored graph in  $\bic(g,Z')_{n+1}$. There are two possible configurations:
\begin{itemize}
\item the point $n+1$ belongs to a rational components with 3 special points. In which case $\psi_{n+1}a_{\Gamma,I}=0$;
\item the point $n+1$ is carried by a general vertex of level $-1$ which is not contracted after the forgetful map. 
\end{itemize}
In the second case, we denote by $(\Gamma',I')$ the twisted  graph obtained after forgetting the marked point. We get:
$$
\pi_*(\psi_{n+1}a_{\Gamma,I})= (2g_{\Gamma',I',1}-2+n_{\Gamma',I',1}) a_{\Gamma',I'},
$$
where $g_{\Gamma,1}$ and $n_{\Gamma,1}$ denote the genus and valency of the vertex of level $-1$. Thus
\begin{eqnarray*}
(p_*\circ \pi_*)\sum_{\bic(g,Z')_{n+1}}m(I) \psi_{n+1} a_{\Gamma,I} &=& \sum_{(\Gamma,I)\in {\rm BS}(g,Z)}  m(I)(2g_{\Gamma',I',1}-2+n_{\Gamma',I',1}) a_{\Gamma,I}.
\end{eqnarray*}
We obtain Relation (4) by replacing all the terms in Equation~(\ref{eqn:psixi}) by their expressions in terms of simple bi-colored graphs. 
\end{proof}


\end{document}